\documentclass{article}
\usepackage{authblk}

\usepackage{amsthm, amsmath, amsfonts, amssymb}
\usepackage{mathtools}
\usepackage{mathrsfs}
\usepackage{enumitem}
\usepackage{hyperref}
\hypersetup{
    colorlinks,
    linkcolor={blue!30!black},
    citecolor={blue!50!black},
    urlcolor={blue!80!black}
}

\usepackage{natbib}
\usepackage{color}
\usepackage{xcolor}
\usepackage[margin=3cm]{geometry}
\usepackage{algorithm}
\usepackage{algorithmic}
\usepackage[normalem]{ulem}
\usepackage{comment}
\usepackage{xparse}
\usepackage{graphicx}
\usepackage{subcaption}

\numberwithin{equation}{section}

\newtheorem{theorem}{Theorem}[section]
\newtheorem{lemma}{Lemma}[section]
\newtheorem{proposition}{Proposition}[section]
\newtheorem{corollary}{Corollary}[section]
\newtheorem{assumption}{Assumption}[section]

\newtheorem{definition}{Definition}[section]
\newtheorem{example}{Example}[section]
\newtheorem{remark}{Remark}[section]

\newcommand{\N}{\mathbb{N}}

\newcommand{\R}{\mathbb{R}}
\newcommand{\I}{\mathbb{I}}
\renewcommand{\P}{\mathbb{P}}
\newcommand{\E}{\mathbb{E}}
\NewDocumentCommand{\borel}{O{\Set}}{\mathcal{B} (#1)}  

\newcommand{\as}[1]{}
\newcommand{\cf}[1]{}

\definecolor{orange}{RGB}{255,165,0}
\definecolor{purple}{RGB}{128,0,128}

\usepackage{xspace}

\newcommand\ie{\emph{i.e.}\xspace}
\newcommand\iid{\ensuremath{\mathit{i.i.d.}}\xspace }

\newcommand{\nset}{\mathbb{N}}
\newcommand{\rset}{\mathbb{R}}

\newcommand{\ind}{\mathbf{1}}
\newcommand{\un}{\ind}
\newcommand{\point}{\,\cdot\,}
\newcommand{\eqd}{\overset{d}{=}}
\usepackage{xifthen}
\newcommand{\PP}{\P}

\newcommand{\sphere}{\mathbb{S}}
\newcommand{\ball}{\mathbb{B}}

\newcommand{\AHS}[1]{H_{0, {#1}} }
\newcommand{\tto}{\xrightarrow[t\to\infty]{}}
\newcommand{\innerprod}[2]{\left\langle #1, #2 \right\rangle}

\newcommand{\pDsymbol}{\mathrm{pD}}  

\newcommand{\XpDsymbol}{\mathrm{XpD}}  

\NewDocumentCommand{\pD}{ e{_} g g }{%
    \IfNoValueTF{#2}{%
        \IfNoValueTF{#3}{%
            \ensuremath{\pDsymbol\IfNoValueF{#1}{_{#1}}}%
        }{%
            \ensuremath{\pDsymbol\IfNoValueF{#1}{_{#1}}\left(\,;\,{#3}\right)}%
        }%
    }{%
        \IfNoValueTF{#3}{%
            \ensuremath{\pDsymbol\IfNoValueF{#1}{_{#1}}\left({#2}\right)}%
        }{%
            \ensuremath{\pDsymbol\IfNoValueF{#1}{_{#1}}\left({#2}\;
            ;\;{#3}\right)}%
        }%
    }%
}

\NewDocumentCommand{\hatpD}{ e{_} g g }{%
    \IfNoValueTF{#2}{%
        \IfNoValueTF{#3}{%
            \ensuremath{\widehat{\pDsymbol}\IfNoValueF{#1}{_{#1}}}%
        }{%
            \ensuremath{\widehat{\pDsymbol}\IfNoValueF{#1}{_{#1}}\left(\,;\,{#3}\right)}%
        }%
    }{%
        \IfNoValueTF{#3}{%
            \ensuremath{\widehat{\pDsymbol}\IfNoValueF{#1}{_{#1}}\left({#2}\right)}%
        }{%
            \ensuremath{\widehat{\pDsymbol}\IfNoValueF{#1}{_{#1}}\left({#2}\;;\;{#3}\right)}%
        }%
    }%
}

\newcommand{\XpD}{\ensuremath{\XpDsymbol}}

\NewDocumentCommand{\hatXpD}{ e{_} g g }{%
    \IfNoValueTF{#2}{%
        \IfNoValueTF{#3}{%
            \ensuremath{\widehat{\XpDsymbol}\IfNoValueF{#1}{_{#1}}}%
        }{%
            \ensuremath{\widehat{\XpDsymbol}\IfNoValueF{#1}{_{#1}}\left(\,;\,{#3}\right)}%
        }%
    }{%
        \IfNoValueTF{#3}{%
            \ensuremath{\widehat{\XpDsymbol}\IfNoValueF{#1}{_{#1}}\left({#2}\right)}%
        }{%
            \ensuremath{\widehat{\XpDsymbol}\IfNoValueF{#1}{_{#1}}\left({#2}\,;\,{#3}\right)}%
        }%
    }%
}

\NewDocumentCommand{\tildepD}{ e{_} g g }{%
    \IfNoValueTF{#2}{%
        \IfNoValueTF{#3}{%
            \ensuremath{\widetilde{\pDsymbol}\IfNoValueF{#1}{_{#1}}}%
        }{%
            \ensuremath{\widetilde{\pDsymbol}\IfNoValueF{#1}{_{#1}}\left(\,;\,{#3}\right)}%
        }%
    }{%
        \IfNoValueTF{#3}{%
            \ensuremath{\widetilde{\pDsymbol}\IfNoValueF{#1}{_{#1}}\left({#2}\right)}%
        }{%
            \ensuremath{\widetilde{\pDsymbol}\IfNoValueF{#1}{_{#1}}\left({#2}\,;\,{#3}\right)}%
        }%
    }%
}

\newcommand{\coD}[2]{\pD{#1}{#2}}
\newcommand{\hatcoD}[2]{\hatpD{#1}{#2}}

\newcommand{\coDs}[2]{\pD_s{#1}{#2}}
\newcommand{\hatcoDs}[1]{\widehat{\XpDsymbol}_s\left({#1}\right)}

\newcommand{\mus}{\mu}

\newcommand{\aD}[2]{{\mathrm{aD}}\left({#1}\,;\,{#2}\right)}

\newcommand{\hataD}[2]{\widehat{\mathrm{aD}}\left({#1}\,;\,{#2}\right)}

\newcommand{\Phitrad}{\Phi_{\mathrm{s}}} 
\newcommand{\hatPhitrad}{\widehat{\Phi}_{\mathrm{s},k}} 

\newcommand{\Phisprob}{\Phi_{s,1}} 

\newcommand{\PropRotationalInvarianceCODText}{$\textbf{D'}_1$}  %

\newcommand{\PropMonotonocityCODText}{$\textbf{D}_3$}  

\newcommand{\PropVanishingInfCODText}{$\textbf{D}_4$}  

\newcommand{\PropUpperSemContCODText}{$\textbf{D}_5$}  

\newcommand{\PropContinuityAsFuncMeasureCODText}{$\textbf{D}_6$}  %

\title{Polar Depth for Potentially Heavy-Tailed Data}
\author[1]{Stephan Clémençon}
\author[1]{Carlos Fern\'andes}
\author[1]{Pavlo Mozharovskyi}
\author[2]{Anne Sabourin}
\affil[1]{LTCI, Télécom Paris, Institut Polytechnique de Paris, France}
\affil[2]{Université Paris Cité, Université Paris Saclay, ENS Paris Saclay, CNRS, SSA, INSERM, Centre Borelli, F-75006, Paris, France
}

\begin{document}

\maketitle

\begin{abstract}
  Motivated by the analysis of the behaviour of extremes from
multivariate heavy-tailed distributions, we introduce a novel notion
of statistical depth, referred to as \textit{polar depth}. The
polar depth function is naturally expressed in polar
coordinates, as is the limiting distribution of a  regularly varying random variable,  beyond asymptotically large thresholds, once its marginals have
been appropriately normalized. Not only does the polar depth function
make it easy to order the extreme values taken by a heavy-tailed
random variable $X$ and finds natural applications in anomaly
detection, but it is also possible to show, as we prove it under
appropriate assumptions in this article, that the polar
depth of the largest observations, \textit{i.e.}, observations $X$
such that $\vert\vert X \vert\vert>t$, converges to the
polar depth of the limiting distribution as $t\to \infty$.
Although designed to quantify the depth of multivariate extremes, the polar depth is interesting in its own right, insofar as this
notion is more relevant for distributions whose support is included
in a halfspace than the alternatives proposed in the literature, the
halfspace depth in particular. Here, we demonstrate its properties
and analyze statistical issues related to its estimation from both finite-sample and asymptotic points of view. We present numerical results to
empirically demonstrate its relevance, particularly for the
statistical analysis of extreme observations and more specifically
for the identification of anomalies among them.

\end{abstract}
\tableofcontents
\section{Introduction}
Since the seminal work of \cite{Tukey75}, the notion of
statistical depth has received increasing attention in the
statistical literature and is at the origin of many tools for
multivariate data analysis. Considering a probability distribution
$P$ on $\mathbb{R}^d$ with $d>1$, a statistical data depth function
$D(.;\; P):\mathbb{R}^d\rightarrow \mathbb{R}_+$ transfers the
natural order on the real line to $\mathbb{R}^d$, thus providing an
order from the center outwards of the points in the support of $P$.
The tools of (signed) rank or order statistics, so useful for
analyzing univariate data, can then be immediately extended to
multivariate observations. Applications in multivariate statistics
are too numerous to list exhaustively here (see \textit{e.g.}
\cite{Mosler13} and~\cite{mosler2022choosing}), but we can mention
hypothesis testing~\cite{DyckerhoffLP15}, robust
inference~\cite{LiuPS99},
classification~\cite{LiCAL12,MozharovskyiML15}, imputation of missing data~\cite{MozharovskyiJH20} or anomaly
detection~\cite{MozharovskyiV25,SatermanAMHSGC23,VallaMFAB25},
among others. Many alternative definitions to the original concept
introduced in \cite{Tukey75}, the \textit{halfspace} depth, have been
proposed in the literature. Examples include the simplicial
depth~\cite{Liu}, the projection depth~\cite{Liu92}, the majority
depth~\cite{LiuSingh}, the Oja depth~\cite{OJA1983}, the
zonoid~\cite{koshevoy1997}, the spatial depth~(\cite{Chaudhuri,
Vardi1423}), the Monge-Kantorovich depth~\cite{chernozhukov2017} or
the IRW depth \cite{NagyGijbels2016} and its affine invariant version
\cite{10.1214/23-EJS2189}. The axiomatic nomenclature proposed in
\cite{ZuoS00a} (see also~\cite{mosler2022choosing}, for a comparison)
enables a systematic comparison of their advantages
and drawbacks. It lists four key properties that should be (ideally)
possessed by a `proper' depth function (see Section
\ref{sec:background} for a detailed formulation): affine
invariance, maximality at center, vanishing at infinity and
monotonicity relative to deepest point.  In addition to checking
these properties, the advantages and disadvantages of any statistical
depth are also examined in terms of the possible existence of
computational algorithms in the case of empirical distributions
\cite{dyckerhoff2020approximate} and of theoretical guarantees for the latter \cite{nagyuniformrates}.

\par A central challenge in statistics and machine learning is the
capacity of methods and algorithms to extrapolate beyond the range of
observed data, or to perform out-of-domain generalization. This
concern also arises in depth estimation. Extreme value theory (EVT)
provides a rigorous framework for such extrapolation, offering strong
theoretical guarantees under weak tail homogeneity assumptions.
Incorporating EVT principles into the analysis of tail properties is
a natural approach, as explored in
\cite{einmahl2015bridging,he2017estimation}, where the focus lies on
the halfspace depth under the assumption of joint regular variation.
However, two main limitations persist.
First, the geometric structure of heavy-tailed multivariate
extremes---specifically, the product-form limit distribution
involving radial and angular components, as described in
Section~\ref{sec:background}---is not reflected in the family of sets
(halfspaces) used to solve the minimization problem. For asymmetric
point clouds extending far in only one halfspace (\textit{e.g.}, when a coordinate at least is strictly positive, as in many real-world applications if we consider the analysis of quantities such as rainfall, pressure, or volatility for instance), the estimated depth
becomes unstable for test points that are extreme both in magnitude
and direction, namely when their norm is large,  while their
direction lies near the boundary of the `angular' support. Further
discussion and illustrative examples are deferred to
Section~\ref{sec:experiments}, which presents numerical experiments.
More generally, the classical requirement of shift-invariance for
depth functions has been recognized as potentially inappropriate in
certain contexts, including directional data
\cite{NagyDemniButtarazziPorzio2023,nagy2024theoretical}.
Shift-invariance, a property satisfied by the halfspace depth, is
undesirable when the origin holds a special status and should be
considered the most central point, as is often the case in
applications. In particular, if one of the components is almost
surely positive, the asymmetric structure of the point cloud suggests
that the origin should not be treated as a `shallow' point, despite
its halfspace depth being zero.
Second, and more critically, the assumption of multivariate regular
variation imposes uniform tail behavior across all components, which
is rarely satisfied outside the study of log-returns in Finance. In
environmental applications, for instance, different components may
exhibit distinct tail indices, with some being short-tailed. While
relaxing this assumption through preliminary marginal standardization
is intuitive, the authors do not provide a clear path forward, and
the technical challenges posed by such standardization leave the
possibility of theoretical guarantees in this setting as an open
question---one we do not address here.

In this paper, we introduce a new notion of depth, the \textit{polar
depth}, defined on the entire input space $\mathbb{R}^d$. This notion
is naturally aligned with the limit structure of multivariate
extremes and permits out-of-domain generalization with guarantees
valid in the classical framework of multivariate extremes. The polar
depth of a test point $x \in \mathbb{R}^d$ is defined as the minimum
probability mass contained in \emph{polar regions} of the form
$$
\{y \in \mathbb{R}^d : \|y\| \geq t, \langle y, u \rangle \geq 0\},
\quad t \geq 0, \, u \in \mathbb{R}^d,
$$
that contain $x$, where $\|\point\|$ denotes the Euclidean norm.
These regions are intersections of directional halfspaces orthogonal
to the vector $u$ ---whose boundaries pass through the origin---and
complements of balls centered at zero with radius $t$.
We derive non-asymptotic statistical guarantees in the form of
uniform error bounds in probability, both in the bulk and in the
tails, under suitable joint regular variation assumptions. These
theoretical guarantees also accomodate  marginal standardization to
unit Pareto margins.
Our analysis builds on recent computational and theoretical advances
in the \emph{angular halfspace depth}, a depth notion tailored to
directional data
\citep{NagyLD21,NagyDemniButtarazziPorzio2023,nagy2024theoretical},
as well as non-asymptotic guarantees for the empirical angular
measure \citep{ClemenconJalalzaiSabourinSegers2023}.
\medskip

Our contributions are as follows. After having briefly recalled the
elements of theory relating to multivariate regular variation and
those concerning the concept of statistical depth in Section
\ref{sec:background}, we rigorously define the polar depth, a
novel notion of depth particularly suited to data supported on a
halfspace going through the origin. Therein, we establish its
properties and study the problem of its estimation from both an
asymptotic and non-asymptotic point of view in Section
\ref{sec:Polar}. Section \ref{sec:Extremes} aims to show that this
notion of statistical depth applied to the limit law of a
heavy-tailed random vector is quite relevant from a statistical point
of view insofar as it can be estimated in a natural way in a
pre-asymptotic regime by means of a fraction of the largest
observations. Asymptotic and non-asymptotic results are proved for
this purpose, and the difference with the approach developed in
\citep{einmahl2015bridging, he2017estimation}, an alternative to
defining the depth of extremes, is also discussed in detail.
Comparable results are next developed in the marginally standardized
framework in Section~\ref{sec:standardize}.
Numerical experiments are presented in Section \ref{sec:experiments},
with the empirical results obtained confirming the relevance of the
theoretical concepts and analyses developed previously. Finally, in
Section \ref{sec:conclusion}, some concluding remarks are collected
and several avenues for further research are outlined. Proofs are deferred to  the Appendix.

\section{Background and Preliminaries}\label{sec:background}
For the sake of clarity, basic notions pertaining to statistical
depth theory are briefly recalled, together with the key concepts of
multivariate EVT that shall be used in the subsequent analysis. Here
and throughout, we denote by $\langle \cdot, \cdot \rangle$ and
$\| \cdot \|$ the usual Euclidean inner product and
norm on $\mathbb{R}^d$, and by $\sphere=\{ z\in
\mathbb{R}^d:\; \| z \|= 1 \}$ the unit sphere of
$\mathbb{R}^d$ w.r.t. the Euclidean norm.
The open unit ball for $\|\point\|$ centered at the origin is
denoted by $\ball$. Xe denote indifferently by $\ball_t$ or $t\ball$ the open ball centered at $0$ with radius $t\ge 0$. 
Pseudo-polar coordinates will be of particular importance:  define the \emph{radius} $r(x)$ and  \emph{direction} (or \emph{angle}) $\theta(x)$ of a vector $x\in\rset^d\setminus\{0\}$ as 
\begin{align*}
    r(x) & = \|x\|,  \qquad  \theta(x) =  r(x)^{-1} x.
\end{align*}
By $\delta_x$ is meant the
Dirac mass at any point $x$, by $\mathbf{1}\{\mathcal{E}\}$ the
indicator function of any event $\mathcal{E}$ and by $f_\# Q$ the push forward measure of a measure $Q$ by a measurable map $f$, namely $f_\#Q(\point) = Q(f^{-1}(\point))$. The boundary of a set $A\subset \rset^d$ is denoted by $\partial A$.  We use the shorthand notation r.v. to denote random vectors or, when clear from context, random variables.   For any r.v. $Z$
(\emph{resp.} random pair $(Z_1,Z_2)$) the notation  $\mathcal{L}(Z)$
(\emph{resp.} $\mathcal{L}(Z_1\,|\, Z_2)$) stands for the distribution
of $Z$ (\emph{resp.} the conditional distribution of $Z_1$ given $Z_2$).
The affine halfspace with normal vector $u\in\rset^d$ and offset
$r\in\rset$ is defined by
$H_{r,u} = \{x \in \rset^d: \langle x,u \rangle \ge  r \}.
$
We consider throughout this article a random vector $X = (X_1,\ldots, X_d)$  valued in $\rset^d$, defined on a probability space $(\Omega,\mathcal{A}, \PP)$,  with joint distribution $P= X_\# \mathbb{P}$ on $\rset^d$. 
For simplicity, we assume that $X$ has  continuous marginal cumulative distribution functions  $u\mapsto F_j(u) = \PP{(X_j \le u)}$. 
Continuity of margins is not essential but avoids technicality for rank-based standardization.  Here and below,
$q(\tau)$ means the left-quantile at level $\tau\in (0,1)$ of the
radial variable $r(X)$ and we set \(U(t)=q(1-t^{-1})\) for $t>0$.

\subsection{Statistical Depth}\label{subsec:depth_basics}

Order and (signed) rank statistics are valuable tools for univariate
statistical analysis and are used to perform a wide range of tasks,
from robust statistical inference to effective statistical hypothesis
testing.
Since $\mathbb{R}^d$ lacks a natural order for $d \geq 2$, the
concept of statistical depth arises to extend univariate ordering
notions to a multivariate setting.
The halfspace depth introduced by ~\cite{Tukey75} is the first
proposal. The (univariate) halfspace depth w.r.t. a probability
distribution $P$ on $\mathbb{R}$ is defined by $ \inf \{ P(]-\infty,
t]), P([t,+\infty [)\}$ for all $x\in \mathbb{R}$.
Considering next a multivariate r.v. $X$ with probability
distribution $P$ on $\mathbb{R}^d$ with $d>1$, its halfspace depth at
$x\in \mathbb{R}^d$ is the infimum of the probability mass taken over
all possible closed halfspaces containing $x$:

\begin{equation}\label{multivariate}
    D_{\mathrm{H}}(x;\; P) = \underset{u\in \sphere}{\inf}
    \mathbb{P} \left( \langle u,X\rangle  \geq  \langle u,x\rangle
    \right)= \underset{u\in \sphere}{\inf} P \left(
    {H}_{\langle x,\; u\rangle,\; u}\right).
\end{equation}
The properties of the halfspace depth \eqref{multivariate} have been
extensively studied in the literature. As proved in
\cite{StruyfR99,Koshevoy02}, it fully characterizes empirical and
finitely discrete probability distributions. Asymptotic analysis of
its statistical counterpart (consistency and asymptotic normality),
constructed from an i.i.d. sample $X_1,\; \ldots,\; X_n$ drawn from
$P$ in a plug-in fashion by replacing $P$ in \eqref{multivariate}
with the raw empirical distribution
${P}_n=(1/n)\sum_{i=1}^n\delta_{X_i}$,  has been carried out in
\cite{Donoho82, RousseeuwS98,ZuoS00a} and the concentration properties of the
empirical halfspace depth and contours have been established in
\cite{BurrF17} and \cite{Brunel19}, respectively. The continuity
properties of $D_H(x;\; P)$, in both \(x\) and \(P\), as well as the
robustness of multivariate location estimators based on $D_H$, have
been studied in \cite{DonohoG92}, while computational issues are
considered in, \textit{e.g.}, \cite{LiuZ14a, dyckerhoff2016exact, LiuMM18}.
In spite of these properties, many alternatives $D$ to the halfspace
depth have been proposed as mentioned in the Introduction section,
see \cite{mosler2022choosing}. A list of four key properties, recalled below
for the sake of clarity, that a depth function must satisfy in order
to guarantee its “center-outward-ordering” interpretation, has been
proposed in \cite{ZuoS00a}.

\begin{itemize}
    \item[$\mathbf{D}_1$] {\sc (Affine invariance)} Denote by $P_X$
        the distribution of any r.v. $X$ valued in $\mathbb{R}^d$.
        Equipped with this notation, we have
        \begin{equation*}\label{eq:AI}
            \forall x\in \mathbb{R}^d,\;\;  D(Ax+b;\;  P_{AX+b})=D(x;\;  P_X),
        \end{equation*}
        for all $d$-dimensional r.v. $X$, $d\times d$ nonsingular
        matrix $A$ with real entries and vector $b$ in $\mathbb{R}^d$.
    \item[$\mathbf{D}_2$] {\sc (Maximality at center)} For any
        probability distribution $P$ on $\mathbb{R}^d$, the depth
        function $D(\cdot;\; P)$ takes its maximum value at $P$'s
        symmetry center $x_P$ if $P$ has one (in a sense to be
            specified, \textit{e.g.}, such that the equality in
        distribution $X-x_P \overset{d}{=} x_P - X$ holds):
        \begin{equation*}
            D(x_P;\;  P)=\sup_{x\in \mathbb{R}^d}D(x;\; P).
        \end{equation*}
    \item[$\mathbf{D}_3$] {\sc (Monotonicity relative to deepest
        point)} Let  $P$ be a probability distribution on
        $\mathbb{R}^d$ and $x_P$ be its deepest point. The depth
        function decreases on any ray that begins at $x_P$:
        \begin{equation*}\label{eq:MoDP}
            \forall x\in \mathbb{R}^d,\; \forall \alpha \geq 0,\;\;
            D(x_P;\; P)\geq D(x_P+\alpha(x-x_P);\; P).
        \end{equation*}

    \item[$\mathbf{D}_4$] {\sc (Vanishing at infinity)} For any
        probability distribution $P$ on $\mathbb{R}^d$, the depth
        function $D$ vanishes at infinity:
        \begin{equation*}\label{eq:Van}
            D(x;\; P)\rightarrow 0, \text{ as } \vert\vert
            x\vert\vert \rightarrow \infty.
        \end{equation*}
\end{itemize}

One may refer to \cite{Dyckerhoff04} and \cite{Mosler13} for a
different formulation of the properties above. Observe also that,
depending on the definition chosen, a depth function $D$ is not
necessarily applicable to the whole ensemble of all probability
distributions on $\mathbb{R}^d$, but only to a subset of it, see
\cite{mosler2022choosing}. For instance, the AI-IRW depth introduced
in \cite{10.1214/23-EJS2189} is defined for square integrable random
vectors with invertible covariance matrix only. The notion of
polar depth we introduce in Section \ref{sec:Polar} is
relevant for r.v's $X$ whose distribution is supported on a
halfspace passing through the origin, \textit{e.g.}, r.v.'s $X$ with
at least one component that is a.s. positive. The concept of depth
has also been recently adapted in a number of ways to deal with
distributions whose support is included in a specific space, non
necessarily linear. One may refer to, \textit{e.g.},
\cite{NagyGijbels2016}, \cite{StaermanMozharovskyiClemencon2020},
\cite{virta2023spatial}, \cite{LafayeDeMicheauxMV20} or \cite{FernandezClemencon2025} for
proposals tailored to different general metric or path spaces. As we
explained in the Introduction, we are particularly interested here in
the case of \textit{angular} random variables, \textit{i.e.}, random
variables taking their values on the sphere $\sphere$, and a notion of
statistical depth applicable to them that can be viewed as a direct
extension of \eqref{multivariate}.
\medskip

\noindent {\bf Angular halfspace depth.}
An appealing notion of depth for angular r.v.'s which enjoys many
desirable properties of a 'depth function' is the \textit{angular
    halfspace
depth}~\citep{NagyDemniButtarazziPorzio2023,nagy2024theoretical}. Let
$W$ be a random variable valued in the unit sphere $\sphere$
with distribution $Q$. Its \emph{angular depth} at any angle $w\in
\sphere$ is defined as
\begin{equation}
    \label{eq:angulardepth}
    \aD{w}{Q} = \inf_{\|u\| = 1, x \in \AHS{u}} Q( \sphere\cap \AHS{u}).
\end{equation}
It is based on the trace of the collection of \emph{angular
halfspaces} $\AHS{u} = \{x\in\rset^d~:~ \langle x, u\rangle \ge 0\}$, $u\in \sphere$, which are closed halfspaces whose
boundary passes through the  origin, on the unit sphere
$\sphere$. Of course, the properties $D_1-D_4$ above are no
longer relevant for depths applicable to distributions whose support
is included in $\sphere$. For instance,
$\mathbf{D}_1$ should naturally be replaced with the rotational
invariance property recalled below, which is naturally satisfied by
\eqref{eq:angulardepth}.

\begin{itemize}
    \item[$\mathbf{D'}_1$] {\sc (Rotational invariance)} Denote by
        $Q_W$ the distribution of any r.v. $W$ valued in the unit
        sphere $\sphere$. Equipped with this notation, we have
        \begin{equation*}\label{eq:RI}
            \forall w\in \sphere,\;\;  D(Ow;\;  Q_{OW})=D(w;\;  Q_W),
        \end{equation*}
        for all random vector  $W$ valued in $\sphere$, $d\times d$
        orthogonal matrix $O$ with real entries.
\end{itemize}
Conditions analogous to $\mathbf{D}_2$, $\mathbf{D}_3$ and
$\mathbf{D}_4$ can also be formulated in this context (although the
    notion of symmetry center is not easy to define for distributions
with support included in $\sphere$ in particular). One may refer
to \cite{NagyDemniButtarazziPorzio2023} for a description of these
conditions and a review of alternative depth concepts for
angular/directional data.

A key property of the angular halfspace depth is that it is necessarily constant over an open hemisphere: for any distribution $Q$ on $\sphere$, there exists  $u_0\in\sphere $ such that $\aD{\point}{Q}$ is constant over the set $\{w \in \sphere: \langle w, u_0\rangle <0 \}$ \citep[see][Theorem 7, for a  stronger version]{nagy2024theoretical}. Furthermore, this constant value is  the minimum depth value over the sphere. Consequently, the angular halfspace depth is most effective for angular distributions concentrated on a single hemisphere, making constancy over the complementary hemisphere a desirable feature. As will be discussed in Subsection~\ref{subsec:polar_def}, this property extends to the polar depth.

\subsection{Multivariate Extremes and Regular Variation}\label{subsec:basics_RV} 

We review here key concepts from multivariate extreme value theory and present the classical assumptions—some stronger than others—used in the related works and in Sections~\ref{sec:Extremes} and~\ref{sec:standardize}. While this background is not strictly necessary to grasp the definition and main properties of the polar depth in Section~\ref{sec:Polar}, the principles of multivariate EVT underpin its motivation. 
For more details, one may refer to, \textit{e.g.},
\cite{resnick2008extreme, beirlant2006statistics, de2007extreme}.

We recall that a function a function $\ell: \rset_+\to\rset_+$ is said to be slowly varying if for all $s>0$,
$\ell(st)/\ell(t) \to 1$ as $t\to \infty$.
A function 
$h:\rset_+\to\rset_+$ is regularly varying with exponent $\alpha\in\rset$ iff  there exists a slowly varying function $\ell$ such that $h(t) = t^{\alpha }\ell(t), t>0$. This is equivalent to the condition that, for any $x>0$, $h(tx)/h(t)\to x^{\alpha}$ as $t\to\infty$. 

\paragraph{Multivariate Regular Variation}
A standard assumption, although sometimes debatable in practice, 
formulated in multivariate EVT is that the random vector $X$ under study is regularly varying, which means that there exists some function $b(t)>0$ and a limit Borel measure $\nu$ defined on $\rset\setminus\{0\}$,
writing $t A = \{ t x : x \in A
\}$ for $t > 0$ and $A \subseteq \mathbb{R}^d$, such that
\begin{equation}\label{eq:rv_weak_with_b}
  b(t)\PP(X\in tA) \tto \nu(A)
  \end{equation}
  for all Borel sets $A$ such that $0\notin\partial A$ and $\nu(\partial A) = 0$. This is a straightforward extension of regular variation of univariate r.v.'s \citep{Bingham1987}, and it is equivalent to $M_0$ convergence of the measures $b(t) P(t\point)$ \citep{hult2006regular}.  Necessarily, the function $b(t)$ is regularly varying with index $\alpha$ for some $\alpha>0$  referred to as the \textit{regular variation index}. 
  
  Multivariate regular variation can also be
described in terms of vague convergence of Radon measures on
the compactified, punctured orthant $[0, \infty]^d \setminus \{(0,
\ldots, 0)\}$ \citep{resnick2007heavy}.
If~\eqref{eq:rv_weak_with_b} holds true, we also have the convergence
\begin{equation}\label{eq:rvWeak}
  p_t^{-1}\PP(X\in tA) \tto \nu_1(A),
  \end{equation}
  where $p_t = \PP(r(X)>t)$ and $\nu_1(\point) = c \nu(\point)$ with $c= \nu(\ball^c)^{-1}$ is a probability measure when restricted to $\ball^c$.
  If~\eqref{eq:rv_weak_with_b} and~\eqref{eq:rvWeak} hold true, the limit measures $\nu$ and $\nu_1$ are necessarily homogeneous of order $- \alpha$: for all $t>0$ and $A\subseteq \mathbb{R}^d$,
\begin{equation}\label{homogeneity-nu}
    \nu(tA) = t^{-\alpha} \nu(A). 
\end{equation} 
Then, all marginal survival functions  $S_j = 1 - F_j, j\le d$, as well as the function $t\mapsto p_t$,  are necessarily regularly varying with the same regular variation index $\alpha$.

 Some theoretical analyses of EVT-based estimators require a refinement (strengthening) of~(\ref{eq:rv_weak_with_b}), namely that the slowly varying function involved in $p_t$ (or $b(t)$) converges  to  a non-zero constant. 
Notably, the framework developed in \cite{einmahl2015bridging} and
\cite{he2017estimation}, interfacing
multivariate extremes with halfspace depth,  requires the  assumption that 
\begin{equation}
    \label{eq:rvStrong}
    t^\alpha \PP(X \in tA) \tto \nu(A), 
\end{equation}
for all Borel sets $A$ such that $0\notin\partial A$ and $\nu(\partial A) = 0$.

If~\eqref{eq:rvWeak} holds true, then
one may define an \emph{angular measure} $\Phi$ relative to $\nu_1$ on the sphere $\sphere$ equipped with the trace Borel $\sigma$-field, through
\begin{equation}\label{eq:def_angular_measure}
    \Phi(C) = \nu_1\left(\{tw, w\in C, t \ge 1 \}\right) 
\end{equation}
for any Borel set $C\subseteq \sphere$.
Observe that, with our normalization choice, $\Phi$ is a probability measure on $\sphere$. The transformation $(r,\theta)$ to polar coordinates allows the limiting distribution to be expressed as a tensor product. Given the homogeneity property satisfied by $\nu_1$, the conditional distribution of $(r(X)/t,\theta(X))$ given $r(X)\geq t$ converges, as $t\to \infty$, to the law of a pair $(R_\infty, \Theta_\infty)$ of independent r.v.'s, valued in $[1,\infty)\times \mathbb{S}$ with margins $R_\infty\sim\mathrm{Pareto}(\alpha)$ and $\Theta_\infty\sim \Phi$:
\begin{equation}
    \lim_{t\to +\infty}{\P\left(r(X)/t\geq r,\; \theta(X)\in C| r(X)\geq t\right)} = r^{-\alpha}\Phi(C) =\PP(R_{\infty}\geq r) \PP(\Theta_\infty \in C),
    \label{eq:limit_form_angular}
\end{equation}
for all $r\geq 1$ and Borel set $B\subseteq \sphere$ s.t. $\Phi(\partial C)=0$. The asymptotic dependence structure of the $X_j$'s as $r(X)$ tends to $\infty$ is thus entirely described by the angular measure $\Phi$, which indicates in which directions the extremes are likely to occur.


\paragraph{Marginal Standardization}


A common assumption for handling different tail behaviors
and separating the dependence structure from the marginal distributions is to
consider a marginal standardization involving integral transformations of marginal distributions, and to assume only regular variation of the transformed vector.
Precisely, transform each margin $X_j$ of $X$ into a unit Pareto variable $V_j = v_j(X_j)$ using $v_j(y) = (1 - F_j(y))^{-1}$ for $j\le d$. Set  $v(x) = (v_1(x_1)\ldots, v_j(x_j))$ and $V = (V_1,\ldots, V_d)$. In multivariate EVT one often uses the fact that if $X$ belongs to a multivariate maximal attraction domain, with margins exhibiting possibly different tail behaviors, then the vector $V$ satisfies \eqref{eq:rvStrong} with $\alpha=1$, that is to say:
\begin{equation}
    \label{eq:standardRV}
    \lim_{t \to \infty} t\PP\left(t^{-1} V \in B\right) = \mu(B)
\end{equation}
for a measure $\mu$ on $\mathbb{R}^d\setminus \{0\}$ and all Borel sets $B$ of $\rset^d$ that are bounded away from the origin and such
that $\mu(\partial B) = 0$. The measure $\mu$ is referred to as the \emph{exponent measure}
for reasons that originate in the theory of max-stable distributions,  refer to , \emph{e.g.}, \cite{resnick2008extreme}, Chapter 5, see also \cite{de2007extreme}, Chapter 6, and Definition 6.1.7. 
As the measure $\nu$, the exponent measure $\mu$ is homogeneous, but with a regular variation $\alpha$ index necessarily equal to $1$ since the margins $V_j$ are unit Pareto. We have $\mu(t\point) = t^{-1}\mu(\point)$ for all $t>0$ and the margins of $\mu$ are standardized: 
\begin{equation*}
    \forall y \in (0, \infty),\;  \forall j \in \{1, \ldots, d\},
    \qquad \mu(\{x \in E : x_j \ge y \}) = y^{-1}.
\end{equation*}
The angular measure $\Phitrad$  on the standard Pareto scale is defined similarly to $\Phi$,  as 
\begin{equation}
    \label{eq:angularMeasure}
    \Phitrad(A) = \mu\{tw, w\in A, t \ge 1 \}. 
  \end{equation}
  The measure $\Phitrad$ has finite mass $\Phitrad(\sphere)\in[1,d]$ but is generally not a probability measure. The statistical properties of empirical counterparts of $\Phitrad$ have been studied in the bi-variate case and an asymptotic framework in \cite{einmahl2001nonparametric, einmahl2009maximum}, and in arbitrary dimension and non-asymptotic framework in \cite{ClemenconJalalzaiSabourinSegers2023}.

\subsection{Statistical Depth and Multivariate EVT - Related Works}\label{subsec:SoA}

As mentioned in the Introduction, the previous works
\cite{einmahl2015bridging,he2017estimation} aimed to define the
concept of depth of extreme observations by focusing on the tail
properties of the
halfspace depth \eqref{multivariate}.
The main finding
regarding the structural
asymptotic properties of $D_H(x\mid P)$ as $\vert\vert x\vert\vert
\to \infty$ is that  the halfspace depth of points $x$ relative to
the supposedly regularly varying distribution $P$ with limit measure $\nu$ in
the tail regions converges to the depth of their rescaled versions,
relative to the limit measure $\nu$. Specifically (cf.
\cite[Proposition 2]{he2017estimation}), if the strong condition
\eqref{eq:rvStrong} holds
true, together with minor continuity assumptions, then the halfspace
depth converges uniformly on regions bounded away from the origin:
\begin{equation}\label{eq:limit}
    \sup_{x: \|x\|>\varepsilon}|t^{\alpha}D_H(tx;\; P) - D_H(x;\;
    \nu)|\tto 0,
\end{equation}
for all $\varepsilon>0$.   Observe that homogeneity of $\nu$ implies
that of $D_H(\cdot ;\; \nu)$ with the same index $\alpha>0$,
\ie, $D_H(tx;\; \nu) = t^{-\alpha} D_H(x;\; \nu)$, $t>0$.
A  natural idea behind both
\cite{einmahl2015bridging,he2017estimation} is to construct estimates
of the limit depth  $\widehat D_H(\cdot \mid  \nu)$ based on an
i.i.d. sample $X_1,\; \ldots,\; X_n$ drawn from $P$ by leveraging the
homogeneity property of $\nu$ and the approximation \eqref{eq:limit}:
a 'plug-in' estimator of the approximant of $D_H(x;\; \nu)$ for
large $\|x\|$ is obtained by combining an estimator of the tail index
$\alpha$ with an empirical version of $\nu$ based on the $k$
observations with largest norm.

The approach we propose in the following section for quantifying the
statistical depth of extreme data is quite different and addresses
certain limitations of the definition outlined above. First, it
should be noted that only weak guarantees (weak consistency of
    quantile regions based on the estimator of $D_H(\cdot ;\; \nu)$
    under appropriate conditions for $k=k_n$ in particular, see Theorem 1
in \cite{he2017estimation}) under the restrictive condition
\eqref{eq:rvStrong} are established for the latter. More importantly,
it does not cover the case, which is very common in practice, where
the coordinates of the heavy-tailed random vector $X$ do not all have
the same tail index (non standard multivariate regular variation).
Finally, in many situations where multivariate extreme value analysis
is applied, the random vector $X$ under study takes its values in a
halfspace.  In many applications, where the data corresponds to
physical measurements, it is very common for at least one of its
components (\textit{e.g.}, thermodynamic temperature, pressure) to
take only positive values, for example. In EVT especially, the
positive orthant has a particular status, as one is often interested
in joint tail events associated with large, positive values (when
    examining both positive and negative extreme values, one usually
    “doubles” the number of features by analyzing their negative and
positive parts, see \textit{e.g.} \cite{chiapino2019vizu}). In these
situations, the notion of halfspace depth \eqref{multivariate}, on
which the approach promoted in
\cite{einmahl2015bridging,he2017estimation} are based, is not
relevant anymore. First because the 'center' (\textit{i.e.} the mode)
of the mass distribution is often near the origin, which is located
at the boundary of the distribution’s support in this case), has then
small halfspace depth, see
Fig. \ref{fig:comparison_fundamental_sets}. Second, the halfspace
depth is not suitable for the tensor product form of the law of
extremes, involving independent radial and angular components: here
too, the extreme points occurring in the most probable directions may
be assigned a halfspace depth close to zero.

\begin{figure}
\centering
\includegraphics[width=0.5\linewidth]{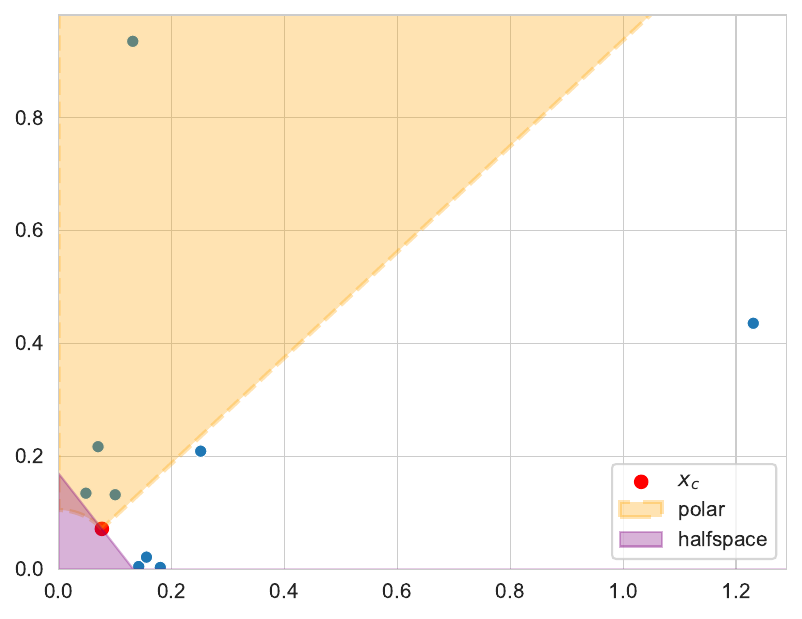}
\caption{Illustrative sample of 10 realizations of the heavy tailed r.v.
    \(Y=(R\cos(\Theta),R\sin(\Theta))\), where
    \(R\)'s distribution is a univariate Pareto with \(\alpha=1\)
    and \(\Theta\)
    is uniform in \([0,\pi/2]\), independent from $R$. The point \(x_c\) (in
    \textcolor{red}{red})
    is the point
    with the smallest radius, and, therefore, the most central point in
    the sample. In \textcolor{purple}{purple}, a halfspace that
    contains \(x_c\) but no other
    point in the sample is depicted, and in \textcolor{orange}{orange}, the region
used to compute the polar depth of \(x_c\), see Definition~\ref{def:coD}.}
\label{fig:comparison_fundamental_sets}
\end{figure}

As shall be seen, the approach we develop in the following section
addresses all of these issues. The definition we propose in Section~\ref{sec:Polar} possibly distinguishes between radial depth and
angular depth, which makes it possible to assign a significant depth
to points located near the boundary of the (empirical) distribution's
support even if the latter is included in a halfspace. Non
asymptotic guarantees for the empirical version can be established
under the weaker condition \eqref{eq:rv_weak_with_b}, see subsection~\ref{subsec:bounds} and our approach can be extended to the non
standard multivariate regular variation case, as shown in Section~\ref{sec:standardize}.

\section{The \textsc{Polar Depth} Function}\label{sec:Polar}

In this section, we introduce the polar depth, based 
on the transformation to pseudo-polar coordinates as its name suggests, and
on the notion of angular halfspace depth recalled in Subsection~\ref{subsec:depth_basics}. Its main properties are established, and
the theoretical and computational issues involved in its
statistical estimation are also addressed. While this notion is of
interest in its own right, we shall see in the next section that it
is the key to quantifying and estimating the depth of the limit law
of a multivariate heavy-tailed random variable.

\subsection{Definition and Main Properties}\label{subsec:polar_def}

Here and below,  for any positive measure $\pi$ on $\rset^d$,   the notation $\pi_r$  stands for the push forward measures of $\pi$ by the mappings $r$, while  $\pi_\theta$ is the pushforward of the restriction of $\pi$ to $\rset^d\setminus\{0\}$, by the mapping $\theta$, namely
$\pi_\theta(\point) = \pi\{x\neq 0: \theta(x) \in \point\}$. More generally   for $t\ge 0$,  and for any positive Borel measure $\pi$  on $\rset^d$, we write 
$$
\pi_{\theta,t}(\point) = \pi\{x \neq 0: \theta(x)\in \point, r(x)\ge t\}. 
$$
so that by definition $\pi_\theta= \pi_{\theta,0}$. 
For $t\ge 0$ such that
$0 < \pi_r([t,\infty)\setminus\{0\})< \infty$, we denote by  $\pi_{\theta|t}$ the distribution of the angular component of a random variable with distribution proportional to $\pi_{\theta,t}$, namely 
    \begin{equation}\label{eq:def_Pr_Ptheta_r}
    \pi_{\theta|t}(\point) = 
    \frac{\pi_{\theta,t}(\point)}{\pi_{r}\big([t,\infty)\setminus\{0\}\big)} =
    \frac{\pi\{x\neq 0: \theta(x)\in \point, r(x)\ge t \}}{\pi\{x\neq 0: r(x)\ge t \}}. 
\end{equation}
  We sometimes refer to $\pi_{\theta|t}$ as a  \emph{conditional angular distribution}. Note that the set condition `$x\neq 0$' in the above definition is weaker than    the condition  `$r(x)\ge t$' for $t>0$,  not for $t=0$.

While allowing for a (potentially infinite) positive  measure \(\pi\) may seem
unnecessary in the current section, it will prove particularly useful
in subsequent Sections~\ref{sec:Extremes} and~\ref{sec:standardize} when discussing the tail properties of the 
polar depth. This approach facilitates a connection with multivariate EVT.
In this context we shall consider limit measures that arise in the context of regular variation. These measures are traditionally defined on $\rset^d\setminus\{0\}$, not $\rset^d$. The  notations,  definitions, and results  in the present section extend immediately  by identifying such a limit measure, say $\nu$, with its extension  $\tilde \nu$ to the full space  $\rset^d$, defined as $\tilde \nu(A) = \tilde\nu(A\setminus\{0\})$, for any Borel set $A\subset \rset^d$. This amounts to assigning the origin a null mass. 
   

We now define the  notion of depth that is central to our work.
\begin{definition}[Polar depth]\label{def:coD}
    Let $\pi$ be a Borel measure on $\rset^d$ such that for $t>0$,
    $\pi(\ball_t^c)<\infty$ (this is the case for any probability
    distribution on $\rset^d$ and fo limit measures in multivariate EVT). We define the \textbf{polar depth}
    of any point $x\in\rset^d\setminus\{0\}$ relative to $\pi$ as
    \begin{equation}
        \label{eq:codepth}
            \coD{x}{\pi}  = \inf_{u \in\sphere,\; x \in \AHS{u} }\pi(
            \ball_{\|x\|}^c\cap \AHS{u} )    
             = \inf_{u\in\sphere , \; x \in \AHS{u} }\pi(\{y:\;  \|y\|  \ge
            \|x\|, \;y \in \AHS{u}\}).
          \end{equation}
          For $x=0$ by convention we let
          $$
\coD{0}{\pi} = \inf_{u \in\sphere }\pi(\AHS{u}\setminus\{0\} ) 
$$

\end{definition}

Recall the definition of the angular depth in~\eqref{eq:angulardepth}.
 Wih  the notations introduced at the begining of this section,  we have
$\coD{0}{\pi} =  \pi(\rset^d\setminus\{0\} ) ~\inf_{w\in\sphere}\aD{w}{\pi_{\theta|0}}$.

\begin{remark}[Homogeneity of basic sets]\label{rem:homog-basic-set}
    For any $t>0$, $x\in\rset^d\setminus \{0\}$ and $u\in \rset^d$ such that  $\|u\| = 1$, we have
    \begin{equation}\label{eq:homogeneity_halfspaces}
        x \in \AHS{u}\iff tx \in \AHS{u},
    \end{equation}
     as well as
     $$ \ball_{\|tx\|}^c\cap \AHS{u}  = t \Big( \ball_{\|x\|}^c\cap \AHS{u}\Big). $$
\end{remark}

\begin{remark}[Special case: independent radial and angular
    components]\label{rem:special-independent}
    When $P$  is  a probability measure defined on $\rset^d\setminus\{0\}$  as a product measure in polar  coordinates,
    namely when $P\{x\neq 0: r(x) \in A, \theta(x)\in B \} = P_r(A)
    P_\theta(B)$ for any Borel sets $A,B$ in $(0,\infty)$ and $\sphere$ respectively, where both $P_r$ and $P_\theta$ are probability measures, 
    then  Definition~\ref{def:coD} simplifies as follows. 
    We have for $x\neq 0$, $P\{y:\;  \|y\| \ge\|x\|, y \in \AHS{u}\} = P_r
    [r(x),\infty)\,P_\theta(\AHS{u}\cap \sphere)$, hence the connection
    between the polar depth and the angular depth in this situation:
    \begin{equation}\label{eq:center-out-angular-independentCase}
       \forall x\neq 0,\;\;      \coD{x}{P} =  P_r[r(x),\infty)\,\aD{\theta(x)}{P_\theta}.
    \end{equation}
    This special case will be of particular interest for the tail
    properties of polar depth investigated in
    Section~\ref{sec:Extremes}, as limit distributions of extremes
    satisfy this independence property under minimal assumptions.
  \end{remark}

The following lemma extends the simple observation made in
Remark~\ref{rem:special-independent}. 
\begin{lemma}[Connection between polar and angular depths]\label{lem:connect-aD-coD}
  For any positive  measure $\pi$ on $\rset^d$ such that $\pi(\ball_t^c)<\infty$ for any $t>0$, and $x\neq 0$ such that
  $\pi(\ball_{r(x)}^c) >0$, it holds that
    \begin{equation}
        \label{eq:equiv-coD-aD}
    \coD{x}{\pi} = \pi_r[r(x),\infty)\, \aD{\theta(x)}{\pi_{\theta|r(x)}}.
  \end{equation}
\end{lemma}

Lemma~\ref{lem:connect-aD-coD} shows that for fixed $t>0$, the polar depth of points on the sphere of radius $t$ is proportional to their angular depth relative to  the conditional angular distribution $\pi_{\theta|t}$. As discussed in Subsection~\ref{subsec:depth_basics}, the latter depth is necessarily constant over a hemisphere. As a consequence, just like the angular depth, the polar depth  is most relevant in
the case where at least one component of $X$ is either nonnegative,
or nonpositive, so that the distribution of $X$ is concentrated on a halfspace. This property  is shared in common by many
geophysical data (\textit{e.g.}, pressure, precipitation, water flow,~\dots).

Our continuity results below concerning the polar depth relative to a probability distribution $P$ require additional assumptions. Although not all of them are systematically necessary, we state them together here. 

\begin{assumption}[Radial continuity]\label{as:continuousConditionalAngular}
  The probability distribution $P$ has continuous radial distribution $P_r$. 
 
\end{assumption}

The following mild assumption regarding continuity on hyperplanes is required in \cite{nagy2024theoretical} to ensure continuity properties of the angular depth. It comes at no surprise, considering Lemma~\ref{lem:connect-aD-coD}, that it is also a convenient condition to ensure continuity of the polar depth. 
\begin{assumption}[Smoothness]\label{as:smoothnessHalfSpaces} The probability distribution \(P\)
    places no mass on the frontiers
    of the angular halfspaces, that is, \(P(\partial H_{0,u})=0\) for
    all \(u\in\sphere\).
\end{assumption}



We now list the desirable `depth properties' satisfied by $\pD$.
Obviously, the standard shift invariance property is not satisfied,
nor should it be in a context where the origin plays a special role,
as mentioned in the Introduction. However, many other standard
properties of depth hold true, see \cite{ZuoS00a,ZuoS00b}, as shown by the result below.

\begin{proposition}[Properties of the polar depth]\label{prop:main_properties}
    Let \(P\) be a probability distribution on \(\rset^d\) such that $P(\rset^d\setminus\{0\})>0$.
    \begin{itemize}
        \item[\PropRotationalInvarianceCODText]\label{prop:PropRotationalInvarianceCOD}
            {\sc (Rotational invariance)} For
            any orthogonal matrix \(O\in\rset^{d\times d}\) and all $x\neq 0$ in $\mathbb{R}^d$, we have: \(\coD{x}{P}=\coD{Ox}{O_\#P}\). 
        \item[\PropMonotonocityCODText]\label{prop:PropMonotonocityCOD}
            {\sc (Monotonicity along rays that pass through 0)} For
            all \(t\geq 1\) and $x\neq 0$, it holds that \(\coD{tx}{P}\leq \coD{x}{P}\),
            while
            \(\coD{x}{P}\geq\coD{tx}{P}\) when \(0<t\leq 1\). In addition, 
            under Assumption~\ref{as:smoothnessHalfSpaces}, we have
            that for any sequence $(t_n)$ of nonnegative numbers such
            that $t_n\to 0$,
            \(\lim_{n\to\infty}\coD{t_nx}{\pi}=P_r(0,\infty)\aD{\theta(x)}{\pi_{\theta|0}}\).
        \item[\PropVanishingInfCODText]\label{prop:VanishingInfCOD} {\sc
            (Vanishing at infinity)} We have \(\coD{x}{P}\to 0\)
            as \(r(x)\to +\infty\).
        \item[\PropUpperSemContCODText]\label{prop:PropUpperSemContCOD}
            {\sc (Upper semicontinuity)} Suppose that $P$ satisfies
            Assumption~\ref{as:continuousConditionalAngular}, then \(\coD{\point}{P}\) is upper semicontinuous, i.e.
            if \(x_n\to x\neq 0\), then
            \begin{equation*}
                \limsup_n{\coD{x_n}{P}}\leq \coD{x}{P}.
            \end{equation*}
            If additionally $P$ satisfies Assumption~\ref{as:smoothnessHalfSpaces} and 
            \(P_r\) is continuous,
            then \(\coD{x}{P}\) is a continuous function of \(x\).
        \item[\PropContinuityAsFuncMeasureCODText]\label{prop:ContinuityAsFuncMeasureCOD}
            {\sc (Continuity in measure)} Suppose that $P$ places no mass on the boundary of affine halfspaces and that $P_r$ is a continuous probability distribution. Then, for any sequence of
           probability 
           measures \((\pi_n)_{n\in \mathbb{N}}\) that weakly converges to \(P\) as $n\to \infty$, 
     we have
            \begin{equation}
                \sup_{x\neq 0  } |\coD{x}{P} -
                \coD{x}{\pi_n}|\to 0 \text{ as } n\to \infty.
            \end{equation}
    \end{itemize}
\end{proposition}

\begin{remark}
    \label{rem:regularity_measures}
    For simplicity and concreteness we focus on probability measures in Proposition~\ref{prop:main_properties}. Similar results could be obtained with positive measures on $\rset^d\setminus\{0\}$ equipped with the topology of $M_0$ convergence \citep{hult2006regular}, which is appropriate in the context of regular variation. It turns out that our developments in Sections~\ref{sec:Extremes} and ~\ref{sec:standardize} do not require such a formalism, so that we leave this topic to the interested reader. 
\end{remark}

Before examining the extent to which it is possible to statistically recover the polar depth of a random vector $X$ from a finite number of independent copies of it with (non-asymptotic) guarantees, we give two examples of probability distributions $P$ for which $\coD{x}{P}$ can be expressed simply, also highlighting how natural this notion of depth is.

\begin{example}[Normal
    distribution]\label{ex:pD_for_normal_distribution} Suppose that $X$ is a $d$-dimensional centred Normal r.v. with covariance matrix
    \(\Sigma=\sigma^2 I_d\), where $\sigma>0$ and $I_d$ is the identity matrix of the Euclidean space \(\R^d\). Denote by \(P\) its Gaussian distribution. In this case, the radial and
    angular components
    are independent \cite[pp. 30]{Wang1990}, therefore, from
    \eqref{eq:center-out-angular-independentCase}, we obtain that
    \(\coD{x}{P}=P_r([r(x),+\infty))\aD{\theta(x)}{P_{\theta}}\). The
    spherical symmetry of \(P\)
    implies that its angular part is uniformly distributed on
    \(\sphere\), hence, \(\aD{\theta(x)}{P_{\theta}}=1/2\) for all
    \(x\neq 0,n, ;\). 
    Finally, for the radial component we have that
    \(P_r=\sigma\sqrt{\chi^2_d}\) where \(\chi^2_d\) is a chi-squared
    distribution
    with \(d\) degrees of freedom \cite[pp. 32]{Wang1990}.
    Putting all together, we get:
    \begin{equation}\label{eq:example:pD_centred_normal_distribution}
        \coD{x}{P}=\frac{1}{2}\P\left( r(X)\geq r(x) \right) =
        \frac{1}{2}\overline{F}_{\chi^2_d}\left(
        r(x)^2/\sigma^2 \right) \text{ for all } x\neq 0.
    \end{equation}
    For \(d=2\), this implies that \(pD(x,P)=\pi f(x)\) where \(f\) denotes
    the density of \(P\).
\end{example}

\begin{example}[Cauchy distribution] When \(P\) is a multivariate
    spherical Cauchy distribution
    with density
    \begin{equation}
        f(x) = \frac{\Gamma\!\left(\frac{d+1}{2}\right)}{\pi^{\frac{d+1}{2}}}
        \left( 1 + r(x)^{2} \right)^{-\frac{d+1}{2}} \text{ for }
     x \in \mathbb{R}^d,
    \end{equation}
    an argument similar to that in Example
    \ref{ex:pD_for_normal_distribution}, but using that
    \(P_r=\sqrt{dF_{d,1}}\) (where \(F_{d,1}\) is the Fisher's F
        distribution with \(d\)
    and \(1\) degrees of freedom \cite[pp. 7]{KotzNadarajah2004}), shows that
    \begin{equation}
        \coD{x}{P}=\frac{1}{2}\overline{F}_{\textnormal{F}_{d,1}}\left(
        r(x)^2/d \right) \text{ for } x \in \mathbb{R}^d\setminus\{0\}.
    \end{equation}
\end{example}

\subsection{Finite-Sample Bounds for the Empirical Polar Depth}\label{subsec:bounds}
A natural statistical counterpart of the polar depth $\coD{.}{P}$ of
a probability distribution $P$ on $\rset^d$ based on $n\geq 1$ independent
observations $X_1,\; \ldots,\; X_n \sim P$ is the
\textit{plug-in} version 
\begin{equation}\label{eq:plugin}
\widehat{\pDsymbol}(x)\overset{def}{=}\coD{.}{P_n}
\end{equation}
that is obtained by replacing
$P$ with the raw empirical distribution
$P_n:=(1/n)\sum_{i=1}^n\delta_{X_i}$. From a theoretical viewpoint, due to
the low complexity of the family of halfspaces in moderate
dimension, strong (non-asymptotic) statistical guarantees can be
proven for this estimator, as for the Tukey depth \cite{BurrF17}
(see also \cite{nagyuniformrates}) or the affine invariant version of
the Integrated Rank Depth \cite{10.1214/23-EJS2189}.

\begin{theorem}[Non-asymptotic bounds]\label{thm:stat-coDepth}
    Let  $P$ be any probability distribution on $\rset^d$ and
    \(\delta\in(0,1)\). The following assertions hold true.
    \begin{enumerate}
        \item If \(\delta^{2d-1}\leq
            1/(17e^8)\) and \(n\geq 4\), it holds with probability
            larger than $1-\delta$,
            \begin{equation}
                \label{eq:deviations-coD-d}
                \sup_{x\in\rset^d\setminus\{0\}  } \left|\widehat{\pDsymbol}(x)
                - \coD{x}{P}\right|
                \le \sqrt{\frac{d\log(n/\delta)}{n}}.
            \end{equation}

        \item   For   $t>0$, let $p_t = \P(r(X)\ge t)$ where $X\sim
            P$. If \((2\delta)^{d}\leq 1/8\), then
            \begin{equation}\label{eq:normalizedDev_2}
                p_t^{-1} \sup_{x:\, r(x)\ge t   } |\widehat{\pDsymbol}(x) - \coD{x}{P}|
                \le 3\sqrt{\frac{d\log((np_t+1)/\delta)}{np_t}} +
                \frac{2}{3}\frac{\log(1/\delta)}{np_t}.
            \end{equation}
    \end{enumerate}
\end{theorem}

Based on the classic
Vapnik-Chervonenkis inequality, the bound \eqref{eq:deviations-coD-d}
is of order $O_{\mathbb{P}}(1/\sqrt{n})$ up to a logarithmic factor
and involves the VC dimension of the class of Borel subsets
$\{H_{0,u}:\, u\in \rset^d\}$, $d$ namely, as expected. The bound
\eqref{eq:normalizedDev_2} is based on a variant of
Vapnik-Chervonenkis inequality tailored to rare events
\cite{lhaut2022uniform}. Refer to the Appendix for the technical proof.

From a practical perspective, it follows from Lemma \ref{lem:connect-aD-coD} that, for any $x\neq 0$, the plug-in estimator $\hatcoD{x}{P}$ can be expressed as a function of an empirical version of the polar depth $\aD{\theta(x)}{P_{\theta|r(x)}}$ and thus inherits the
computational advantages of the angular half-space depth, see
\cite{NagyLD21, 10.1007/s11222-025-10700-z}. Indeed, when applied to $P_n$, \eqref{eq:equiv-coD-aD} is written as follows:
\begin{equation*}
     \widehat{\pDsymbol}(x)  =
    P_{n,r}\left([r(x),+\infty]\right)\aD{\theta(x)}{P_{n,\theta|r(x)}}.
\end{equation*}
Algorithm
\ref{alg:CO-depthEstimation} summarizes the successive steps for
computing $\widehat{\pDsymbol}(x)$ exactly.

\begin{algorithm}\caption{Computation of the Empirical Polar
    Depth}\label{alg:CO-depthEstimation}
    {\bf Input} Dataset $\mathcal{D}=\{X_1, \ldots,\; X_n\}$,
    test point $x\in\rset^d\setminus\{0\}$ s.t. $r(x)\leq \max_{1\leq i\leq
    n}r(X_i)$.  
    \begin{enumerate}
        \item Compute the empirical radial depth (survival function)
            \begin{equation}\label{eq:alg_1:radial_part}
                \widehat P_r ([r(x),\infty)) := \frac{1}{n}\sum_{i=1}^n
                \un\{r(X_i)\ge r(x) \}.
            \end{equation}
        \item Form the angular outwards dataset $\mathcal{D}_\theta = \{
            \theta(X_i):\; r(X_i)\ge r(x) \text{ for } 1\le i\le n   \}$
            of size $n \widehat P_r ([r(x),\infty))$. Denote by
             $\widehat P_{\theta | r(x)}$ the associated empirical distribution 
        \item  Use the gnomonic projection in
            \cite{nagy2024theoretical} combined with the
            computational method in \cite{dyckerhoff2016exact} to
            compute the angular half-space
            depth of  $\theta(x)$ relative to $\widehat P_{\theta | r(x)}$:
            \begin{equation}\label{eq:alg_1:angular_part}
                \widehat{\mathrm{aD}}_{r(x)}(\theta(x)) =
                \aD{\theta(x)}{\widehat P_{\theta | r(x)}}.
            \end{equation}

    \end{enumerate}
    {\bf Output} Form the empirical version of $\coD{x}{P}$: 
    \begin{equation}\label{eq:algo_1:final_Form}
        \widehat{\pDsymbol}(x) ~=~ \widehat P_r([r(x),\infty))
         \widehat{\mathrm{aD}}_{r(x)}(\theta(x)).
    \end{equation}
\end{algorithm}

\medskip

The complexity of Algorithm~\ref{alg:CO-depthEstimation} is
essentially that of computing the angular halfspace depth component, whose
complexity is $O(n^{d-2})$,
with the distinctive feature that computation of the depth of the training
data points does not change the complexity, see \cite{NagyLD21}.

\section{Extrapolation under Regular Variations Assumptions}\label{sec:Extremes}

In this section, we study the tail behavior of the polar depth $\coD{x}{P}$ and explore
estimation strategies in extrapolation regimes, \textit{i.e.} in out-of-sample
regions. More precisely, we focus on test
points \( x \) such that \( \widehat{P}_r[r, \infty) \leq k/n \),
where \( k \geq 1 \) is significantly smaller than the sample
size \( n \) and place ourselves in the case where the
data distribution $P$  satisfies a multivariate regular variation property. This
assumption is legitimate in certain applications, particularly in finance,
when analyzing the logarithmic returns of multiple assets.
The more complex case in which marginal standardization is necessary, and
the assumption of regular variation applies to the standardized data, is analyzed in Section~\ref{sec:standardize}.

\subsection{Regularly Varying Distributions - Tail Properties of the
Polar Depth}\label{sec:tail_codepth-proba}

Turning to the polar depth properties in the tails, we now assume
that the probability distribution $P$ is regularly varying with index
$\alpha>0$, as described in Subsection \ref{subsec:basics_RV}, where
the notation used here is introduced. Key to the subsequent analysis, an immediate consequence
of the homogeneity
property~\eqref{homogeneity-nu} of the limit measures $\nu$  and $\nu_1$ in
 \eqref{eq:rv_weak_with_b}, \eqref{eq:rvWeak}, 
and Definition~\ref{def:coD},  is
that the polar depth functions $\coD{\point}{\nu}$ and
$\coD{\point}{\nu_1}$ are homogeneous as well, namely
for all $x\neq 0$ and any $t>0$, 
\begin{equation}\label{eq:homogeneous_coD_nu}
    \coD{tx}{\nu} = t^{-\alpha} \coD{x}{\nu} \text{ and }
    \coD{tx}{\nu_1} = c\, \coD{tx}{\nu}= t^{-\alpha} \coD{x}{\nu_1} .
\end{equation}

 Furthermore, since 
the limit measure $\nu$ may be expressed as a tensor product in pseudo-polar coordinates, the
observation made in Remark~\ref{rem:special-independent} allows us to
state the following result.
\begin{proposition}\label{prop:cod_nu}
    Assume that $P$ is a regularly varying distribution with regular variation
    index $\alpha>0$, limit measures $\nu,\nu_1$ in \eqref{eq:rvWeak} and limit
    angular probability measure $\Phi$. For any point
    \(x\in\R^d/\{0\}\), we have:
    \begin{equation}\label{eq:relationship_ad_limit_cod_exponent}
        \coD{x}{\nu_1} = \coD{r(x)\theta(x)}{\nu_1} =
        r(x)^{-\alpha}\coD{\theta(x)}{\nu_1} =
        r(x)^{-\alpha}\aD{\theta(x)}{\Phi}.
    \end{equation}
\end{proposition}
Our proposed depth extrapolation leverages the homogeneity property inherent in its definition. This follows the exact same principle used to analyze multivariate extreme values under the hypothesis of regular variation.  The theorem below shows that, under the smoothness
assumption introduced in Subsection~\ref{subsec:polar_def}, the polar
depth of the limit probability measure $\nu_1$ provides the limits of
the polar depth of asymptotically large observations when properly renormalized.

\begin{theorem}[Tail behavior of $\pD$ under regular
    variation]\label{theo:tailscoD}
    Suppose that $P$ satisfies the hypotheses of Proposition
    \ref{prop:cod_nu}, as well as
    Assumption~\ref{as:smoothnessHalfSpaces}. The following
    assertions hold true.
    \begin{enumerate}
        \item We have:
            \begin{equation}\label{eq:uniform_limit_extremes}
                \sup_{x \in K} | p_t^{-1}\coD{tx}{P} -
                \coD{x}{\nu_1} |\xrightarrow[t\to\infty]{} 0,
            \end{equation}
            where \(K\) is any compact set that does not contain the origin.
        \item If in addition ~\eqref{eq:rvStrong} holds, then, for
            any $\delta>0$, we have:
            \begin{equation}\label{eq:uniform_limit_extremes_constant_l_with_b}
                \sup_{x \in \rset^d: r(x)\ge \delta} | p_t^{-1}\coD{tx}{P} -
                \coD{x}{\nu_1} |\xrightarrow[t\to\infty]{} 0.
            \end{equation}
            and
            \begin{equation}\label{eq:uniform_limit_extremes_constant_l}
                \sup_{x \in \rset^d: r(x)\ge \delta} | t^{\alpha}\coD{tx}{P} -
                \coD{x}{\nu} |\xrightarrow[t\to\infty]{} 0,
            \end{equation}
    \end{enumerate}
\end{theorem}

The technical proof is given in the Appendix. The statistical method
we propose in the next subsection to estimate the polar depth in tail
regions is therefore based on the following heuristic derived from
Theorem~\ref{theo:tailscoD}, the first statement assertion in
particular: for $x$
with large radius $r(x)$,
\begin{equation}\label{eq:from_coD_P_to_coD_nu_1}
    \coD{x}{P} = p_{r(x)} \left(p_{r(x)}^{-1} \coD{r(x)
    \theta(x)}{P} \right)
    \approx  p_{r(x)}  \coD{\theta(x)}{\nu_1}  = p_{r(x)}  \aD{\theta(x)}{\Phi}.
\end{equation}
This suggests to use
as an approximation for $\coD{x}{P}$
when $x$ is (sufficiently) far from the origin, the product $
\PP(r(X)>r(x))\, \aD{\theta(x)}{\Phi} $, each factor of which can be
estimated with guarantees, as explained subsequently.  %


\subsection{Statistical Inference in Tail Regions - The $\XpD$
Approach}\label{sec:tail-codepth-stats}

In this section, we develop a statistical procedure for estimating
$\pD(x ; P)$ in tail regions where standard empirical methods fail due
to data scarcity. Specifically, we focus on test points $x$ such
that $\|x\|$ is so large that the number of training
observations $X_i$ satisfying $\|X_i\| \ge \|x\|$ is too small for
reliable empirical estimation. Our approach, referred to as
\textsc{eXtrapolated Polar Depth} ($\XpD$ in abbreviated form),
leverages the theoretical results from
Subsection~\ref{sec:tail_codepth-proba}. The key insight is the
asymptotic relationship~(\ref{eq:from_coD_P_to_coD_nu_1}), which
claims that for large $x$, the polar depth $\coD{x}{P}$ can be
approximated by $\PP(r(X)>r(x))\, \aD{\theta(x)}{\Phi}$. This
decomposition allows us to estimate the radial and angular
components separately, following established practices in
multivariate extreme value theory.

The inference method we propose thus proceeds in two steps. First, we estimate
the tail probability $p_{r(x)} = \PP(r(X)>r(x))$. It is a very
classic problem in
univariate extreme value analysis, for which many statistical
techniques are available. For $t>1$, set $U(t)=F^{\leftarrow}(1-1/t)$
where $F^{\leftarrow}$ denotes the left continuous generalized
inverse of the cdf $F(v)=\PP(r(X)\leq v)$.  For $r(x) \ge U(n/k)$ where $k
\ll n$, they rely on the approximation:
\begin{equation*}
    p_{r(x)}  = \PP(r(X)>U(n/k))\times \frac{\PP(r(X) >
    U(n/k)\frac{r(x)}{U(n/k)} )}{ \PP(r(X)>U(n/k))}
    \approx  \frac{k}{n} \Big( \frac{r(x) }{ U(n/k)}\Big)^{-\alpha}.
\end{equation*}

Second, we estimate the angular half-space depth
$\aD{\theta(x)}{\Phi}$ by means of a plug-in strategy using the
angles corresponding to the $k$
largest radii among the $n$ observations.

Precisely, given an i.i.d. sample $X_1,\; \ldots,\; X_n$ drawn from
$P$, a test point $x$, and
an integer $k$ satisfying $1 \ll k \ll n$, we construct our
estimators as follows. Let $R_k$ denote the $k$-th largest radius
among the sample, \textit{i.e.}, $r(X_{(1)})\geq \ldots\geq r(X_{(n)})$ and
$R_k=r(X_{(k)})$. Provided that an estimator $\widehat{\alpha}$ of the tail
index $\alpha$ is available, we
define the estimator $\widehat{p}_{r(x)}$ of $p_{r(x)}$ as
\begin{equation}\label{eq:hat_p_r}
    \widehat{p}_{r(x)} = \frac{k}{n} (r(x)/R_k)^{-\widehat \alpha}.
\end{equation}

For the angular component, consider the truncated angular sample
$\mathcal{T}_k =
\{\theta(X_{(1)}), \ldots, \theta(X_{(k)})\}$ and compute the
empirical angular distribution given
by $\widehat \Phi_k= (1/k)\sum_{i\leq k}
\delta_{\theta(X_{(i)})}$, which approximates the conditional
distribution of $\theta(X)$ given $r(X) \ge q({1-k/n})$, where
$q(x)=\inf\{t\geq 0: P_r([0, t)]\geq x)\}$ for $x\in
(0,1)$. The plug-in angular depth estimator on the region $\{x: r(x)\ge R_k\}$ is then written as
\begin{equation}\label{eq:hat_ahd_theta_x}
\widehat{\mathrm{aD}}_\infty(\theta(x)) = \aD{\theta(x)}{\widehat \Phi_k}.
\end{equation}

Combining both components, our estimator $\hatXpD{x}{P}$ is
eventually defined as
\begin{equation}\label{eq:estimator_coD_extreme_x}
\widehat{\mathrm{XpD}}(x) ~= ~
\widehat{p}_{r(x)} ~\widehat{\mathrm{aD}}_\infty(\theta(x))
~=~\frac{k}{n}~\Big(r(x)/R_k\Big)^{-\widehat \alpha}
~{\aD{\theta(x)}{\widehat \Phi_k}}.
\end{equation}


\begin{algorithm}[h!]\caption{$\XpD$: Polar Depth Extrapolation
in Tail Regions}\label{alg:depthEstimation-standard}
{\bf Input}: Data set $\mathcal{D}=\{X_1,\; \ldots,\; X_n\}$, integer
$k$ satisfying $1\ll k\ll n$, (large)
test point $x$ in $\rset^d\setminus\{0\}$. 
\begin{enumerate}
    \item Compute an estimator $\widehat \alpha$ of the regular
        variation index based on $r(X_1),\; \ldots,\; r(X_n)$.
    \item  Sort the data by decreasing order of magnitude
        $r(X_{(1)})\ge \dots \ge
        r(X_{(n)})$. Denote by
        $R_k = r(X_{(k)})$ the $k^{th}$ largest radius and by $\Theta_{(k)} =
        \theta(X_{(k)})$ the corresponding angle. Form the truncated
        angular sample $\{
        \Theta_{(1)}, \ldots, \Theta_{(k)}\}$ and, based on it,
        compute the empirical angular distribution:
        $$\widehat \Phi_k= \frac{1}{k}\sum_{i=1}^k
        \delta_{\Theta_{(i)}}.$$
    \item Compute the angular half-space depth of $\theta(x)$
        relative to $\widehat \Phi_k$:
        $$
       \widehat{\mathrm{aD}}_\infty(\theta(x)) = \aD{\theta(x)}{\widehat \Phi_k}.
        $$

    \item Estimate  the quantity $p_{r(x)} = \PP(r(X)>r(x))$ by
        \begin{equation}
            \widehat{p}_{r(x)} = \frac{k}{n}
            (r(x)/R_k)^{-\widehat \alpha}.
        \end{equation}

\end{enumerate}
{\bf Output}: Estimate of $\coD{x}{P}$ by
\begin{equation}
    \widehat{\mathrm{XpD}}(x) ~= ~
    \widehat{p}_{r(x)} ~\widehat{\mathrm{aD}}_\infty(\theta(x)).
\end{equation}

\end{algorithm}

Before establishing guarantees for the accuracy of the method
sketched above and summarized by Algorithm \ref{alg:depthEstimation-standard}, a few remarks are in
order. First, as shown by \eqref{eq:estimator_coD_extreme_x}, the
angular and radial components are estimated independently. This
suggests the possibility of using  a fraction $k/n$ of the largest
data points that is different for each component, one for computing
$\widehat{p}_{r(x)}$ using \eqref{eq:hat_p_r} and another for
$\widehat{\mathrm{aD}}_\infty(\theta(x))$ defined in \eqref{eq:hat_ahd_theta_x}.
Although such differentiation could lead to improvements regarding
the accuracy of the $\XpD$ estimator in practice, we use the same $k$
for the radial and angular components in the subseequent
theoretical analysis for the sake of simplicity, extensions being
left to the reader.
Second, we point out that any consistent estimator $\widehat{\alpha}$
of the tail index of $\alpha$ may be used in
\eqref{eq:estimator_coD_extreme_x}. For instance, the Hill estimator
for $\gamma= 1/\alpha$ is asymptotically normal (under
appropriate second-order assumptions) with asymptotic variance of
order $\gamma/k$, where $k$ is the number of order statistics
used for estimation, up to a bias term that vanishes as $k/n\to
0$ and  $n\to \infty$ under second-order assumptions
\citep{beirlant2006statistics}. Many other inference techniques for
tail index estimation with (non-) asymptotic guarantees have been
documented in the literature, see \cite{boucheron2015tail},
\cite{Carpentier2015}, \cite{bertail2025tail}
or~\cite{lederer2025adaptivetailindexestimation} for instance.

\subsection{Non-asymptotic Guarantees for $\XpD$ - Finite Sample
Bounds}\label{sec:tail-codepth-stats2}

We now study the accuracy of the method described in Algorithm
\ref{alg:depthEstimation-standard}. Observe first that the error made
when estimating \(\coD{x}{P}\) by $\widehat{\mathrm{XpD}}(x)$ can be decomposed
in two error terms, one due to the radial component and the other
caused by the angular component. Indeed, it follows from the
triangle
inequality that
\begin{equation*}
\big| \widehat{\mathrm{XpD}}(x) -\coD{x}{P}\big| \le
|\widehat p_{r(x)} -p_{r(x)}\big| \widehat{\mathrm{aD}}_\infty(\theta(x))+
p_{r(x)}\big| \widehat{\mathrm{aD}}_\infty(\theta(x)) - \aD{\theta(x)}{
P_{\theta|r(x)}} \big|,
\end{equation*}
and consequently,
\begin{equation}\label{eq:error_decomposxcod}
p_{r(x)}^{-1} \big| \widehat{\mathrm{XpD}}(x) -\coD{x}{P}\big| \le
\left|\frac{\widehat p_{r(x)}}{p_{r(x)}}-1\right| + \big|
\widehat{\mathrm{aD}}(\theta(x)) - \aD{\theta(x)}{
P_{\theta|r(x)}} \big|.
\end{equation}

The main results stated in this subsection provide an upper
confidence bound for each of the two error terms, see Proposition
\ref{thm:radial-error-part} and Proposition
\ref{thm:control-angular-part} below, which permits next to deduce
non-asymptotic guarantees for the estimator
\eqref{eq:estimator_coD_extreme_x}, see Theorem \ref{theo:xcod}. With
regard to the (relative) radial error term, as the analysis we
conduct does not require any assumption about the inference technique
chosen to compute an estimate $\widehat{\alpha}$ of the regular
variation index $\alpha$, the bound established naturally involve the
absolute deviation \(\Delta_{\alpha}=|\widehat{\alpha}-\alpha|\).

The following (deterministic) quantities also appear in the bounds
established below.
\begin{align}
\underline{u}_{n,k} &=
U\left(\frac{n}{k}\frac{1}{1+k^{-1/4}}\right)\quad,\quad
\overline{u}_{n,k}=U\left(\frac{n+1}{k}\frac{1}{1-k^{-1/4}}\right)\quad,\quad
C_{n,k}=\frac{n+1}{n}\frac{k^{1/4}+1}{k^{1/4}-1},\label{eq:u_quantiles}\\
B_{U} &=
\sup_{s\in[3/4,4]}{\left|\left(\frac{U(ns/k)}{U(n/k)}\right)^{\alpha}-s\right|}
\quad,\quad
B_{F} = \sup_{s\geq 1}\left|
s^{-\alpha}-\frac{n}{k}p_{U(n/k)s}
\right|,\label{eq:u_and_f_biases}\\
B_L&=\sup_{s\geq 1}s^{\alpha}\left|
s^{-\alpha}-\frac{n}{k}p_{U(n/k)s} \right|\quad,\quad
B_{\theta} = \sup_{{t\geq
U(2n/k),\omega\in\sphere} }{\left|\aD{\omega}{P_{\theta|t}} -
\aD{\omega}{\Phi}\right|},\label{eq:L_and_angle_biases}\\
B_{RV}&=B_{U}+B_{F}\quad\textnormal{and}\quad B_{SRV} =
B_U+B_L.\label{eq:rv_and_strong_rv_biases}
\end{align}

The result stated below provides a nonasymptotic control of the
relative deviations of the survival function of the radial component
of a heavy-tailed random vector, which is key to bound (the first
component of) the term on the right-hand side of
\eqref{eq:error_decomposxcod}. It should be noted that this
result is also of independent interest, as it can be used to control the
relative deviations when estimating the survival function of a univariate
heavy-tailed distribution.

\begin{proposition}[Radial error]\label{thm:radial-error-part} Let
\(\delta\in(0,1)\) and
\(k\geq (12\log(2/\delta))^2\). The following assertions hold true.
\begin{enumerate}[label=(\roman*)]
    \item If \eqref{eq:rvWeak} holds, then, for all \(n>k\) and \(M>2\)
        the following inequality holds with probability larger
        than \(1-\delta\):
        \begin{equation*}
            \sup_{r(x)\in \left[\overline{u}_{n,k},
            U(\frac{Mn}{k})\right]}
            \left|\frac{\widehat{p}_{r(x)}}{p_r(x)}-1\right|  \leq M
            \left(\frac{\Delta_{\alpha}}{e \min{(\alpha,
                \hat{\alpha})}} +
                \sqrt{\frac{12\log(2/\delta)}{k}} +
                \frac{12\log(2/\delta)}{k} + \frac{4}{n} +
            B_{RV}\right).
        \end{equation*}
        In addition, \(B_{RV}(n,k)\to 0\) as \(n/k\to +\infty\).
    \item Suppose that \eqref{eq:rvStrong} holds and \(n, \; k\)
        are such that
        \begin{equation}\label{eq:main_result:uniform_boundness_slowly_varying}
            \sup_{s\in[3/4,4]}{\left|\left(\frac{U(ns/k)}{U(n/k)}\right)^{\alpha}-s\right|}\leq
            \frac{1}{2},\quad
            \frac{\overline{u}_{n,k}}{\underline{u}_{n,k}}\leq 2,
        \end{equation}
        and there exists an event \(\mathcal{D}_{n,k}\) of
        probability larger than \(1-\delta\)
        such that
        \begin{equation}\label{eq:hypotheses_on_Delta}
            \sqrt{\Delta_{\alpha}}\leq \frac{4}{7},
        \end{equation}
        then the following inequality holds with probability bigger
        than \(1-2\delta\):
        \begin{equation*}
            \sup_{\overline{u}_{n,k}\leq r(x)\leq
            \exp{\left(\Delta_\alpha^{-1/2}\right)}\underline{u}_{n,k}}
            \left|\frac{\widehat{p}_{r(x)}}{p_{r(x)}}-1\right|\leq
            \frac{7}{4}\sqrt{\Delta_{\alpha}} +
            4\sqrt{\frac{3\log(2/\delta)}{k}} +
            \frac{24\log(2/\delta)}{k} + \frac{8}{n} + 4B_{SRV},
        \end{equation*}
        and the bias term \(B_{SRV}\) converges to \(0\) as
        \(n/k\to +\infty\).
\end{enumerate}
\end{proposition}
The proof relies on the concentration properties of order statistics, as well as on the form and properties of the survival function of a regularly varying distribution. Details are provided in the Appendix.  Before considering the angular error term, the following points should be noted.
\begin{remark}[On hypothesis
\eqref{eq:main_result:uniform_boundness_slowly_varying}] Under
the strong regular variation assumption
\eqref{eq:rvStrong},
the condition
\eqref{eq:main_result:uniform_boundness_slowly_varying}
is eventually
satisfied as \(
\sup_{s\in[3/4,4]}{\left|\left(\frac{U(ns/k)}{U(n/k)}\right)^{\alpha}-s\right|}\)
converges to \(0\) as \(n/k\to +\infty\).
\end{remark}

\begin{remark}[On the order of \(\Delta_{\alpha}\) and condition
\eqref{eq:hypotheses_on_Delta}]
Typically, if $\widehat \alpha$ is the multiplicative inverse of the Hill estimator,
\citep[see][]{boucheron2015tail}, then, with probability bigger
than \(1-\delta\) we have that $\Delta_{\alpha}\le
C\sqrt{\log(1/\delta)/k} + l(k)$ where $l(k)$ is a remainder term
including a bias term that vanishes as $k/n\to 0$ and a negligible
deviation term of order $O(1/k)$ and $C$ is a constant.
This guarantees
that condition \eqref{eq:hypotheses_on_Delta} is satisfied.
\end{remark}

The proposition below provides confidence upper bounds for the second term on the right hand side of \eqref{eq:error_decomposxcod}. As expected, these bounds are of the order $O_{\mathbb{P}}(1/\sqrt{k})$ (up to a logarithmic factor), where $k$ denotes the number of data points explicitly used in the calculation of the empirical version of the angular depth stipulated in Algorithm \ref{alg:depthEstimation-standard}.

\begin{proposition}[Angular error]\label{thm:control-angular-part}
Let \(0<\delta\leq 2^{-(3/d+1)}\) and \(k\geq (12\log(2/\delta))^2\).
If Assumptions \ref{as:continuousConditionalAngular} and
\ref{as:smoothnessHalfSpaces}
are satisfied, then with probability
greater than \(1-2\delta\) we have
\begin{align*}
    \sup_{x:r(x)\geq
    U\left(\frac{n}{2k}\right)}\left|\widehat{\mathrm{aD}}_\infty(\theta(x))-\aD{\theta(x)}{P_{\theta|r(x)}}\right|
    & \leq
    6C_{n,k}\sqrt{\frac{d\log{(2(k+k^{3/4}+1)/\delta)}}{k+k^{3/4}}}\nonumber\\&+\frac{4C_{n,k}}{3}\frac{\log(2/\delta)}{k+k^{3/4}}
    + 2B_\theta.
\end{align*}
\end{proposition}

The proof, detailed in the Appendix, is mainly based on the arguments of Theorem \ref{thm:stat-coDepth}, applied in the context of rare events. The following theorem, which is a direct application of Propositions
\ref{thm:radial-error-part}
and \ref{thm:control-angular-part}, gives finite sample guarantees
for the procedure summarized in Algorithm \ref{alg:depthEstimation-standard}.

\begin{theorem}[Non-asymptotic guaranties for \(\hatXpD\)]\label{theo:xcod}
Let \(0<\delta\leq 2^{-(3/d+1)}\) and \(k\geq (12\log(2/\delta))^2\).

Suppose Assumptions \ref{as:continuousConditionalAngular} and
\ref{as:smoothnessHalfSpaces} are satisfied. The assertions below hold true.

\begin{enumerate}[label=(\roman*)]
    \item If the weak regular variation condition
        \eqref{eq:rvWeak} is satisfied, then, for any \(M>2\), with
        probability bigger
        than \(1-3\delta\), it
        holds that 
        \begin{multline}
        \sup_{r(x)\in \left[\overline{u}_{n,k},
            U(Mn/k)\right]} p_{r(x)}^{-1}\big|
        \hatXpD(x) -\coD{x}{P}\big|\leq 
            \frac{M\Delta_{\alpha}}{e
            \min{(\alpha,\hat{\alpha})}} \\  +
            6C_{n,k}\sqrt{\frac{d\log{(2(k+k^{3/4}+1)/\delta)}}{k+k^{3/4}}}
            + \sqrt{\frac{12M^2\log(2/\delta)}{k}} \\
            +\frac{\log(2/\delta)}{k}\left(12M+\frac{4C_{n,k}}{3}\right)
            +\frac{4M}{n} +
            MB_\textnormal{RV} + 2B_{\theta}
            \label{eq:xcod_finite_sample_bound_fixed_M}.
        \end{multline}
        Moreover, \(B_{RV}\) and
        \(B_\theta\) converge to 0 as \(n/k\to +\infty\).
    \item Suppose the strong regular variation condition
        \eqref{eq:rvStrong} is in place
        and \(n\) and \(k\) are such that
        \begin{equation*}
            \sup_{s\in[3/4,4]}{\left|\left(\frac{U(ns/k)}{U(n/k)}\right)^{\alpha}-s\right|}\leq
            \frac{1}{2},\quad\frac{\overline{u}_{n,k}}{\underline{u}_{n,k}}\leq
            2,
        \end{equation*}
        and there exists an event \(\mathcal{D}_{n,k}\) of
        probability bigger than \(1-\delta\)
        such that
        \begin{equation*}
            \sqrt{\Delta_{\alpha}}\leq \frac{4}{7}.
        \end{equation*}
        Then, with probability bigger than \(1-4\delta\) it holds
        that 
        \begin{multline}\label{eq:xcod_finite_sample_bound_ration_version}
        \sup_{\overline{u}_{n,k}\leq r(x)\leq
            \exp{\left(\Delta_\alpha^{-1/2}\right)}\underline{u}_{n,k}}
            p_{r(x)}^{-1}\big|
        \hatXpD(x) -\coD{x}{P}\big|\leq 
            \frac{7}{4}\sqrt{\Delta_{\alpha}} \\
           + 6C_{n,k}\sqrt{\frac{d\log{(2(k+k^{3/4}+1)/\delta)}}{k+k^{3/4}}}
            + 4\sqrt{\frac{3\log(2/\delta)}{k}}
            +\frac{\log(2/\delta)}{k}\left(24+\frac{4C_{n,k}}{3}\right)
            + \frac{8}{n} +
            4B_\textnormal{SRV} +
        2B_{\theta},
        \end{multline}
        and the bias terms \(B_\textnormal{SRV}\) and \(B_{\theta}\)
        converge to \(0\) as \(n/k\to+\infty\).
\end{enumerate}
\end{theorem}

The proof of the theorem above is omitted insofar  as it is based on \eqref{eq:error_decomposxcod}, combined with the union bound and the bounds stated in Propositions
\ref{thm:radial-error-part}
and \ref{thm:control-angular-part}. The proofs of these results are detailed in the Appendix. 
The above upper confidence bounds should be compared with those of Theorem \ref{thm:stat-coDepth}. Apart from the fact that the number of data points used to calculate the statistical estimates is no longer $n$ but $k$, they involve bias terms inherent in the asymptotic assumption of multivariate regular variation, as in \textit{e.g.} \cite{ClemenconJalalzaiSabourinSegers2023}.

\section{Standardized Polar Depth and Extrapolation}\label{sec:standardize}

In multivariate analysis, marginal transformations aimed at standardizing the components of a dataset are a common practice. This approach is particularly relevant when the objective is to perform out-of-domain generalization using EVT tools. In such contexts, the probability integral transform  emerges as a natural and standard choice, both in multivariate EVT  and copula analysis.
Standardization  is especially useful when different components exhibit varying scales or, in heavy-tailed settings, when they possess distinct tail indices. Beyond facilitating interpretation within the observed domain, this standardization serves as a convenient and widely adopted tool for extrapolation. It also enables the separation of two key analytical challenges: the study of the dependence structure and the characterization of marginal behaviors.

In this section, we introduce a notion of depth that incorporates such a marginal standardization. We establish theoretical guarantees—both probabilistic and statistical—concerning the tail behavior of the proposed \emph{standardized polar depth}, as well as that of its estimator. We define the standardized polar depth
in terms of marginally standardized variables with unit Pareto
marginal distributions,

\begin{equation}
    \label{eq:def_polardepth_standardized}
    \coDs{x}{P} = \coD{v(x)}{v_\#P}    ,
\end{equation}
where $v(x) = (v_1(x_1)\ldots, v_j(x_j))$, $v_j(y) = (1 - F_j(y))^{-1}$, see the background subsection~\ref{subsec:basics_RV}.
In this section, we focus on the multivariate tail behavior on the standardized scale, specifically $\coDs{x}{P}$ when $\|v(x)\|$ is large.

\subsection{The Standardized Margins Framework}
We assume throughout this Section that the vector $V=v(X)$ satisfies \eqref{eq:standardRV}, in other words, that it is regularly varying with limit measure $\mu$ and scaling function $b(t) = t$, see Subsection~\ref{subsec:basics_RV} for some context about this assumption. In particular note that Assumption~(\ref{eq:rvWeak}) regarding $X$ implies \eqref{eq:standardRV} for its standardized version $V$. Recall from \eqref{eq:angularMeasure} the notation $\Phitrad$ for the angular measure in the standard scale and introduce the probability measure on the sphere, $\Phisprob = \mu(\ball^c)^{-1}\Phitrad$. Then, for any measurable set $A\subset\sphere$,
$$\Phisprob(A) = \lim_{t\to\infty} \PP[\theta(V)\in A~|~r(V)>t].$$
In the sequel, the angular depth with respect to limit angular measures $\Phitrad,\Phisprob$ will play an important role. Note that in the original definition of the  angular halfspace depth \citep{NagyDemniButtarazziPorzio2023,nagy2024theoretical}, the second argument (the measure) is necessarily a probability measure. Here we straightforwardly extend the original definition to the case where the measure argument (say, $\varphi) $ is a finite, positive measure on the sphere,  by letting 
$$
\aD{u}{\varphi} = \varphi(\sphere)\aD{u}{ \varphi(\sphere)^{-1}\varphi}, \qquad u\in\sphere.
$$

To proceed, we introduce two additional assumptions rooted in multivariate EVT. 
We shall need some uniformity in the convergence in~\eqref{eq:standardRV}. The following condition  is verified
under regular variation of the density of $V$ similar to those
made \emph{e.g.} in \cite{cai2011estimation} or \cite{clemenccon2025regression}.

 \begin{assumption}[Uniform regular variation of $V$]\label{as:V_uniform_rv}
  $$
\sup_{u\in\sphere} | \PP[\theta(V)  \in H_{0,u}\cap \sphere~|~\| V \|>t] - \Phisprob( H_{0,u}\cap \sphere) | \tto 0
$$
\end{assumption}

A key step of our subsequent analysis is the uniform control of the empirical angular measure on the standardized scale obtained  in \cite{ClemenconJalalzaiSabourinSegers2023}. Unfortunately this control does not extend to subsets of the sphere that are in proximity to the boundaries of the positive orthant in a simple manner. This limitation is characteristic of multivariate EVT \citep{einmahl2001nonparametric,einmahl2009maximum,ClemenconJalalzaiSabourinSegers2023}. Consequently, we assume that $\Phitrad$ assigns negligible mass near the boundary of the positive orthant, thereby justifying the exclusion of such sets from our error analysis. An alternative approach could involve relaxing these assumptions and incorporating an additional error term. However, for the sake of simplicity, we refrain from pursuing this direction and leave  these details to the interested reader. 
  For any $\tau>0$, denote by $\sphere^\tau$ the subset of $\sphere_+$ defined by $ \sphere^\tau= \{ x \in \sphere_+: \min_j x_j > \tau \} $. 
\begin{assumption}[No angular mass near the boundaries]\label{as:nomass-near-boundaries}
  For some $\tau \in (0,1)$, $\Phitrad(\sphere_+\setminus\sphere^\tau) = 0$. 
\end{assumption}
Let $\sphere_+ = \sphere\cap\rset_+^d$ denote the positive orthant of the sphere. Since $V\in[1,\infty)^d$ with probability one, $\Phitrad$ necessarily concentrates on $\sphere_+$. In the sequel we let $\sigma_{d-1}$ denote the surface measure on $\sphere$, with total mass
\begin{equation}\label{eq:surface_measure_sphere}
\sigma_{d-1}(\sphere) = \frac{2 \pi^{d/2} }{\Gamma(d/2)} :=s_{d-1}. 
\end{equation}
By symmetry, the surface measure of the intersection of the positive orthant with the unit sphere is $\sigma_{d-1}(\sphere_+) = 2^{-d}s_{d-1}$.
\medskip

\noindent {\bf Tail equivalent of the standardized polar depth.}
We first present a reformulation of $\coD{y}{v_\#P}$ for large $y$, expressed in terms of $\Phitrad$ (and the corresponding angular depth), under the RV assumptions outlined above. 

\begin{proposition}[Limit behaviour of $\coD{y}{v_\# P} $ for large $y$]\label{prop:limit_coD_stand}
Let  
  $$
B_1(t) =  \sup_{ y: \|y \| \ge t}   \Big|~ \|y\| \coD{y}{v_\# P} - \aD{\theta(y)}{\Phitrad}~\Big| 
 $$
  Under 
  Assumption~\ref{as:V_uniform_rv}, we have $B_1(t) \tto 0$. 
\end{proposition}

An important interpretation of Proposition~\ref{prop:limit_coD_stand}, analogous to the results in Section~\ref{sec:Extremes} for the original scale, is that for large $y$,
\begin{equation}
\label{eq:heuristic_estimate_cod_standardized}
\coD{y}{v_\# P} \approx_{y\to\infty} \frac{1}{\|y\|}\aD{\theta(y)}{\Phitrad} = \pD{y}{\mus}, 
\end{equation}
where the latter identity derives immediately from the homogeneity of the polar depth and the fact that $\mus\{\|x\|>t, \theta(x)\in A) = t^{-1}\Phitrad(A)$ for $t>0$ and $A$ a Borel subset of $\sphere$. 

\subsection{Statistical Inference}

Given the
definition of the standardized polar depth $\coDs{x}{P}$ in~\eqref{eq:def_polardepth_standardized} and
Heuristic~\eqref{eq:heuristic_estimate_cod_standardized}, we consider 
an estimator of $\coDs{x}{P}$, designed for a new input $x$ such that $\|v(x)\|$ (or an estimation thereof) is large, of the form 
\begin{equation}\label{eq:def_estimator_standardizedDepth}
\hatcoDs{x} = \frac{1}{\| \tilde v (x)\|} \aD{\tilde v(x)}{\hatPhitrad}, 
\end{equation}
where $\tilde v(x)$ and $\hatPhitrad$ are estimators of the Pareto transformation $v(x)$ and the angular measure $\Phitrad$, respectively.
In this paper, we derive non-asymptotic guarantees for a specific version of $\hatcoDs{x}$ where $\hatPhitrad$ 
is chosen as the empirical estimator of the angular
measure, as in \cite{ClemenconJalalzaiSabourinSegers2023},  see
also \cite{einmahl2001nonparametric,einmahl2009maximum}, 
constructed using a marginal 
rank-transformation of the input. Namely, with $P_{n,j}(\point)= n^{-1}\sum_{i\le n}\delta_{X_{i,j}}(\point)$, 
$$\hat v_j(x) =  \frac{1}{1- \frac{n}{n+1}P_{n,j}(-\infty,x_j)},$$
and $\hat V_i = \hat v (X_i)$, we consider here 
$$
\hatPhitrad(\point) = \frac{n}{k}\sum_{i\le n} \un\{ \theta(\hat V_i)
\in \point, ~ \| \hat V_i\|\ge n/k \}.
$$

Our main result Theorem~\ref{thm:main_standardized} encompasses in principle any estimator $\tilde v$ in the definition of $\hatcoDs{x}$ in~\eqref{eq:def_polardepth_standardized},  of the kind 
$\tilde v(x) = (\tilde v_1(x_1),\ldots, \tilde v_d(x_d))$ where $\tilde v_j(x_j)= 1/\tilde p_j(x_j)$ and $\tilde p_j(t)$ is an estimator of the marginal survival function $p_j(t) = \PP[X_j>t]$. It is important to consider estimators  $\tilde v$ that are different from $\hat v$ used in the empirical angular measure,  in order to allow extrapolation, namely,  meaningful estimation of $v(x)$ (and thus of $\coDs{x}{P}$) even when some component $x_j$ is as large as, or larger than, the largest observation $X_{(1),j}$ in direction $j$. 

The upper bound on the error includes a term $h$ representing the maximum relative deviation $ |\tilde p_j(x_j)/p_j(x_j) -1 |$. For concreteness we focus further our analysis in Proposition~\ref{prop:control_relative_marginal_dev_x_small} and Corollary~\ref{cor:total_error_standardized} below, on a widely used semi-parametric version of $\tilde p_j$, see \emph{e.g} \cite{coles2001introduction},
using an empirical estimator below a high threshold and (a variant) of a Pareto model above the same high threshold, namely we take $\tilde{ p}_j(t) = \tilde{p}_{sp,j}$ where 
\begin{equation}
\label{eq:def_clever_tildev}
\tilde{p}_{sp,j}(t) =
\begin{cases}
    P_{n,j}(t, \infty) & \text{ if } t \le X_{( \lfloor k/2 \rfloor), j },  \\
    \frac{k}{n}\Big(\frac{t}{X_{(k)}}\Big)^{-\hat \alpha_j}& \text{
    if } t >  X_{( \lfloor k/2 \rfloor), j } .
\end{cases}
\end{equation}
There is no reason why the same integer $k$ should be used both for estimating the angular measure and the margins. However, since we are not examining the theoretical implications of the choice of $k$ here, it seems simpler to us to use a single index $k$.

Guided by~(\ref{eq:heuristic_estimate_cod_standardized}), our analysis is structured around the following error decomposition,  
\begin{multline}
\label{eq:error_decompos_standardized}
    |\hatcoDs{x} - \coDs{x}{P}|
     \le | \hatcoDs{x}  - \pD{\tilde  v(x)}{\mus} | ~+~ 
    | \pD{\tilde v(x)}{\mus} - \pD{ v(x)}{\mus} |   \\ 
      +  |\pD{ v(x)}{\mus} - \pD{v(x)}{v_\# P} |  
    = \Delta_{s,1}(x) + \Delta_{s,2}(x)+\Delta_{s,3}(x).
\end{multline}

Controlling the first  term $\Delta_{s,1}(x)$ involves  adapting some results in~\cite{ClemenconJalalzaiSabourinSegers2023} regarding the deviations of the empirical angular measure, and Lipschitz continuity properties. 
The middle term $\Delta_{s,2}$ requires a control of $\|\tilde v - v\|$,  which is shown to be of the same order as the relative deviations $\tilde p_j(x_j) /p_j(x_j) -1$.
The last term $\Delta_{s,3}(x)$ above is a bias term which vanishes  when $v(x)$ is large, under the conditions of  Proposition~\ref{prop:limit_coD_stand}. 
These controls of $\Delta_{s,i}, i=1,2,3$,  permit to prove the following main result of this section. 
\begin{theorem}\label{thm:main_standardized}
  Let Assumptions~\ref{as:V_uniform_rv} and \ref{as:nomass-near-boundaries} hold true, and suppose that $\Phitrad$ has a bounded density $\phi_s$ w.r.t. the surface measure $\sigma_{d-1}$ on $\sphere$, \textit{i.e.} $\|\phi_s\|_\infty<\infty$.  
  Let $k$ be the number of extreme observations involved in the construction of the estimator $\hatcoDs{\point}$ through the empirical angular measure. Let $\tilde p_j, j\le d$, be some estimators of the marginal survival functions $p_j$, $\tilde v_j(t) = 1/\tilde p_j(t)$ and $\tilde v(x) = \big(\tilde v_1(x_1), \ldots, \tilde v_d(x_d)\big)$. 
  Let $(\delta_1,\delta_2) \in(0,1)^2$ such that $\delta_1+\delta_2\le 1$, $0<h<1/8$ and  $\Omega\subset\rset^d$ such that,  on an event $\mathcal{E}_1$ of probability at least $1-\delta_1$, $\max_{j\le d}\sup_{x\in\Omega}|\tilde p_j(x_j)/p_j(x_j) - 1|< h$.   
    Then for any $t\ge 0$,  
    we have with probability at least $1-\delta_1 - \delta_2$, 
\begin{multline*}
  \sup_{x \in \Omega, \|\tilde v(x)\|\ge t}
  \|\tilde v(x)\| ~  \Big| \hatcoDs{x} - \coDs{x}{P} \Big| \le
    D(k,\delta_2) +  B(k,n)  \\ 
    +\frac{h}{1-h}\Big(\Phitrad(\sphere)+  \|\phi_s\|_\infty\frac{64\sqrt{2} d^{3/2}}{3\pi}\sigma_{d}(\sphere_+)\Big) +  \frac{1}{1 - h} B_1(t(1-h)), 
    \end{multline*}
where 
$$D(k,\delta)\le \frac{C}{\sqrt{k}}   \log((d+1)/\delta) \Big(   (d+1)^{5/4} +
\tau^{-1/2}  (d + c(d) \log(d/c(d)) + \log(k))
  \Big), 
  $$
$C$  is a  universal constant, and 
$$
c(d) \le \|\phi_s\|_\infty\bigg(
 \frac{32\sqrt{2} d^{2}}{3\pi} + 
 \frac{d^{3/2}}{\sqrt{2\pi}}
      \max\Big(\frac{2 }{\sqrt{1 - 4\tau^2}}, 
       \pi 
       \Big)
\bigg)\sigma_{d-1}(\sphere_+).
$$
The quantity $B(k,n)$ is a   bias term, 
arising from the empirical angular measure, described in Proposition~\ref{prop:bound_error_phitrad}, and 
$B_1$ is an additional bias term defined in Proposition~\ref{prop:limit_coD_stand} and stemming from the asymptotic approximation of the polar depth, which vanishes at infinity.
\end{theorem}

To illustrate the scope of Theorem~\ref{thm:main_standardized}, when using $\tilde p_j = \tilde p_{sp,j}$ defined in \eqref{eq:def_clever_tildev}, the next proposition controls the amplitude of the error term $h$ related to the marginal relative errors $\tilde p_{sp,j}(x_j)/p_{j}(x_j) - 1$. 

\begin{proposition}[Relative deviations of estimated survival functions at intermediate levels]\label{prop:control_relative_marginal_dev_x_small}
  Let $\tilde p_{sp,j}$ be as defined in~(\ref{eq:def_clever_tildev}). For $\delta\in(0,1)$, let $$\ell(n,\delta) = 2 \sqrt{ 2 \log(en) + \log(16d/\delta)  }.$$  
For $r\ge 1$,  let $X_{(r,j)}$ denote the $\lfloor r \rfloor^{th}$ largest order statistic of the marginal sample $(X_{i,j}, i\le n)$ for $j\le d$. 
 With probability at least $1-\delta$, we have 
 $$
 \max_{j\le d} \sup_{ x_j  \le X_{(k/2,j)} } \Big|
 \frac{\tilde p_{sp,j}(x_j)}{p_j(x_j)} - 1\Big|
\le \frac{2 \ell(n,\delta)}{\sqrt{k}} \sqrt{1 +1/n} + \frac{4 \ell(n,\delta)^2}{k} (1 +1/n) . 
$$
\end{proposition}

The following corollary synthesizes our analysis by combining Proposition~\ref{prop:control_relative_marginal_dev_x_small}, which governs the deviations of $\tilde{p}_{sp,j}(x_j)$ for small to "intermediate" values of $x_j$, and Lemma~\ref{lem:relative_deviation_evt_estimator} in the Appendix, which addresses the case of large $x_j$. It provides three explicit error bounds for the standardized polar depth, integrating marginal error sources with those arising on the standardized scale. The hierarchy of these bounds reflects progressively stronger assumptions on marginal regular variation.
 \begin{corollary}[Total error bound for standardized polar depth]\label{cor:total_error_standardized}
   Suppose that Assumptions~\ref{as:V_uniform_rv},~\ref{as:nomass-near-boundaries} hold true. Let $\tilde p_{sp,j}$ be the semi parametric estimator of $1- F_j$ defined in~(\ref{eq:def_clever_tildev}), and let $\tilde v$ denote the associated empirical marginal transformation function.
   With the notations of Theorem~\ref{thm:main_standardized}, define for $h\in(0,1)$, $t>0$ a fixed radial threshold on the transformed scale, a small probability $\delta\in(0,1)$, and $0<k\le n$, an error term 
   $$
   \mathrm{Err}(h,k,n,t,\delta) =  D(k,\delta) +  B(k,n) + \frac{h}{1-h}(\Phitrad(\sphere)+ \|\phi_s\|_\infty\frac{64\sqrt{2} d^{3/2}}{3\pi}\sigma_{d}(\sphere_+))
    + \frac{1}{1 - h} B_1(t(1-h)). 
$$
Recall from Proposition~\ref{prop:control_relative_marginal_dev_x_small} the notation
   $$
   \ell(n,\delta) = 2 \sqrt{ 2 \log(en) + \log(16d/\delta)  }. 
   $$
\begin{enumerate}
\item Define the (random) domain
  $$
  \Omega_1(t) =  \prod_{j=1}^d (-\infty, X_{(k/2,j)}] \bigcap \{x\in\rset^d: \| \tilde v(x)\| \ge t \}, 
  $$
  and let  $h_1(k,n,\delta)$ be a marginal error term defined as 
  $$
  h_1(\delta) = \frac{2 \ell(n,\delta)}{\sqrt{k}} \sqrt{1 +1/n} + \frac{4 \ell(n,\delta)^2}{k} (1 +1/n) .
  $$
   For  $(k,n)$  such  that
$h_1(\delta)\le 1/8$, 
we have    with probability at least $1 - 2\delta $,
  $$
  \sup_{ x \in \Omega_1(t)}  \|\tilde v(x)\|~\Big| \hatcoDs{x} - \coDs{x}{P} \Big| \le \mathrm{Err}\Big(h_1(\delta), k, n, t, \delta \Big).  
  $$

\item Assume additionally that the marginal variables $X_j$ are regularly varying with regular variation index $\alpha_j>0$, namely $\PP[X_j>t ] = x^{-\alpha_j}/L_j(t)$, where $L_j$ is a slowly varying function. Define
  $$
  \Delta_{\alpha} = \max_{j\le d} |\hat \alpha_j - \alpha_j |. 
  $$
  Let $\underline u_{n,k,j}, \overline u _{n,k,j}, B_{RV,j}, B_{SRV,j}$ be defined as in  \eqref{eq:u_quantiles}--~\eqref{eq:rv_and_strong_rv_biases}, up to replacing the distribution of $\|X\|$ with the distribution of $X_j$, in other words up to replacing $U(t)$ with $U_j(t)= F_j^{\leftarrow}(1 - 1/t)$ and $p_t$ with $p_j(t)$, and let
  $$
  B_{RV,\textrm{margins}} = \max_{j\le d} B_{RV,j}~;~
  B_{SRV,\textrm{margins}} = \max_{j\le d} B_{SRV,j}. 
  $$
  For $M\ge 2$, define the domain
  $$
  \Omega_2 (M,t)= \prod_{j=1}^d (-\infty, U_j(M n/k)] \bigcap \{x\in\rset^d: \| \tilde v(x)\| \ge t \},   
  $$
  and the marginal error
  
  \begin{multline*}
    h_{2}(\delta,M) =\\ \max\Bigg(h_1(\delta), 
     ~~ M
                \bigg(\frac{\Delta_{\alpha}}{e \min{\{\alpha_j,
                    \hat{\alpha}_j, j\le d\}}} +
                    \sqrt{\frac{12\log(2/\delta)}{k}} +
                    \frac{12\log(2/\delta)}{k} + \frac{4}{n} +
                B_{RV,\textrm{margins}}\bigg)   \Bigg).
  \end{multline*}
  If, in addition to the constraints in Statement 1., we have $ k\ge (12 \ln(2/\delta))^2$, $n\ge 13$,  and $h_2(\delta,M)\le 1/8$, 
  then 
  letting 
  $$\delta'(k) = \exp\bigg(-
  \frac{k}{13} ~ (1- 2/k^{1/4} - 2/k)^2  
 \bigg),$$ 
 we have with probability at least
  $1 - (d+2)\delta - d\delta'(k)$, 
$$
  \sup_{ x \in \Omega_2(M,t)}  \|\tilde v(x)\|~\Big| \hatcoDs{x} - \coDs{x}{P} \Big| \le \mathrm{Err}\Big(h_2(\delta,M), k,n,t, \delta\Big).  
  $$
  
\item Suppose further that the slowly varying functions $L_j$ in
  Statement 2. are bounded, and that  
  $$
\max_{j\le d}    \sup_{s\in[3/4,4]}{\left|\left(\frac{U_j(ns/k)}{U_j(n/k)}\right)^{\alpha_j}-s\right|}\leq
                \frac{1}{2}.
                $$
                Assume additionally  that for all $j\in\{1,\; \ldots,\; d\}$,  the estimator $\hat \alpha_j$, the chosen $k$ and $n$ are such that,  with probability $1-\delta$, 
                $$
                \overline u_{n,k,j} /\underline u_{n,k,j} \le 2, \textrm{~~~and~~~}
        \Delta_{\alpha,j}^{1/2}\leq \frac{4}{7}.            
$$
  Define the domain 
  $$
\Omega_3(t) = \bigcap_{j\le d} \Big( - \infty,  \exp{\left(\Delta_\alpha^{-1/2 }\right)}\underline{u}_{n,k,j} \Big]~\cap~\{\|\tilde v(x)\|\ge t\}, 
$$
and
  \begin{align*}
    h_{3}(\delta)
    &= \max\Bigg(h_1(\delta),   \frac{7}{4} \Delta_{\alpha}^{1/2} +
                4\sqrt{\frac{3\log(2/\delta)}{k}} +
                \frac{24\log(2/\delta)}{k} + \frac{8}{n} + 4B_{SRV,\textrm{margins}}   \Bigg)
  \end{align*}
  Then, if $h_3(\delta)<1/8$, we have with probability at least $ 1 - (2d+2)\delta - d  \delta'(k)$, 
$$
  \sup_{ x \in \Omega_3(t)}  \|\tilde v(x)\|~\Big| \hatcoDs{x} - \coDs{x}{P} \Big| \le \mathrm{Err}\Big(h_3(\delta), k,n,t,\delta\Big).  
  $$

\end{enumerate}

\end{corollary}

 \begin{remark}[on Corollary~\ref{cor:total_error_standardized}: regular variation assumption and out-of-domain generalization]
   The first statement in Corollary~\ref{cor:total_error_standardized} does not require regular variation of the marginal distributions, nor any assumption on the estimators $\hat \alpha_j$,  since the specific form of $\tilde{p}_{sp,j}(x)$ above threhsold $X_{(k/2,j)}$ is not involved in that result.
   Note that the domain $\Omega_1(t)$ is `observable', that is, for a given training set and a test point $x$, the statistician  may assert whether or not $x\in \Omega_1(t)$. Although $\Omega_2(M,t)$ is not observable because $U_j(Mn/k)$ is not observable, one may use instead an order statistic, say $X_{(k/3M,j)}$, as the probability that the latter exceeds $U_j(Mn/k)$ is negligible, following the computation of the probability $\delta'(k)$ in statement 2.
   The same  applies to $\Omega_3(t)$. 
The third statement provides an \emph{out-of-domain} error bound, that is an error bound on the estimated standardized polar depth for a new test point $x$ outside the `observed range', namely much larger than any $\lfloor k/M\rfloor ^{th}$ marginal order statistic for fixed $M\ge 2$. Indeed, the typical order of magnitude of  $\Delta_\alpha$ (say, for the Hill estimator) is $\Delta_\alpha \lessapprox C(\delta) k^{-1/2}$, thus the third statement provides a guarantee in probability  for $x$ such that $x_j$ is as large as approximately 
   $$
 \exp\{ C(\delta) k^{1/4} \} X_{(k,j)} \gg_{k\to\infty} M^\alpha X_{(k,j)} \approx X_{(k/M, j)}. 
   $$
Alternative statements using $\Delta_\alpha^{-\varepsilon}$ instead of $\Delta_\alpha^{-1/2}$ in the definition of $\Omega_3(t)$ are immediately obtained by the same argument.

 \end{remark}

\section{Numerical Experiments}\label{sec:experiments}
In this section, we study the accuracy of the proposed estimator from an empirical perspective and illustrate next its use in an anomaly detection task. We consider the following two bi-variate distributional settings, both supported on the non-negative halfspace in the second coordinate axis
$\{(x_1,x_2)\in\R^2: x_2\geq 0\}$, covering homogeneously heavy tails and heterogeneous angles.
\begin{itemize}
    \item \emph{Setting A}: the points are generated on the basis of a centered bi-variate Cauchy distribution $X=(X_1,X_2)$, with negative values for the second coordinate being positively reflected
    \begin{equation*}
        Y=(X_1, |X_2|).
    \end{equation*}
    \item \emph{Setting B}: the bi-variate data are generated as
    \begin{equation*}
        Y=\bigl( R\sin(\theta),R\cos(\theta) \bigr)\,,
    \end{equation*}
    with the probability densities of $R$ and $\theta$ being
    \begin{equation*}
        f_R(r) = 2r^{-3}\I_{[1,+\infty]}{(r)}\quad \text{ and } \quad
        f_\theta(t)
        =
        0.9\frac{2}{\pi}\mathbf{1}_{\left[\frac{\pi}{4},\,\frac{3\pi}{4}\right]}(t)
        +
        0.1\frac{1}{\pi}\mathbf{1}_{[0,\pi]}(t)\,,
    \end{equation*}
    respectively.
\end{itemize}

For each of the two settings, empirical evaluation of the accuracy of the estimation (Section~\ref{sec:experiments_estimation}) and of the application to anomaly detection (Section~\ref{sec:experiments_anomaly_detection}) have been performed.

\subsection{Empirical Evaluation of the Estimation Accuracy}\label{sec:experiments_estimation}
\begin{figure}[!t]
    \centering
    \begin{subfigure}{0.49\textwidth}
        \centering
        \includegraphics[width =
            0.90\textwidth,
        height=6cm]{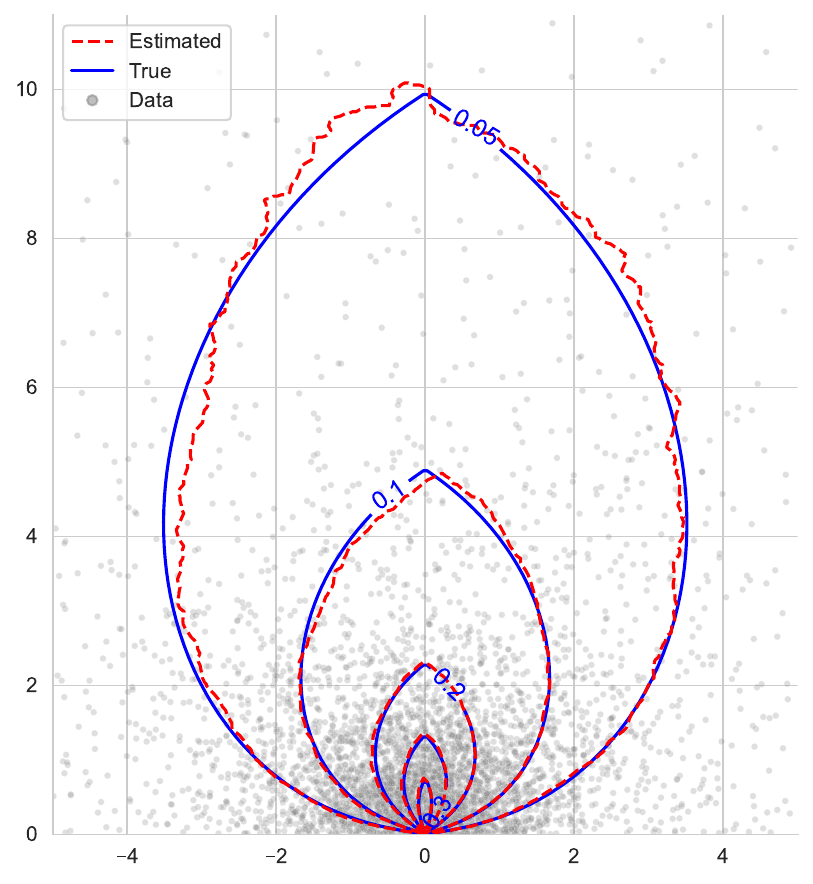}
        \caption{Setting A}
        \label{fig:accurracy_estimator_folded_centered_cauchy_points_in_the_bulk}
    \end{subfigure}
    \begin{subfigure}{0.49\textwidth}
        \centering
        \includegraphics[width =
            0.90\textwidth,
        height=6cm]{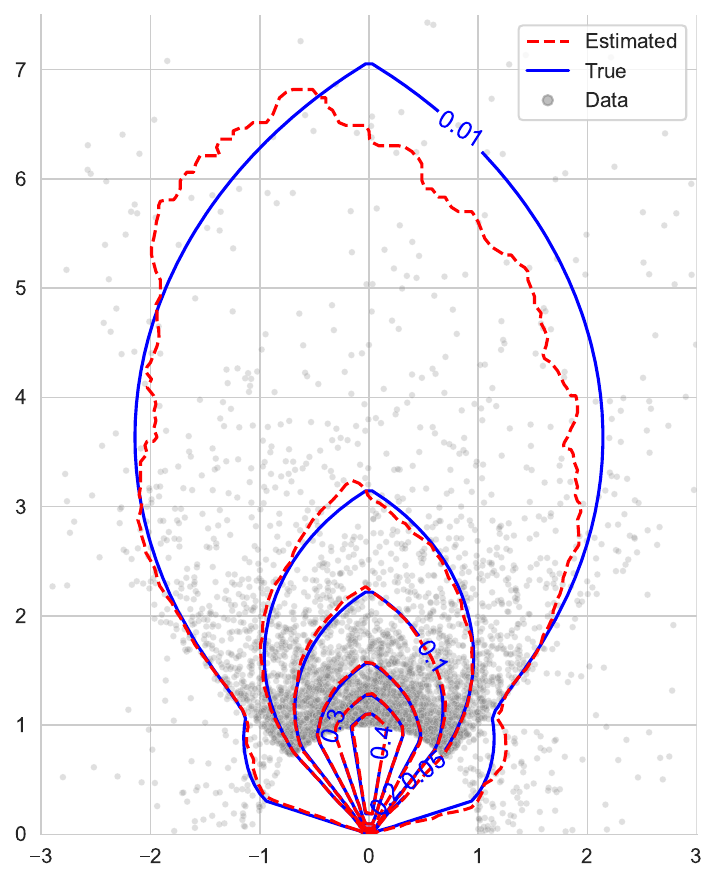}
        \caption{Setting B}
        \label{fig:accurracy_estimator_spade_shaped_points_in_the_bulk}
    \end{subfigure}
    \caption{Level sets for points in the bulk of the
    distribution.}
    \label{fig:accurracy_estimator_points_in_the_bulk_algorithm_1}
\end{figure}

\begin{figure}[!b]
    \centering
    \begin{subfigure}{0.49\textwidth}
        \centering
        \includegraphics[width =
            0.90\textwidth,
        height=6cm]{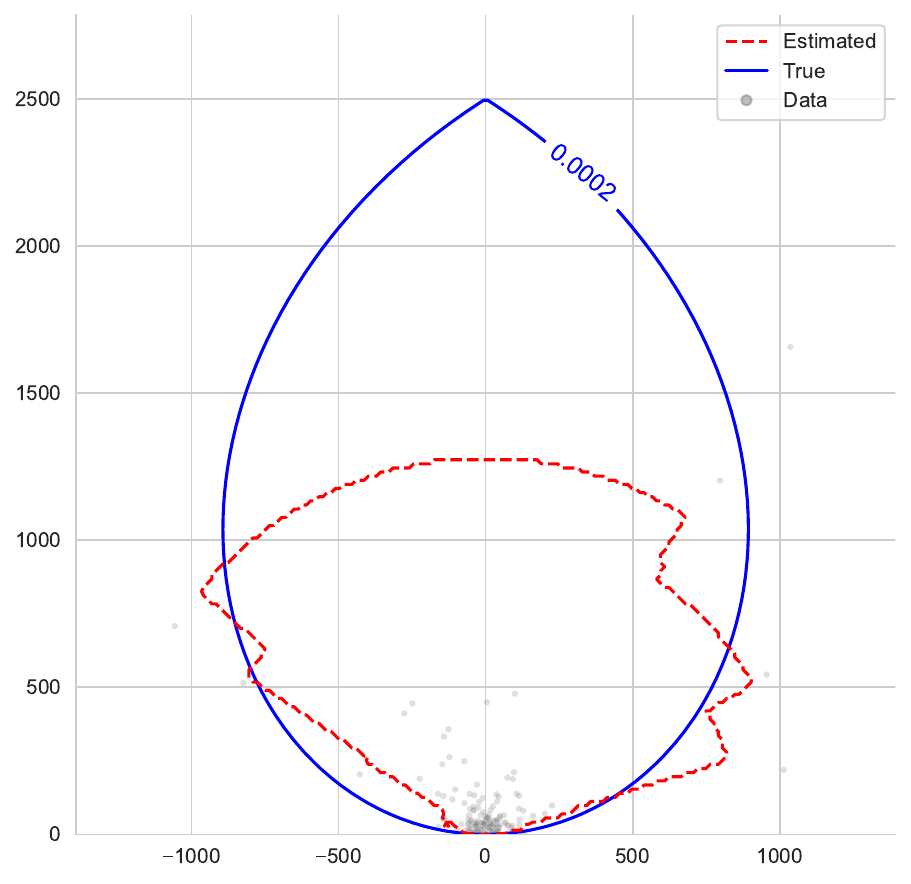}
        \caption{Setting A}
        \label{fig:accurracy_estimator_folded_centered_cauchy_points_in_the_tail}
    \end{subfigure}
    \begin{subfigure}{0.49\textwidth}
        \centering
        \includegraphics[width =
            0.90\textwidth,
        height=6cm]{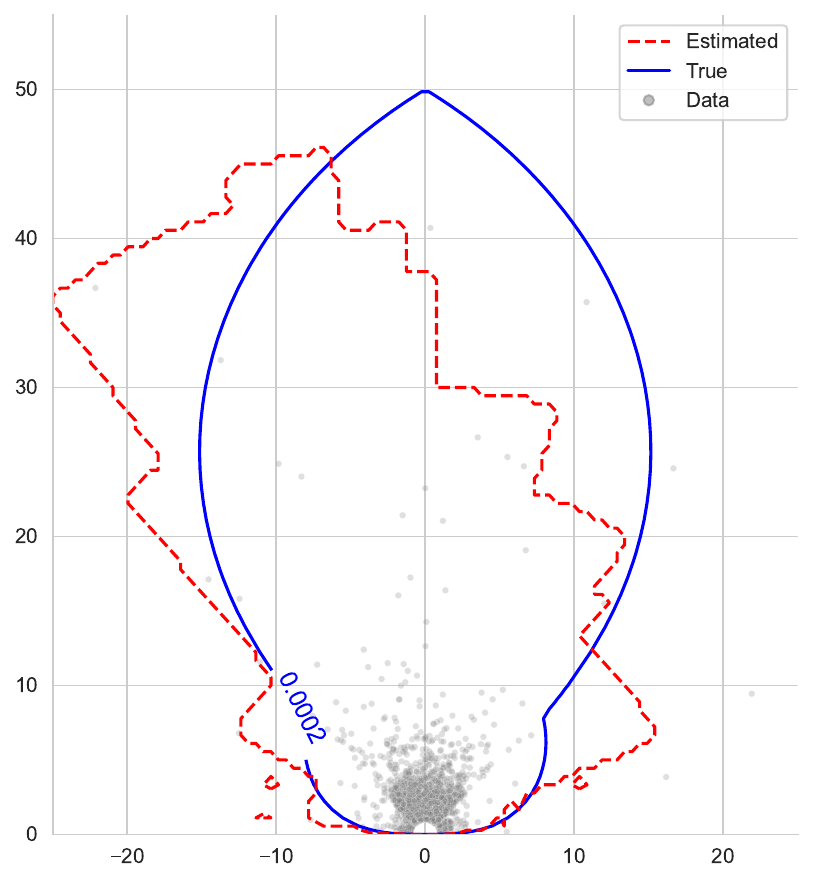}
        \caption{Setting B}
        \label{fig:accurracy_estimator_spade_shaped_points_in_the_tail}
    \end{subfigure}
    \caption{Level set at level \(1/n=0.0002\).}
    \label{fig:accurracy_estimator_points_in_the_tail_algorithm_1}
\end{figure}

Here we investigate the accuracy of Algorithm~\ref{alg:CO-depthEstimation} within both Setting A and Setting B, illustrate its precision in the bulk of the distribution, and motivate and justify the necessity and accuracy of Algorithm~\ref{alg:depthEstimation-standard} in the tail of the distribution. In order to to this, for each of the Setting A and Setting B, we generate a sample of $n=5 000$ observations.

\begin{figure}[!t]
    \centering

    \begin{subfigure}{0.24\textwidth}
        \centering
        \includegraphics[width=0.87\textwidth,height=6cm]{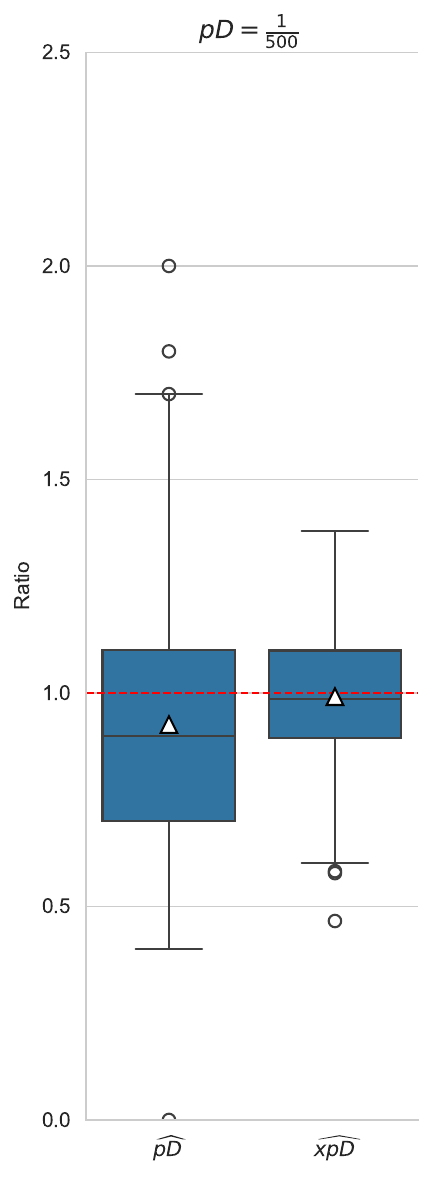}
    \end{subfigure}
    \begin{subfigure}{0.24\textwidth}
        \centering
        \includegraphics[width=0.87\textwidth,height=6cm]{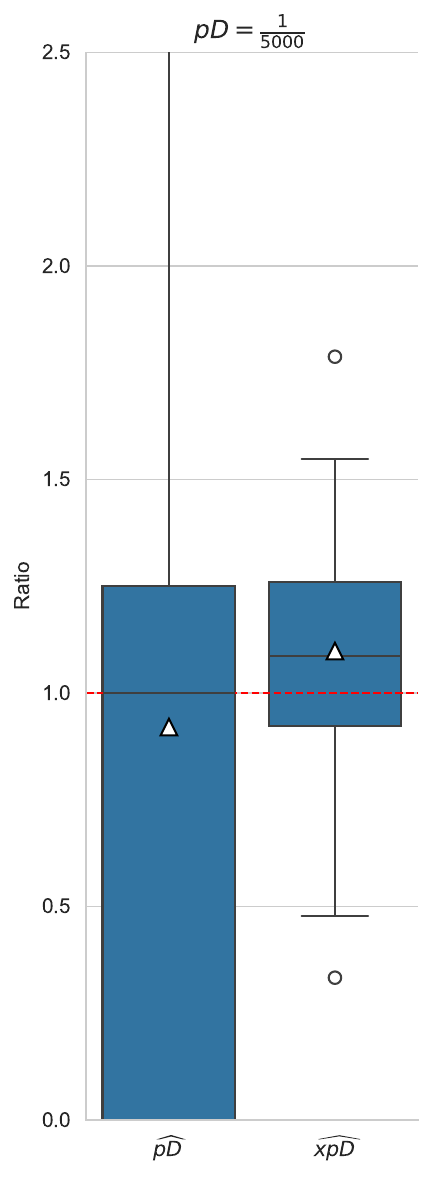}
    \end{subfigure}
    \begin{subfigure}{0.24\textwidth}
        \centering
        \includegraphics[width=0.87\textwidth,height=6cm]{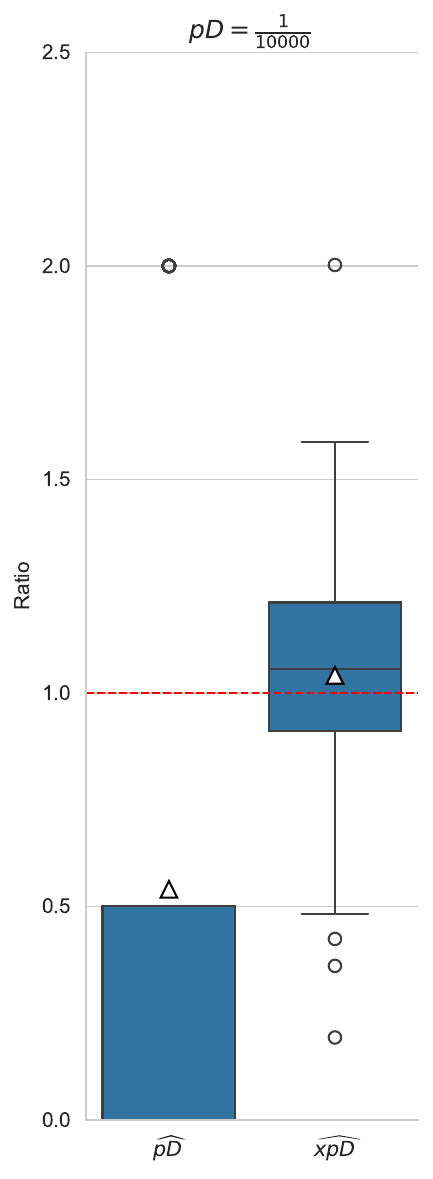}
    \end{subfigure}
    \begin{subfigure}{0.24\textwidth}
        \centering
        \includegraphics[width=0.87\textwidth,height=6cm]{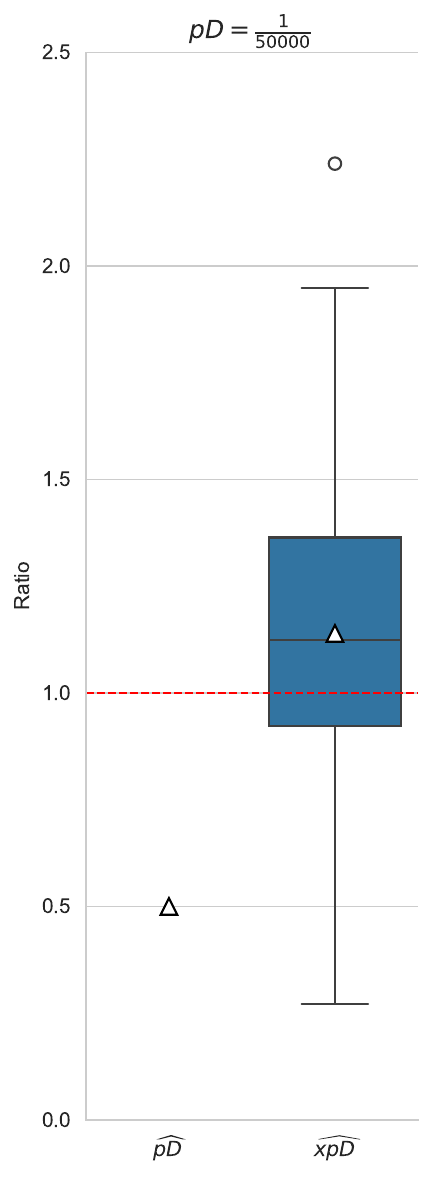}
    \end{subfigure}

    \begin{subfigure}{0.24\textwidth}
        \centering
        \includegraphics[width=0.87\textwidth,height=6cm]{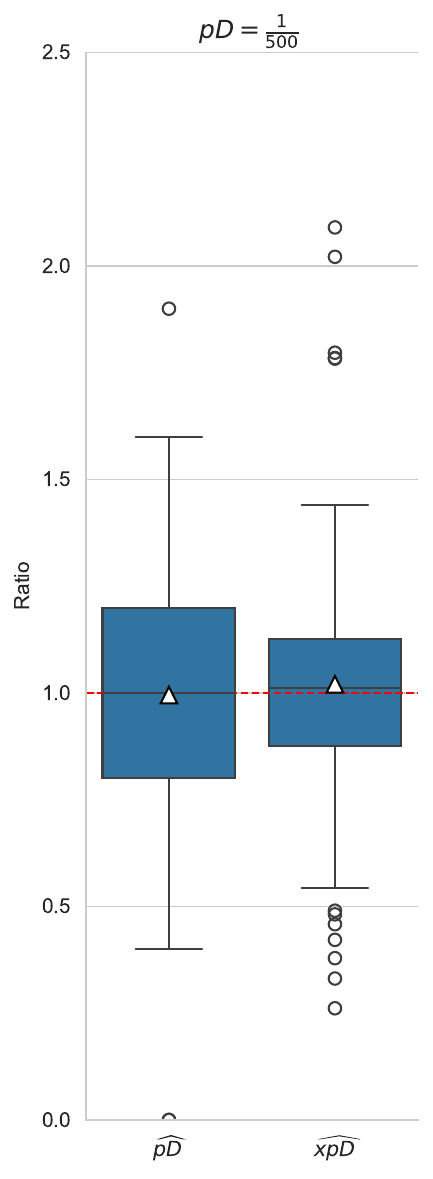}
    \end{subfigure}
    \begin{subfigure}{0.24\textwidth}
        \centering
        \includegraphics[width=0.87\textwidth,height=6cm]{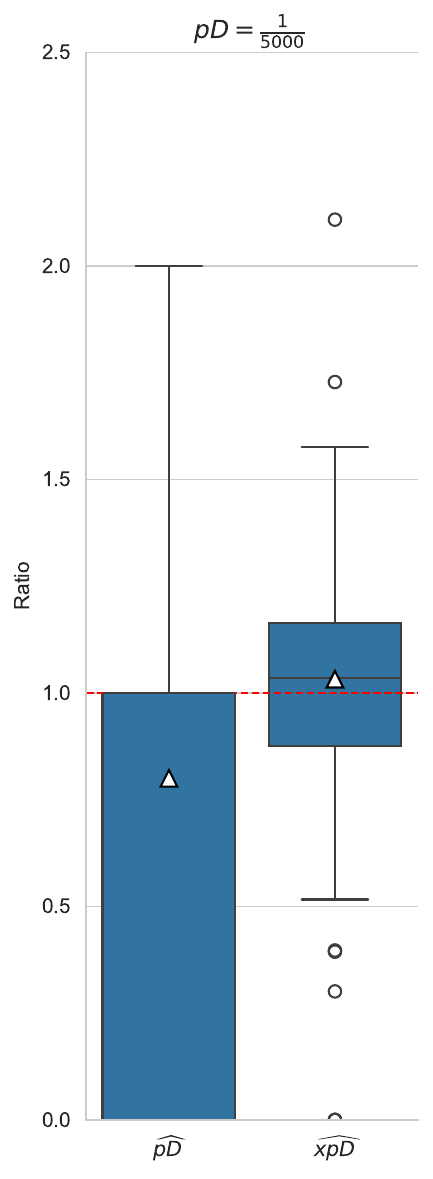}
    \end{subfigure}
    \begin{subfigure}{0.24\textwidth}
        \centering
        \includegraphics[width=0.87\textwidth,height=6cm]{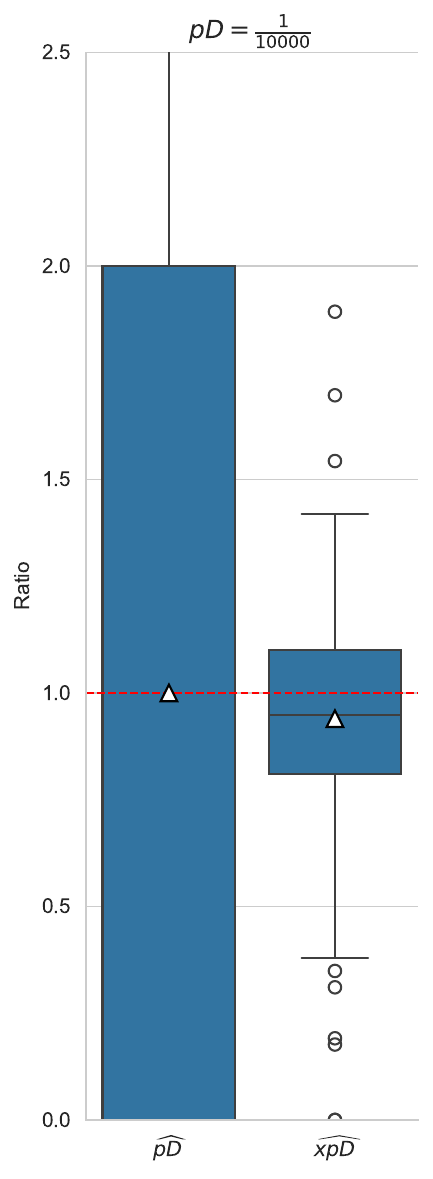}
    \end{subfigure}
    \begin{subfigure}{0.24\textwidth}
        \centering
        \includegraphics[width=0.87\textwidth,
        height=6cm]{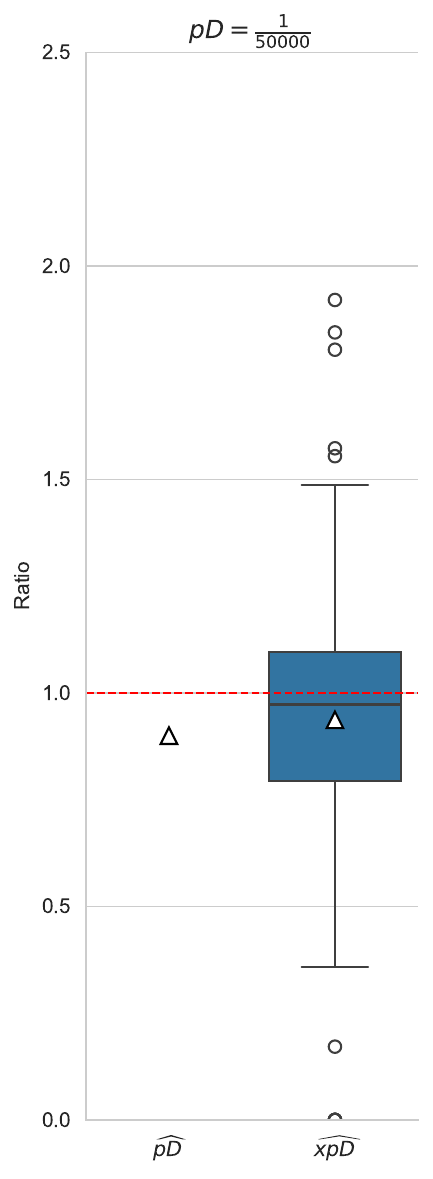}
    \end{subfigure}
    \caption{Boxplots of the ratio of the estimated and population depth, for different population level sets ($10/n$, $1/n$, $1/(2n)$, $1/(10n)$) in the tail of the distribution, over $100$ observations (having the selected population depth), for Algorithm~\ref{alg:CO-depthEstimation} (each left boxplot) and for Algorithm~\ref{alg:depthEstimation-standard} (each right boxplot). Top row presents the results for Setting~A, while bottom row for the Setting~B.}
    \label{fig:box_plots_algorithm_2}
\end{figure}

First, we study the behavior of Algorithm~\ref{alg:CO-depthEstimation} in the bulk of the distribution, \textit{i.e.}, for the population depth level sets of order $100/n$. In Figure~\ref{fig:accurracy_estimator_points_in_the_bulk_algorithm_1}, we compare contours of estimated and population depth level sets for depth levels $\pD(x,P) =$ 0.4, 0.3, 0.2, 0.1, 0.05, for the two settings considered. Visually, the estimator is quite accurate.

We further challenge the estimator from Algorithm~\ref{alg:CO-depthEstimation} by exploring its behavior in the tail of the distribution. More precisely, Figure~\ref{fig:accurracy_estimator_points_in_the_tail_algorithm_1} reveals its high imprecision for the population depth level set of order $\frac{1}{n}$ ($=0.0002$ in our case). In what follows, we shall illustrate efficacy of Algorithm~\ref{alg:depthEstimation-standard} in the tail of the distribution.

To validate the advantageous behavior of Algorithm~\ref{alg:depthEstimation-standard} in the tails of the distribution, we first choose several tail level sets with small values: $\pD(x,P)=\beta$ with $\beta = 10/n$, $1/n$, $1/(2n)$, $1/(10n)$. Then, for each of these level sets we generate $100$ observations with this value of their population depth, \textit{i.e.}, such that $\pD(x,P)=\beta$, and estimated $\pD$ for each of them using Algorithm~\ref{alg:depthEstimation-standard} w.r.t. the original data set (\textit{i.e.}, with $n=5 000$). Figure~\ref{fig:box_plots_algorithm_2} presents the boxplots of the ration between the true (population) depth and the estimated one. (We have also included the corresponding ratios for Algorithm~\ref{alg:CO-depthEstimation} in the comparison.). These results illustrates the accuracy of Algorithm~\ref{alg:depthEstimation-standard}, which is able to properly estimate the depth of observations that are located well outside the cloud of points of the data set, w.r.t. which the depth is computed. One further notes that the estimated by Algorithm~\ref{alg:depthEstimation-standard} ratio is almost unbiased for the Setting~B, which can be explained by the fact that the radial component of the bi-variate distribution is exactly Pareto (and hence the bias term \(B_{SRV}\) of
\eqref{eq:xcod_finite_sample_bound_ration_version} vanishes).

To select $k$ in Algorithm~\ref{alg:depthEstimation-standard},
we have generated $10$ additional samples and have picked
$k$ as the first part where the chart stabilises.

\subsection{Application to Anomaly Detection}\label{sec:experiments_anomaly_detection}
\begin{figure}[!t]
    \centering
    \begin{subfigure}{0.49\textwidth}
        \centering
        \includegraphics[width =
            0.90\textwidth,
        height=6cm]{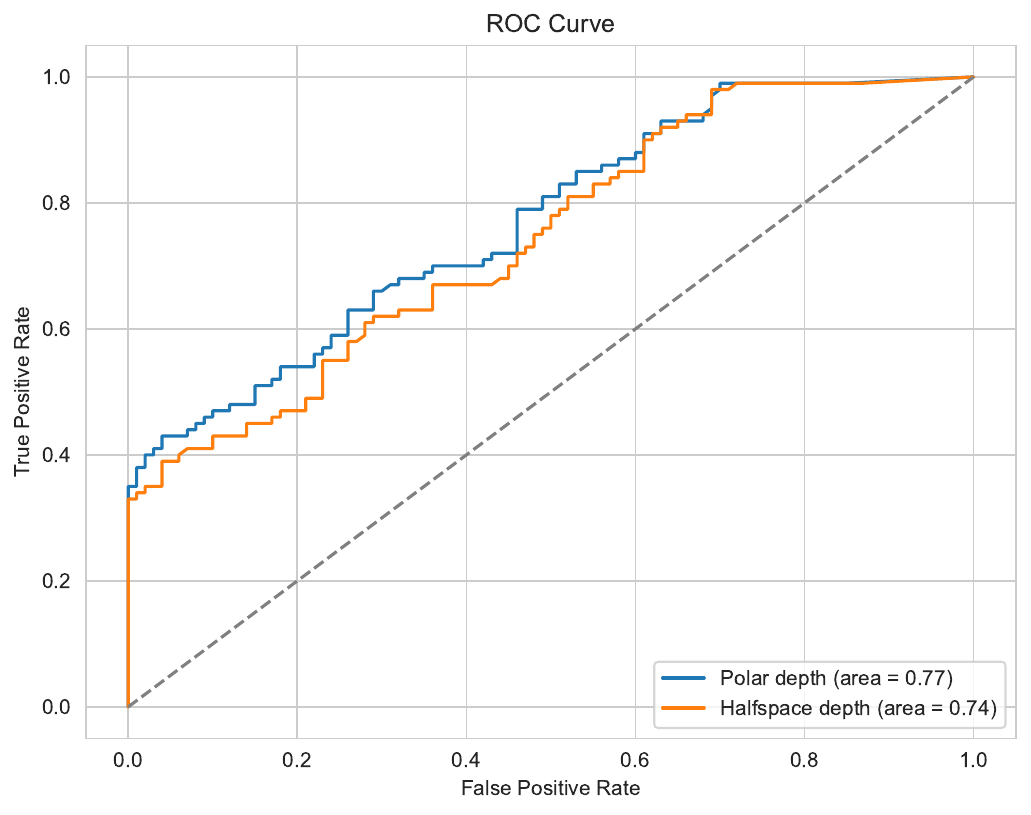}
        \caption{Setting~A}
        \label{fig:roc_curves_bulk_cauchy}
    \end{subfigure}
    \begin{subfigure}{0.49\textwidth}
        \centering
        \includegraphics[width =
            0.90\textwidth,
        height=6cm]{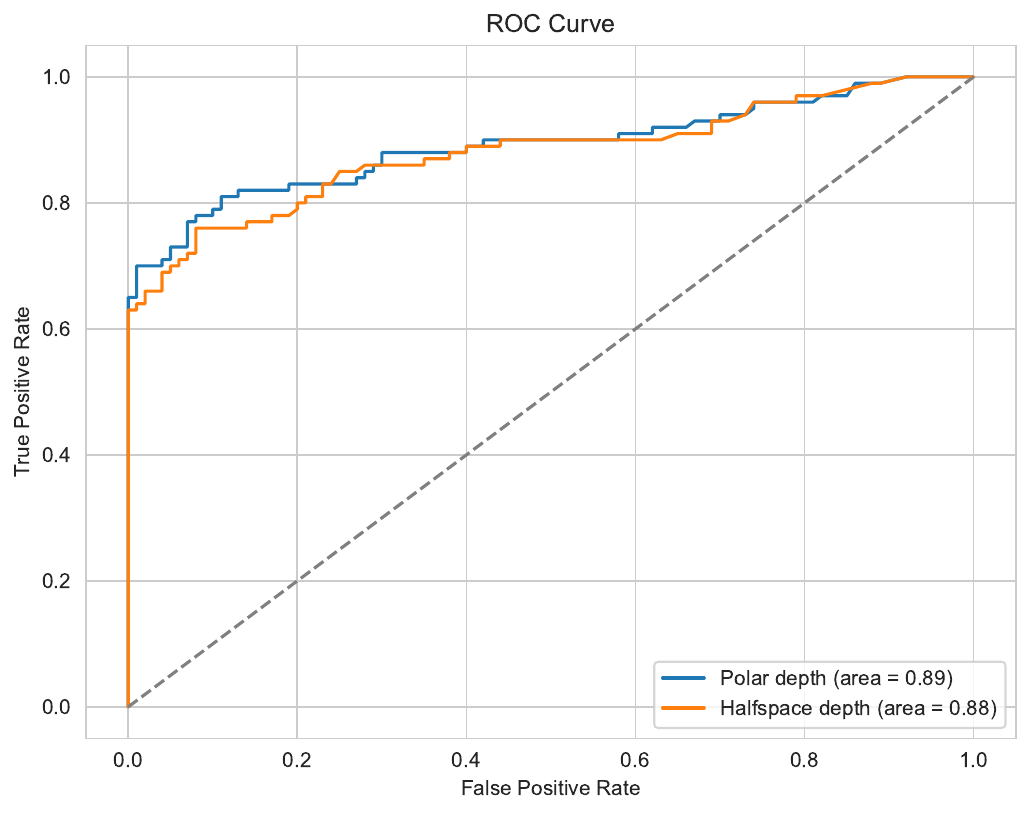}
        \caption{Setting~B}
        \label{fig:roc_curves_bulk_spade_shaped}
    \end{subfigure}
    \caption{ROC curves for the anomaly detection task in the bulk of the distribution, for a sample of $200$ points (containing $100$ abnormal observations).}
    \label{fig:roc_curves_bulk}
\end{figure}

We further evaluate the performance of the proposed depth for the purpose of anomaly detection. We first consider  the bulk of the distribution. From both Setting~A and Setting~B, we generate $n=100$ points, which constitute the normal observations. We then add $m=100$ anomalies. These abnormal observations are generated according to the following two scenarios:
\begin{itemize}
    \item \emph{Setting~A}: The radial component is generated from a super heavy tail distribution with $\alpha=1$ and the angular component is generated as $\bigl(\sin(\theta),cos(\theta)\bigr)$, where $\theta\sim\mathcal{U}\bigl([0,\pi]\bigr)$.
    \item \emph{Setting~B}: The radial component is generated from a super heavy tail distribution with $\alpha=2$, and the angular component follows the same distribution as in normal data.
\end{itemize}

All observations (normal and abnormal ones) are then ranked  based on $\pD$ (estimated using Algorithm~\ref{alg:CO-depthEstimation}), as well as using the traditional halfspace depth $D_{\mathrm{H}}$ (as defined in~\eqref{multivariate}). Based on this ordering, ROC curves---reflecting the anomaly detection performance--- are obtained, see Figure~\ref{fig:roc_curves_bulk}. The performance of both methods in the bulk of the distribution is similar, with $\pD$ proving slightly better (visually, on average).

Further, we explore the performance of anomaly detection in the tail of the distribution. For this, we generate normal data from both distributions of the Setting~A and Setting~B as $Y\,|\,r(Y) > r_1$ with $r_1=20$ and $r_1=5$, respectively. (This choice of $r_{1}$ ensures that $\P(r(Y)>r_{1})<0.05$ for both distributions.) Abnormal observations are then generated as follows:
\begin{itemize}
    \item \emph{Setting~A}: The radial component is generated from a heavy tail distribution with $\alpha=0.1$ and $x_{min}=20$, and the angular component is generated as $\bigl(\sin(\theta),cos(\theta)\bigr)$, where $\theta\sim\mathcal{U}\bigl([0,\pi]\bigr)$.
    \item \emph{Setting~B}: The radial component is generated from a heavy tail distribution with $\alpha=0.1$ and $x_{min}=5$, and the angular component follows the same distribution as in normal data.
\end{itemize}

Now, we used Algorithm~\ref{alg:depthEstimation-standard} to estimate $\pD$, while including in the comparison the halfspace depth as before, in its EVT-extrapolated version as described in~\cite{einmahl2015bridging}. As can be seen in Figure~\ref{fig:roc_curves_tail} (which plots the corresponding anomaly detection ROC curves), the proposed polar depth clearly proves advantageous performance (for the vast majority of the thresholds for the false positive rate).

\begin{figure}[!t]
    \centering
    \begin{subfigure}{0.49\textwidth}
        \centering
        \includegraphics[width =
            0.90\textwidth,
        height=6cm]{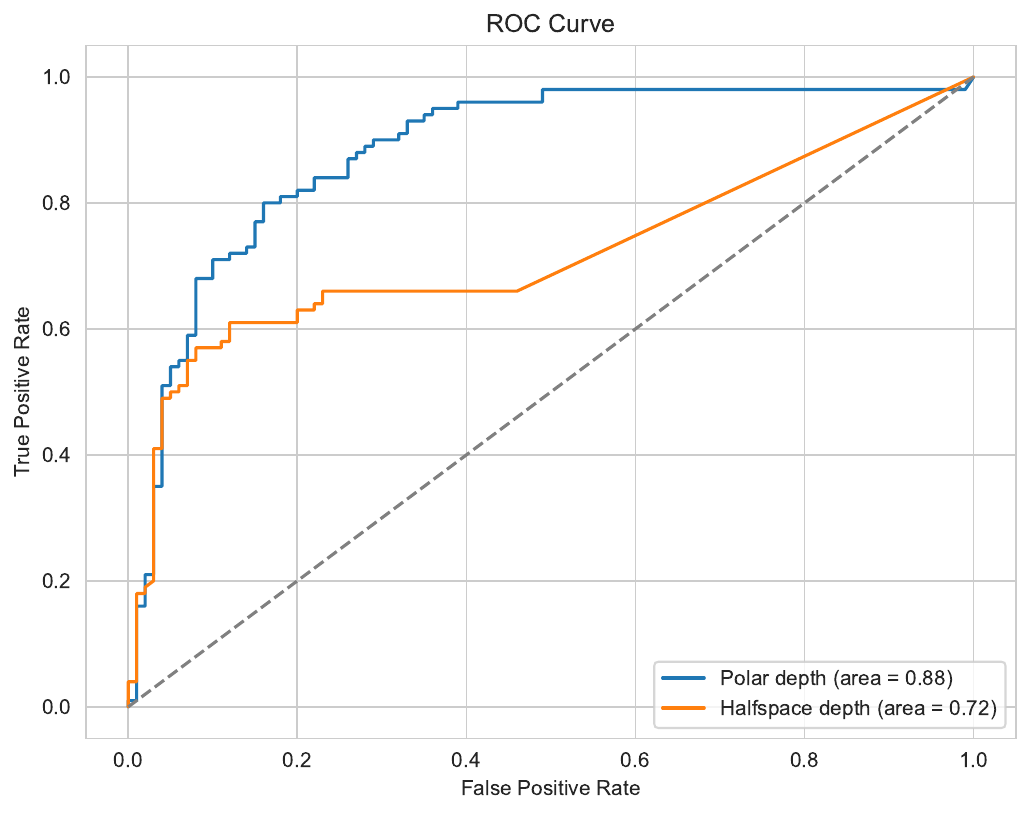}
        \caption{Setting~A}
        \label{fig:roc_curves_tail_cauchy}
    \end{subfigure}
    \begin{subfigure}{0.49\textwidth}
        \centering
        \includegraphics[width =
            0.90\textwidth,
        height=6cm]{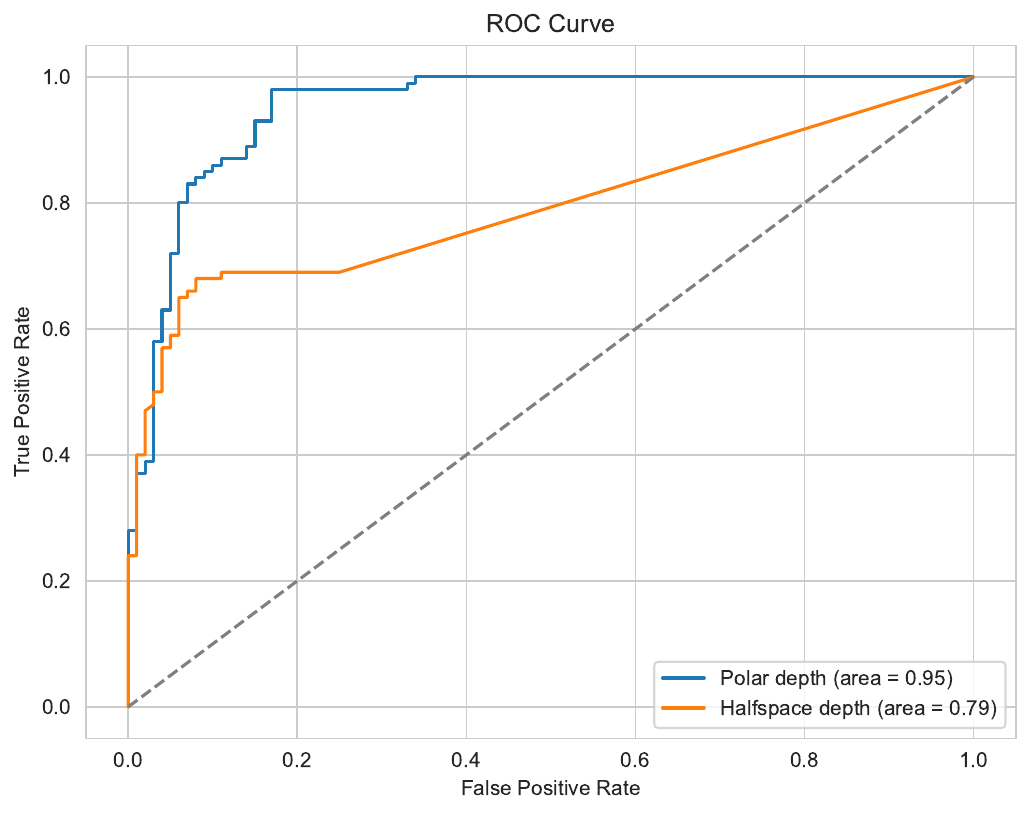}
        \caption{Setting~B}
        \label{fig:roc_curves_tail_spade_shaped}
    \end{subfigure}
    \caption{ROC curves for the anomaly detection task in the tail of the distribution, for a sample of $200$ points (containing $100$ abnormal observations).}
    \label{fig:roc_curves_tail}
\end{figure}

\section{Conclusion}\label{sec:conclusion}

The ultimate goal of this paper was to define a concept of statistical depth applicable in practice to the extreme values of a multivariate regularly varying distribution. The essential property of homogeneity that characterizes their distribution, which is fundamental to statistical extrapolation in extreme value analysis, led us to propose a new notion of depth, which we have named 'polar depth' and which, by construction, inherits the homogeneity property. As we have demonstrated theoretically and illustrated numerically, this definition, which allows us to view the (appropriately rescaled) values taken by the depth of the extreme limit distribution as the limits of the depth of extreme points, offers considerable advantages in terms of the guarantees that can be provided regarding statistical estimation, in addition to its interpretability. Beyond its initial motivation, polar depth has intrinsic value, as we have shown: not only does it satisfy many desirable axiomatic, statistical, and numerical properties, but it also resolves the problems encountered by Tukey’s depth when applied to distributions whose support is contained in a half-space (\textit{i.e.} the occurrence of points with low depth even though they are located near the center of mass). We therefore hope that this novel concept will find many statistical applications in the future, similar to its use for anomaly detection (in extreme regions) discussed in this article.

\bibliographystyle{abbrv}
\bibliography{biblio.bib}

@article{billingsley1967uniformity,
  title={Uniformity in weak convergence},
  author={Billingsley, Patrick and Tops{\o}e, Flemming},
  journal={Zeitschrift f{\"u}r Wahrscheinlichkeitstheorie und verwandte Gebiete},
  volume={7},
  number={1},
  pages={1--16},
  year={1967},
  publisher={Springer}
}

@book{coles2001introduction,
  title={An introduction to statistical modeling of extreme values},
  author={Coles, Stuart and Bawa, Joanna and Trenner, Lesley and Dorazio, Pat},
  volume={208},
  year={2001},
  publisher={Springer}
}

@article{clemenccon2025regression,
  title={On regression in extreme regions},
  author={Cl{\'e}men{\c{c}}on, Stephan and Huet, Nathan and Sabourin, Anne},
  journal={Electronic Journal of Statistics},
  volume={19},
  number={2},
  pages={4784--4828},
  year={2025},
  publisher={The Institute of Mathematical Statistics and the Bernoulli Society}
}

@article{cai2011estimation,
  title={Estimation of extreme risk regions under multivariate regular variation},
  author={Cai, Juan-Juan and Einmahl, John HJ and De Haan, Laurens},
  year={2011}
}

@article{boucheron2005theory,
  title={Theory of classification: A survey of some recent advances},
  author={Boucheron, St{\'e}phane and Bousquet, Olivier and Lugosi, G{\'a}bor},
  journal={ESAIM: probability and statistics},
  volume={9},
  pages={323--375},
  year={2005},
  publisher={EDP Sciences}
}

@book{reiss2012approximate,
  title     = {Approximate distributions of order statistics: with applications to nonparametric statistics},
  author    = {Reiss, Rolf-Dieter},
  year      = {2012},
  publisher = {Springer science \& business media}
}

@article{boucheron2015tail,
  title  = {Tail index estimation, concentration and adaptivity},
  author = {Boucheron, St{\'e}phane and Thomas, Maud},
  journal = {Electronic Journal of Statistics},
  pages= {2751 - 2792},
  doig= {https://doi.org/10.1214/15-EJS1088},
  year   = {2015}
}

@article{dyckerhoff2016exact,
  title     = {Exact computation of the halfspace depth},
  author    = {Dyckerhoff, Rainer and Mozharovskyi, Pavlo},
  journal   = {Computational Statistics \& Data Analysis},
  volume    = {98},
  pages     = {19--30},
  year      = {2016},
  publisher = {Elsevier}
}

@article{lhaut2022uniform,
  title     = {Uniform concentration bounds for frequencies of rare events},
  author    = {Lhaut, St{\'e}phane and Sabourin, Anne and Segers, Johan},
  journal   = {Statistics \& Probability Letters},
  volume    = {189},
  pages     = {109610},
  year      = {2022},
  publisher = {Elsevier}
}

@article{he2017estimation,
  title     = {Estimation of extreme depth-based quantile regions},
  author    = {He, Yi and Einmahl, John HJ},
  journal   = {Journal of the Royal Statistical Society Series B: Statistical Methodology},
  volume    = {79},
  number    = {2},
  pages     = {449--461},
  year      = {2017},
  publisher = {Oxford University Press}
}

@article{einmahl2015bridging,
  title     = {Bridging centrality and extremity: Refining empirical data depth using extreme value statistics},
  author    = {Einmahl, John HJ and Li, Jun and Liu, Regina Y},
  journal   = {Annals of Statistics},
  volume    = {43},
  number    = {6},
  pages     = {2738--2765},
  year      = {2015},
  publisher = {Institute of Mathematical Statistics}
}

@article{ClemenconJalalzaiSabourinSegers2023,
  author     = {St{\'e}phan Cl{\'e}men{\c{c}}on and Hamid Jalalzai and St{\'e}phane Lhaut and Anne Sabourin and Johan Segers},
  doi        = {10.3150/22-BEJ1562},
  journal    = {Bernoulli},
  number     = {4},
  pages      = {2797 -- 2827},
  publisher  = {Bernoulli Society for Mathematical Statistics and Probability},
  title      = {{Concentration bounds for the empirical angular measure with statistical learning applications}},
  url        = {https://doi.org/10.3150/22-BEJ1562},
  volume     = {29},
  year       = {2023}
}

@article{NagyDemniButtarazziPorzio2023,
  author     = {Nagy, Stanislav and Demni, Houyem and Buttarazzi, Davide and Porzio, Giovanni C.},
  doi        = {10.1007/s11634-023-00557-3},
  isbn       = {1862-5355},
  journal    = {Advances in Data Analysis and Classification},
  title      = {Theory of angular depth for classification of directional data},
  url        = {https://doi.org/10.1007/s11634-023-00557-3},
  year       = {2023}
}

@article{einmahl2001nonparametric,
  title={Nonparametric estimation of the spectral measure of an extreme value distribution},
  author={Einmahl, John H. J. and de Haan, Laurens and Piterbarg, Vladimir I.},
  journal={The Annals of Statistics},
  pages={1401--1423},
  year={2001},
  volume={29},
  number={5}
}

@article{einmahl2009maximum,
	AUTHOR = {Einmahl, John H. J. and Segers, Johan},
	TITLE = {Maximum empirical likelihood estimation of the spectral measure of an extreme-value distribution},
	JOURNAL = {The Annals of Statistics},
	VOLUME = {37},
	YEAR = {2009},
	NUMBER = {5B},
	PAGES = {2953--2989},
}

@article{nagy2024theoretical,
  title={Theoretical properties of angular halfspace depth},
  author={Nagy, Stanislav and Laketa, Petra},
  journal={Bernoulli},
  volume={31},
  number={2},
  pages={1007--1031},
  year={2025},
  publisher={Bernoulli Society for Mathematical Statistics and Probability}
}

@book{resnick2007heavy,
  title     = {Heavy-tail phenomena: probabilistic and statistical modeling},
  author    = {Resnick, Sidney I},
  year      = {2007},
  publisher = {Springer Science \& Business Media}
}

@book{resnick2008extreme,
  title     = {Extreme values, regular variation, and point processes},
  author    = {Resnick, Sidney I},
  volume    = {4},
  year      = {2008},
  publisher = {Springer Science \& Business Media}
}

@book{de2007extreme,
  author    = {De Haan, Laurens and Ferreira, Ana},
  publisher = {Springer Science \& Business Media},
  title     = {Extreme value theory: an introduction},
  year      = {2007}
}

@article{hult2006regular,
  title   = {Regular variation for measures on metric spaces},
  author  = {Hult, Henrik and Lindskog, Filip},
  journal = {Publications de l'Institut Math{\'e}matique},
  volume  = {80},
  number  = {94},
  pages   = {121--140},
  year    = {2006}
}

@book{beirlant2006statistics,
  title     = {Statistics of extremes: theory and applications},
  author    = {Beirlant, Jan and Goegebeur, Yuri and Segers, Johan and Teugels, Jozef L},
  year      = {2006},
  publisher = {John Wiley \& Sons}
}

@article{NagyGijbels2016,
  author  = {Nagy, Stanislav and Gijbels, Irène and Omelka, Marek and Hlubinka, Daniel},
  title   = {Integrated depth for functional data: statistical properties
             and consistency},
  doi     = {10.1051/ps/2016005},
  journal = {ESAIM: PS},
  year    = 2016,
  volume  = 20,
  pages   = {95-130}
}

@article{ZuoS00b,
  author  = {Zuo, Y. and Serfling, R.},
  journal = {The Annals of Statistics},
  number  = {2},
  pages   = {483--499},
  title   = {Structural properties and convergence results for contours of  sample statistical depth functions},
  volume  = {28},
  year    = {2000}
}

@article{10.1214/23-EJS2189,
  author    = {Stephan Cl{\'e}men{\c{c}}on and Pavlo Mozharovskyi and Guillaume Staerman},
  title     = {{Affine invariant integrated rank-weighted statistical depth: properties and finite sample analysis}},
  volume    = {17},
  journal   = {Electronic Journal of Statistics},
  number    = {2},
  publisher = {Institute of Mathematical Statistics and Bernoulli Society},
  pages     = {3854 -- 3892},
  year      = {2023},
  doi       = {10.1214/23-EJS2189},
  url       = {https://doi.org/10.1214/23-EJS2189}
}

@article{BurrF17,
  author  = {M.~A. Burr and R.~J. Fabrizio},
  title   = {Uniform convergence rates for halfspace depth},
  journal = {Statistics and Probability Letters},
  volume  = {124},
  pages   = {33--40},
  year    = {2017}
}

@article{virta2023spatial,
  title   = {Spatial depth for data in metric spaces},
  author  = {Virta, Joni},
  journal = {arXiv preprint arXiv:2306.09740},
  year    = {2023}
}

@inproceedings{StaermanMozharovskyiClemencon2020,
  title        = {The area of the convex hull of sampled curves: a robust functional statistical depth measure},
  author       = {Staerman, Guillaume and Mozharovskyi, Pavlo and Cl{\'e}mencon, St{\'e}phan},
  booktitle    = {International Conference on Artificial Intelligence and Statistics},
  pages        = {570--579},
  year         = {2020},
  organization = {PMLR}
}

@inproceedings{FernandezClemencon2025,
  title     = {Anomaly Detection based on {M}arkov Data: \\A Statistical Depth Approach},
  author    = {Fernandez, Carlos and Cl{\'e}mencon, St{\'e}phan},
  booktitle = {Proceedings of the LOD Conference},
  year      = {2025}
}

@misc{lederer2025adaptivetailindexestimation,
      title={Adaptive tail index estimation: minimal assumptions and non-asymptotic guarantees}, 
      author={Johannes Lederer and Anne Sabourin and Mahsa Taheri},
      year={2025},
      eprint={2505.22371},
      archivePrefix={arXiv},
      primaryClass={stat.OT},
      url={https://arxiv.org/abs/2505.22371}, 
}

@book{devroye2013probabilistic,
  title     = {A probabilistic theory of pattern recognition},
  author    = {Devroye, Luc and Gy{\"o}rfi, L{\'a}szl{\'o} and Lugosi, G{\'a}bor},
  volume    = {31},
  year      = {2013},
  publisher = {Springer Science \& Business Media}
}

@incollection{lugosi2002pattern,
  title     = {Pattern classification and learning theory},
  author    = {Lugosi, G{\'a}bor},
  booktitle = {Principles of nonparametric learning},
  pages     = {1--56},
  year      = {2002},
  publisher = {Springer}
}

@book{Bingham1987,
  author    = {N. H. Bingham and C. M. Goldie and J. L. Teugels},
  publisher = {Cambridge University Press},
  title     = {Regular variation},
  year      = {1987},
  isbn      = {0521307872,9780521307871},
  series    = {Encyclopedia of mathematics and its applications 27}
}

@article{dyckerhoff2020approximate,
  author  = {Dyckerhoff, Rainer and Mozharovskyi, Pavlo and Nagy, Stanislav},
  title   = {{Approximate computation of projection depths}},
  journal = {Computational Statistics \& Data Analysis},
  year    = 2021,
  volume  = {157},
  number  = {C}
}

@phdthesis{Donoho82,
  title  = {Breakdown Properties of Multivariate Location Estimators},
  author = {Donoho, David L.},
  school = {Harvard University},
  year   = {1982}
}

@article{DonohoG92,
  title     = {Breakdown properties of location estimates based on halfspace depth and projected outlyingness},
  author    = {Donoho, David L. and Gasko, Miriam},
  journal   = {The Annals of Statistics},
  volume    = {20},
  pages     = {1803--1827},
  year      = {1992},
  publisher = {JSTOR}
}

@article{Dyckerhoff04,
  author  = {Dyckerhoff, Rainer},
  title   = {Data depths satisfying the projection property},
  journal = {AStA - Advances in Statistical Analysis},
  year    = {2004},
  volume  = {88},
  number  = {2},
  pages   = {163--190}
}

@article{chernozhukov2017,
  author  = {Chernozhukov, Victor and Galichon, Alfred and Hallin, Marc and Henry, Marc},
  journal = {The Annals of Statistics},
  month   = {02},
  number  = {1},
  pages   = {223--256},
  title   = {Monge–Kantorovich depth, quantiles, ranks and signs},
  volume  = {45},
  year    = {2017}
}

@article{Chaudhuri,
  author  = { Probal   Chaudhuri },
  title   = {On a Geometric Notion of Quantiles for Multivariate Data},
  journal = {Journal of the American Statistical Association},
  volume  = {91},
  number  = {434},
  pages   = {862-872},
  year    = {1996}
}

@article{OJA1983,
  title   = {Descriptive statistics for multivariate distributions},
  journal = {Statistics \& Probability Letters},
  volume  = {1},
  number  = {6},
  pages   = {327 - 332},
  year    = {1983},
  author  = {Hannu Oja}
}

@article{Vardi1423,
  author  = {Vardi, Yehuda and Zhang, Cun-Hui},
  title   = {The multivariate L1-median and associated data depth},
  volume  = {97},
  number  = {4},
  pages   = {1423--1426},
  year    = {2000},
  journal = {Proceedings of the National Academy of Sciences}
}

@article{LiuSingh,
  author  = {Regina Y. Liu and Kesar Singh},
  journal = {Journal of the American Statistical Association},
  number  = {421},
  pages   = {252--260},
  title   = {A Quality Index Based on Data Depth and Multivariate Rank Tests},
  volume  = {88},
  year    = {1993}
}

@inbook{Liu92,
  author    = {Regina Y. Liu},
  booktitle = {$L_1$-statistical analysis and related methods},
  pages     = {279–294},
  publisher = {North-Holland, Amsterdam},
  title     = {Data Depth and Multivariate Rank Tests},
  year      = {1992}
}

@article{Liu,
  author        = {Regina Y. Liu},
  title         = {On a notion of data depth based upon random simplices},
  journal       = {The Annals of Statistics},
  volume        = {},
  year          = {1990},
  url           = {},
  archiveprefix = {},
  eprint        = {},
  bibsource     = {}
}

@article{koshevoy1997,
  author  = {Koshevoy, Gleb and Mosler, Karl},
  journal = {The Annals of Statistics},
  month   = {10},
  number  = {5},
  pages   = {1998--2017},
  title   = {Zonoid trimming for multivariate distributions},
  volume  = {25},
  year    = {1997}
}

@article{Koshevoy02,
  title     = {{The Tukey depth characterizes the atomic measure}},
  author    = {Koshevoy, Gleb A.},
  journal   = {Journal of Multivariate Analysis},
  volume    = {83},
  pages     = {360--364},
  year      = {2002},
  publisher = {Elsevier}
}

@incollection{Mosler13,
  year      = {2013},
  booktitle = {Robustness and Complex Data Structures: Festschrift in Honour of Ursula Gather},
  editor    = {Becker, C. and Fried, R. and Kuhnt, S.},
  title     = {Depth Statistics},
  publisher = {Springer},
  author    = {Mosler, Karl.},
  pages     = {17--34}
}

@article{MozharovskyiML15,
  title     = {{Classifying real-world data with the $DD\alpha$-procedure}},
  author    = {Mozharovskyi, P. and Mosler, K. and Lange, T.},
  year      = {2015},
  journal   = {Advances in Data Analysis and Classification},
  volume    = {9},
  publisher = {Springer. Berlin},
  pages     = {287--314}
}

@article{RousseeuwS98,
  author  = {Rousseeuw, Peter J.
             and Struyf, Anja},
  title   = {Computing location depth and regression depth in higher dimensions},
  journal = {Statistics and Computing},
  year    = {1998},
  volume  = {8},
  number  = {3},
  pages   = {193--203}
}

@article{StruyfR99,
  title     = {Halfspace depth and regression depth characterize the empirical distribution},
  author    = {Struyf, A.J. and Rousseeuw, P.J.},
  journal   = {Journal of Multivariate Analysis},
  volume    = {69},
  pages     = {135--153},
  year      = {1999},
  publisher = {Elsevier}
}

@inproceedings{Tukey75,
  title     = {Mathematics and the picturing of data},
  author    = {Tukey, John W.},
  editor    = {James, R.D.},
  booktitle = {Proceedings of the International Congress of Mathematicians},
  volume    = {2},
  pages     = {523--531},
  year      = {1975}
}

@article{ZuoS00a,
  title     = {General notions of statistical depth function},
  author    = {Zuo, Yijun and Serfling, Robert},
  journal   = {The Annals of Statistics},
  pages     = {461--482},
  publisher = {Institute of Mathematical Statistics},
  volume    = {28},
  number    = {2},
  year      = {2000}
}

@article{LiuPS99,
  title     = {Multivariate analysis by data depth: descriptive statistics, graphics and inference,(with discussion and a rejoinder by Liu and Singh)},
  author    = {Liu, R.Y. and Parelius, J.M. and Singh, K.},
  journal   = {The Annals of Statistics},
  volume    = {27},
  number    = {3},
  pages     = {783--858},
  year      = {1999},
  publisher = {Institute of Mathematical Statistics}
}

@article{mosler2022choosing,
  title     = {Choosing among notions of multivariate depth statistics},
  author    = {Mosler, Karl and Mozharovskyi, Pavlo},
  journal   = {Statistical Science},
  volume    = {37},
  number    = {3},
  pages     = {348--368},
  year      = {2022},
  publisher = {Institute of Mathematical Statistics}
}

@article{nagyuniformrates,
  author    = {Stanislav Nagy and Rainer Dyckerhoff and Pavlo Mozharovskyi},
  title     = {{Uniform convergence rates for the approximated halfspace and projection depth}},
  volume    = {14},
  journal   = {Electronic Journal of Statistics},
  number    = {2},
  publisher = {Institute of Mathematical Statistics and Bernoulli Society},
  pages     = {3939 -- 3975},
  year      = {2020}
}

@article{LiuZ14a,
  issn     = {},
  url      = {},
  author   = {Liu, Xiaohui and Zuo, Yijun},
  journal  = {Communications in Statistics - Simulation and Computation},
  number   = {},
  pages    = {},
  title    = {Computing halfspace depth and regression depth.},
  volume   = {},
  year     = {2014}
}

@article{LiuMM18,
  author  = {Xiaohui Liu and Karl Mosler and Pavlo Mozharovskyi},
  title   = {Fast computation of Tukey trimmed regions and median in dimension $p > 2$},
  journal = {Journal of Computational and Graphical Statistics},
  year    = {2018},
  note    = {in press}
}

@article{DyckerhoffLP15,
  author  = {Rainer Dyckerhoff and Christoph Ley and Davy Paindaveine},
  journal = {Annals of the Institute of Statistical Mathematics},
  pages   = {917--941},
  title   = {Depth-based runs test for bivariate central symmetry},
  volume  = {67},
  year    = {2015}
}

@article{SatermanAMHSGC23,
  author  = {Guillaume Staerman and Eric Adjakossa and Pavlo Mozharovskyi and Vera Hofer and Jayant Sen Gupta and Stephan Cl\'{e}men\c{c}on},
  journal = {International Journal of Data Science and Analytics},
  number  = {16},
  pages   = {101--117},
  title   = {Functional anomaly detection: a benchmark study},
  volume  = {16},
  year    = {2023}
}

@article{Carpentier2015,
  author    = {Carpentier, Alexandra and Kim, Arlene K. H.},
  journal   = {Statistica Sinica},
  title     = {Adaptive and Minimax Optimal Estimation of the Tail Coefficient},
  year      = {2015},
  issn      = {10170405, 19968507},
  number    = {3},
  pages     = {1133--1144},
  volume    = {25},
  publisher = {Institute of Statistical Science, Academia Sinica}
}

@book{bullen1998dictionary,
  title     = {Dictionary of inequalities},
  author    = {Bullen, Peter},
  year      = {1998},
  edition   = {1},
  publisher = {CRC Press}
}

@article{bertail2025tail,
  author    = {Fern{\'a}ndez, Carlos A. and Bertail, Patrice and Cl{\'e}men{\c c}on, Stephan},
  title     = {Tail index estimation for discrete heavy-tailed distributions with application to statistical inference for regular {M}arkov chains},
  journal   = {TEST},
  year      = {2025},
  issn      = {1863-8260},
  doi       = {10.1007/s11749-025-00975-9},
  url       = {https://doi.org/10.1007/s11749-025-00975-9},
  pages     = {1--23},
  publisher = {Springer}
}

@article{LiCAL12,
  title={{DD-classifier: Nonparametric classification procedure based on DD-plot}},
  author={Li, Jun and Cuesta-Albertos, Juan A and Liu, Regina Y},
  journal={Journal of the American Statistical Association},
  volume={107},
  pages={737--753},
  year={2012},
  publisher={Taylor \& Francis Group}
}

@article{LafayeDeMicheauxMV20,
  year = {2022},
  journal = {Journal of the American Statistical Association},
  title = {Depth for curve data and applications},
  author = {Lafaye De Micheaux, Pierre and Mozharovskyi, Pavlo and Vimond, Myriam},
  volume = {116},
  number = {536},
  pages = {1881--1897}
}

@article{MozharovskyiJH20,
  author={Pavlo Mozharovskyi and Julie Josse and Fran\c{c}ois Husson},
  title={Nonparametric imputation by data depth},
  journal={Journal of the American Statistical Association},
  volume={115},
  number={529},
  pages={241--523},
  year={2020}
}

@article{MozharovskyiV25,
	author = {Mozharovskyi, Pavlo and Valla, Romain},
	journal = {International Journal of Data Science and Analytics},
	title = {Anomaly detection using data depth: multivariate case},
	year = {2025},
    volume = {116},
    pages = {5171--5196}
}

@article{VallaMFAB25,
    author = {Valla, Romain and Mozharovskyi, Pavlo and d'Alch\'{e}-Bud, Florence'},
	journal = {Technometrics},
	title = {Abnormal component analysis},
	year = {2025},
    volume = {In press},
}

@article{Brunel19,
	author = {Brunel, Victor-Emmanuel},
	journal = {Probability Theory and Related Fields},
	number = {3},
	pages = {1165--1196},
	title = {Concentration of the empirical level sets of Tukey's halfspace depth},
	volume = {173},
	year = {2019}}

@article{10.1007/s11222-025-10700-z, author = {Dyckerhoff, Rainer and Nagy, Stanislav}, title = {Exact computation of angular halfspace depth}, year = {2025}, issue_date = {Oct 2025}, publisher = {Kluwer Academic Publishers}, address = {USA}, volume = {35}, number = {6}, issn = {0960-3174}, url = {https://doi.org/10.1007/s11222-025-10700-z}, doi = {10.1007/s11222-025-10700-z}, journal = {Statistics and Computing}, month = aug, numpages = {29} }

@article{chiapino2019vizu,
  title={A multivariate extreme value theory approach to anomaly clustering and visualization},
  author={Chiapino, Ma{\"e}l and Cl{\'e}men{\c{c}}on, St{\'e}phan and Feuillard, Vincent and Sabourin, Anne},
  journal={Computational Statistics},
  volume={35},
  number={2},
  pages={607--628},
  year={2020},
  publisher={Springer}
}

@inproceedings{NagyLD21,
    author = {Nagy, Stanislav and Laketa, Petra and Dyckerhoff, Rainer},
    title = {Angular halfspace depth: computation},
    booktitle = {CLADAG 2021. Book of Abstracts and Short Papers},
    year = {2021},
    editor = {Giovanni C. Porzio and Carla Rampichini and Chiara Bocci},
    pages = {169--172},
    publisher = {Firenze University Press}
}

@book{Wang1990,
   title =     {Symmetric Multivariate and Related Distributions},
   author =    {Kai Wang Fang},
   publisher = {Chapman \& Hall},
   isbn =      {9781315897943; 1315897946; 9781351077040; 135107704X},
   year =      1990
}

@book{KotzNadarajah2004,
   title =     {Multivariate T-Distributions and Their Applications},
   author =    {Samuel Kotz, Saralees Nadarajah},
   publisher = {Cambridge University Press},
   isbn =      {0521826543; 9780521826549; 9780511550683; 0511550685},
   year =      {2004}
}

\newpage
\appendix

\section{Proofs and Intermediate Results}\label{sec:proofs}
\subsection{Proofs of Results from Section~\ref{sec:Polar}}
\subsubsection{Proof of Lemma \ref{lem:connect-aD-coD}}

From the definitions, 
for $x\neq 0$, 
\begin{align*}
    \coD{x}{\pi}
    & = \inf_{\| u \| = 1 , x \in \AHS{u} }\pi\big\{y\neq 0:~  r(y) \ge  r(x), y \in
    \AHS{u}\big\}.                                                             \\
    & = \inf_{\| u \| = 1 , \theta(x) \in\AHS{u} }\pi\big\{ y\neq 0: ~r(y) \ge  r(x) ]~
    \frac{\pi\{y\neq0:  ~r(y) \ge  r(x), ~ \theta(y) \in \AHS{u} \}}{\pi\big\{ y\neq 0: ~r(y) \ge  r(x) ]} 
    \\
    & \qquad \qquad \qquad (\text{because } \theta(y) \in \AHS{u}
    \iff y \in  \AHS{u}, \forall y\neq 0 )                                \\
    & = \pi_r[r(x),\infty)  \inf_{\| u \| = 1 , \theta(x) \in \AHS{u}
    }\pi_{\theta|r(x)}(\AHS{u}),
\end{align*}

which is \eqref{eq:equiv-coD-aD}.

\subsubsection{Proof of Proposition \ref{prop:main_properties}}
The proof requires an intermediate technical lemma

\begin{lemma}\label{lem:continuousRadialImpliesContinuousConditional}
  If the probability distribution $P_r$ is continuous, then
   for any $t>0$ such that $P_r([t,\infty))>0$, for any sequence $t_n\to t$, we have  $P_r[t_n,\infty)>0$ for sufficiently large $n$, and the sequence of probability measures $(P_{\theta|t_n})$ converges weakly to $P_{\theta|t}$. In other words, 
  the mapping   $t\mapsto P_{\theta\mid t}$ from
    $(0,\infty)$ to the space of probability measures equipped with the
    topology of the weak convergence, is continuous at any  $t>0$.
\end{lemma}
\begin{proof}[Proof of Lemma~\ref{lem:continuousRadialImpliesContinuousConditional}]
Let $t>0$ be  such that $P_r([t,\infty))>0$, and let $t_n\to t$. By the radial continuity assumption in the statement, there exists $n_0\in\nset$ such that $P_{r}([t_n,\infty))>0$ for $n\ge n_0$, so that $P_{\theta|t_n}$ is well defined for $n\ge n_0$.

We show that 
\(P_{\theta|t_n}(A)\to P_{\theta|t}(A)\) for any Borelian set
\(A\) of \(\sphere\). 
This implies
the weak convergence of \(P_{\theta|t_n}\) towards \(P_{\theta|t}\)
and Assumption \ref{as:continuousConditionalAngular}.

Let \(A\) (Borelian set of \(\sphere\)) be fixed. For any \(l>0\) define the set
\(A_l=\{ y\in \R^d:\theta(y)\in A \;\textnormal{and}\;  r(y)\geq l \}
= \ball_{l}^{c}\cap \{y\in\R^d:\theta(y)\in A\}\). 
This
collection of sets indexed by $l>0$ is monotonously non-increasing:  
if
\(l_1\leq  l_2\) then \(A_{l_2}  \subset A_{l_1}\).

Define the sequences
\begin{equation*}
    \bar{t}_n = \sup_{m\geq n}
    t_m\quad\textnormal{and}\quad\underbar{t}_n=\inf_{m\geq n} t_m,\qquad n\ge n_0.
\end{equation*}
The sequence of sets $(A_{\bar{t}_n})$ is non-decreasing for the inclusion relation, $A_{\bar{t}_n}\subset A_{\bar{t}_{n+1}}$, while the sequence $(A_{\underbar{t}_n})$ is non-increasing, $A_{\bar{t}_{n+1}}\subset A_{\bar{t}_{n}}$. Also we have 
$$A_t = \bigcup_{n\ge 1} A_{\bar{t}_n} \cup\{x : r(x) = t, \theta(x)\in A \}$$ and 
 $A_t = \bigcap_{n\ge 1} A_{\underbar{t}_n}$.
By continuity of $P_r$, we also have that $P(\{x : r(x) = t, \theta(x)\in A \}) = 0$. Thus, by countable additivity from below of $P$,
$$P(A_t) = \lim_n P(A_{\bar{t}_n} \cup\{x : r(x) = t, \theta(x)\in A \} ) = \lim_n P(A_{\bar{t}_n}). $$
Also by continuity from above,
$P(A_t) = \lim_n P(A_{\underbar{t}_n})$.  Since
$ P(A_{\bar{t}_n}) \le P(A_{t_n}) \le P(A_{\underbar{t}_n})$, the latter
two convergences yield $P(A_t) = \lim_n P(A_{t_n})$. In addition, by
continuity of $P_r$ again, it holds that $P_r([t_n,\infty))\to P_r ([t,\infty))$. Finally, for $n\ge n_0$,
$$
P_{\theta|t_n}(A) = \frac{P(A_{t_n})}{P_{r}([t_n,\infty))} \to \frac{P(A_{t})}{P_{r}([t,\infty))}
= P_{\theta|t}(A).
$$
  
\end{proof}

We now prove Proposition~\ref{prop:main_properties}

    \noindent\textbf{Rotational invariance:} Notice that
    \( (O_\#P)_r=P_r\) and that
    \begin{align*}
        (O_\#P)_{\theta | r(Ox)} & = \mathcal{L}(\theta(OX) \; |\;
        r(OX)\ge r(Ox)) = \mathcal{L}(\theta(OX) \; |\;  r(X)\ge r(x)) \\
        & =\mathcal{L}(O\theta(X) \; |\;  r(X)\ge r(x)) = O_\#(P_{\theta|r(x)}).
    \end{align*}
    Combining this with Lemma \ref{lem:connect-aD-coD}, we have that
    \begin{align*}
        \coD{Ox}{O_\#P)} & = (O_\#P)_r\left[ r(Ox),+\infty \right)
        \aD{\theta(Ox)}{(O_\#P)_{\theta | r(Ox)}} \\
        & = P_r\left[ r(x),+\infty \right)
        \aD{\theta(Ox)}{O_\#(P_{\theta|r(x)})}       \\
        & = P_r\left[ r(x),+\infty \right)
        \aD{\theta(x)}{P_{\theta|r(x)}}         \\
        & = \coD{x}{P},
    \end{align*}
    where the third equality follows from the rotational invariance
    of the angular halfspace depth (\cite[Theorem 2]{nagy2024theoretical}).
\medskip

    \noindent\textbf{Monotonicity along rays:} The homogeneity
    property of the halfspaces (Equation~\eqref{eq:homogeneity_halfspaces}),
    implies that, for any \(t\geq 1\), 
    \begin{equation*}
        \coD{tx}{P} = \inf_{ \mathclap{\substack{tx \in \AHS{u}
        \\ \| u \| = 1} }}P( \ball_{\|tx\|}^c\cap \AHS{u} )  =\inf_{
        \mathclap{\substack{x \in \AHS{u} \\ \| u \| = 1} }}P(
        (t\ball_{\|x\|})^c\cap \AHS{u} ) \leq \inf_{ \mathclap{\substack{x
        \in \AHS{u} \\ \| u \| = 1} }}P( \ball_{\|x\|}^c\cap \AHS{u} )
        = \coD{x}{P}.
    \end{equation*}
    The result for \(t\leq 1\) is a consequence of the one for $t\ge 1$. 

    Regarding the limit at $0$ along fixed directions,  we first show that 
    for any sequence $(r_n)\to 0$ with $r_n\ge 0$,  the sequence 
\(    P_{\theta | r_n}=\mathcal{L}(\theta(X)| r(X)\geq r_n)\)
converges weakly to the distribution  \( \pi_{\theta|0} = \mathcal{L}(\theta(X) | X \neq 0)\). 
The argument is a simplified version of the proof of Lemma~\ref{lem:continuousRadialImpliesContinuousConditional}. 
Let $r_n\to 0$ be sequence of nonnegative numbers.  We show that for any Borel set $A\subset \sphere$,  $P_{\theta|r_n}(A)\to P_{\theta|0} (A) $. Let $A\subset \sphere$ as above and  let $\bar r_n = \sup_{k\ge n} r_n$. The sequence $(\bar r_n)$ is nonincreasing, and the sequence of sets $A_{\bar r_n}$ is nondecreasing for inclusion. Also, we have  $\bigcup_n A_{\bar r_n} = A_{0^+}:=\{x\in\rset^d\setminus\{0\}: \theta(x)\in A\}$. By continuity from below, we have $P(A_{0^+}) = \lim_n P(A_{\bar r_n})$. In addition $ A_{\bar r_n}\subset A_{r_n} \subset  A_{0^+}$,  whence
$ P(A_{\bar r_n})\le P(A_{r_n}) \le P(A_{0^+})$,  whence
$$P(A_{r_n})\to P(A_{0^+})  \text{ as } n\to\infty.$$
This result, applied to $A=\sphere$, implies that  $ P_r[r_n,\infty) \to P_r(0,\infty)$ as well. The latter quantity is nonzero since we have assumed that $P$ is not a Dirac mass at zero. This, combined with the above display, yields
$$
P_{\theta|r_n}(A) = \frac{P(A_{r_n})}{P_r[r_n,\infty)} \to \frac{P(A_{0^+})}{P_r(0,\infty)} =
P_{\theta|0}(A). 
$$

Consider now $x\neq 0$ and $t_n\to 0$ a sequence of nonnegative numbers. The above argument yields that $P_{\theta|t_n\|x\|}$ converges weakly to $P_{\theta|0}$. Assumption~\ref{as:smoothnessHalfSpaces} is precisely the one required in \citep[][Theorem 10]{nagy2024theoretical} to ensure that the angular halfspace depth is continuous with respect to both arguments, in particular it is continuous with respect to the angular distribution. 
    Thus, it holds that 
    that \(\aD{\theta(x)}{P_{\theta|\|t_nx\|}}\) converges to
    \(\aD{\theta(x)}{P_{\theta|0}}\) as \(n\to \infty\).
    Also for $n$ sufficiently large we have $P_{r}[t_n\|x\|, \infty) >0$, thus from Lemma~\ref{lem:connect-aD-coD}
    $$
    \coD{t_nx}{P} =  P_r[t_n\|x\|,\infty) \aD{\theta(x)}{P_{\theta|\|t_n x\|}}
    \to  P_r(0,\infty) \aD{\theta(x)}{P_{\theta| 0}}
    $$
\medskip

    \noindent\textbf{Vanishing at infinity:} From the definition of
    the center-outward depth we have
    \begin{align*}
        \coD{x}{P} & = \inf_{\| u \| = 1,\; x \in \AHS{u} }P(
        \ball_{\|x\|}^c\cap \AHS{u} )\leq P \left(\ball_{\|x\|}^c\right)\to 0.
    \end{align*}
\medskip

    \noindent\textbf{Upper semicontinuity:}
   From Lemma~\ref{lem:continuousRadialImpliesContinuousConditional} and radial continuity (Assumption~\ref{as:continuousConditionalAngular}), we have that  $P_{\theta|t_n}$ weakly converges to $P_{\theta|t}$ for any sequence $t_n\to t>0$. Thus \cite[Theorem 3]{nagy2024theoretical} implies that
    \(\aD{\theta(x)}{P_{\theta|r(x)}}\)
    is an upper semicontinuous function of \(x\) on the domain $\rset^d\setminus\{0\}$. The result now
    follows from Lemma \ref{lem:connect-aD-coD},
    and the fact that the product of positive
    semicontinuous functions is also a semicontinuous function.
\medskip

    \noindent\textbf{Continuity in measure:}
    The proof follows some ideas in \cite{NagyGijbels2016}, Theorem A.3, with a different class of sets however, requiring a specific proof for completeness. 
    Let \[\mathcal{D}_0 = \{ t\ball^c \cap \AHS{u}, u\in\sphere, t\ge 0.\}\]
    From \eqref{eq:uniform_bound}, it is enough to show that if $\pi_n\to P$ weakly, then also 
    \begin{equation}\label{eq:unif_cv_Pn_Dclass}
    \sup_{D\in\mathcal{D}_0}|\pi_n(D)-P(D)|\to 0.
    \end{equation}
    In other words we need to show that the class of sets $\mathcal{D}_0$ is a $P$-uniformity  class, as defined in \cite{billingsley1967uniformity}.
Introduce the $\delta$-boundary sets 
   $$\partial D(t,u,\delta) = \{ x\in\rset^d: r(x) \in (t-\delta, t+\delta), d(x, H_{0,u})<\delta \}, ~ u\in \sphere, t\ge 0.$$
      By Theorem 2 in \cite{billingsley1967uniformity} it is sufficient  to show that 
      \begin{equation}\label{eq:unifCV_delta_boundaries}
      \lim_{\delta\to 0} \sup_{t\ge 0, u \in\sphere } P(\partial D(t,u,\delta) ) = 0.     
      \end{equation}
      Fix $\varepsilon>0$ and let $M>1$
 such that $P(M\ball^c)< \varepsilon$. Then for $\delta<1(<M)$, we have:
 \begin{equation*}
     \sup_{t\ge 0, u \in\sphere } P(\partial D(t,u,\delta) ) \le 
   \max\Big(   \sup_{t\in[0,2 M], u \in\sphere } P(\partial D(t,u,\delta) ), ~\varepsilon \Big).
 \end{equation*}
 By our continuity assumptions on $P$, the mapping $(t,u)\mapsto P(\partial D(t,u,\delta) )$ is continuous for any $\delta>0$. Hence, by virtue of Dini's theorem, $P(\partial D(t,u,\delta) )$ converges uniformly to $0$ as $\delta\to 0$ on the compact set
 $[0,2M]\times \sphere$. This proves~\eqref{eq:unifCV_delta_boundaries} and the result follows.

\subsubsection{Proof of Theorem \ref{thm:stat-coDepth}}
For the proof of Theorem \ref{thm:stat-coDepth}, we need
the following auxiliary lemmas.

\begin{lemma}\label{lem:inf_sup_bound}
    Let \(P\) and \(Q\) be two finite measures in \(\R^d\) and
    \(\mathcal{A}\) be any collection
    of measurable subsets on \(\R^d\). Then,
    \begin{equation*}
        \left| \inf_{A\in\mathcal{A}}{P(A)} -
        \inf_{A\in\mathcal{A}}{Q(A)}
        \right|\leq\sup_{A\in\mathcal{A}}\left|P(A)-Q(A)\right|.
    \end{equation*}
\end{lemma}
\begin{proof}[Proof of Lemma \ref{lem:inf_sup_bound}]
    Without loss of generality, suppose
    \(\inf_{A\in\mathcal{A}}{P(A)}>\inf_{A\in\mathcal{A}}{Q(A)}\),
    then, there exists a sequence
    \(\{C_n\}\subseteq\mathcal{A}\) such that \(Q(C_n)\downarrow
    \inf_{A\in\mathcal{A}}{Q(A)}\), and
    \(\inf_{A\in\mathcal{A}}{P(A)}>Q(C_n)\geq\inf_{A\in\mathcal{A}}{Q(A)}\)
    for all \(n\). Then, for all \(n\) it holds
    that,
    \begin{align*}
        \left| \inf_{A\in\mathcal{A}}{P(A)} -
        \inf_{A\in\mathcal{A}}{Q(A)} \right|
        & =\inf_{A\in\mathcal{A}}{P(A)} -
        \inf_{A\in\mathcal{A}}{Q(A)}\leq P(C_n) -
        \inf_{A\in\mathcal{A}}{Q(A)} \\
        & \leq P(C_n) - Q(C_n) +
        \left(Q(C_n)-\inf_{A\in\mathcal{A}}{Q(A)}\right)
        \\
        & \leq \sup_{A\in\mathcal{A}}\left|P(A)-Q(A)\right| +
        \left(Q(C_n)-\inf_{A\in\mathcal{A}}{Q(A)}\right).
    \end{align*}
    Because the last inequality holds for all \(n\) and
    \(Q(C_n)\downarrow \inf_{A\in\mathcal{A}}{Q(A)}\), the result follows.
\end{proof}

\begin{lemma}\label{lem:shattering_coef} For all \(t\geq 0\), the
    shattering coefficients of the class of sets
    \(\mathcal{D}_t=\{\ball_{t'}^c\cap H_{0,u}\}_{t'\geq t, u\in\sphere}\)
    satisfy \(s(\mathcal{D}_t, n)\leq 2(n+1)((n-1)^{d-1}+1)\) for all
    \(n\geq 1\). Moreover, for all \(k>1\) and \(n\geq 1/(k-1)\), it
    holds that \(s(\mathcal{D}_t,n)\leq 2k n^d\).
\end{lemma}
\begin{proof}[Proof of Lemma \ref{lem:shattering_coef}]
    The class \(\mathcal{D}_t\) is obtained by interesections of sets in the classes  \(\mathcal{B}_t=\{ {B}^c_{t'} \}_{t'\geq t}\) and
    \(\mathcal{H}_0= \{H_{0,u}\}_{\|u\|=1}\)
    which are both VC \cite[Chapter
    13]{devroye2013probabilistic}, then we have that
    \(s(\mathcal{D}_t,n)\leq
    s(\mathcal{B}_t,n)s(\mathcal{H}_0,n) \) \citep[][Theorem 1.12]{lugosi2002pattern}
    From \cite[Corollary 13.1]{devroye2013probabilistic} we have
    that \(s(\mathcal{H}_{0},n)\leq 2(n-1)^{d-1}+2\)
    and \(s(\mathcal{B}_t,n)=n+1\).
\end{proof}

\begin{lemma}\label{lem:bound_expectation} Let \(P_n\) be the
    empirical measure associated with
    an i.i.d. sample drawn for \(P\). For all \(t>0\), let
    \(p_t=\P(r(Z)\geq t)\) where \(Z\sim P\) and
    \(\mathcal{D}_t=\{\ball_{r}^c\cap H_{0,u}\}_{
    r\geq t,
    u\in\sphere}\). Then it holds that
    \begin{equation*}
        \E \left[ \sup_{A\in\mathcal{D}_t} \left| P_n(A) - P(A) \right|
        \right]\leq\sqrt{\frac{2p_t}{n}\Bigl( (d+3)\log(2) +
        d\log(np_t+1) \Bigr)} + \sqrt{\frac{p_t}{n}}.
    \end{equation*}
\end{lemma}
\begin{proof}[Proof of Lemma \ref{lem:bound_expectation}]
    This is just Proposition 4.3 and Lemma 4.2 in
    \cite{lhaut2022uniform},
    but instead of using the bound given by Sauer's Lemma we use our
    Lemma \ref{lem:shattering_coef}

    Following the notation of Proposition 4.3 in
    \cite{lhaut2022uniform}, take
    \(\mathbb{A}=\{x\in\R^d: r(x)\geq t\}\).
    \begin{align*}
        \E \left[ \frac{K}{n} \sup_{D\in\mathcal{D}_t}
        \left|P^Y_K(D)-P_{\mathbb{A}}(D)\right| \right] & =
        \sum_{k=1}^{n}{ \E \left[ \frac{k}{n}
                \sup_{D\in\mathcal{D}_t}\left| P_k^Y(D) - P_{\mathbb{A}}(D)
        \right| \Bigr| K=k \right]\P(K=k) },
        \\
        & = \frac{1}{n}\sum_{k=1}^{n}{k \E
            \left[\sup_{D\in\mathcal{D}_t} \left| P_k^Y(D) -
        P_{\mathbb{A}}(D) \right| \Bigr| K=k \right]\P(K=k) },
        \\
        & \leq \frac{1}{n}\sum_{k=1}^{n}{k\sqrt{\frac{2
        \log(2s(\mathcal{D}_t, 2k))}{k}}\P(K=k)},
        \\
        & \leq\frac{1}{n} \E \left[ \sqrt{2K \log (2
        s(\mathcal{D}_t,2K))} \I\{K\geq 1\} \right],
    \end{align*}
    where the third inequality follows from \cite[Theorem
    1.9]{lugosi2002pattern}.

    By Lemma \ref{lem:shattering_coef}, \(\sqrt{2K \log (2
        s(\mathcal{D}_t,2K))} \I\{K\geq 1\} \leq \sqrt{2K \Bigl(
    (d+3)\log(2)+d\log(K+1) \Bigr)}\).
    The function \(g(x)=\sqrt{2x\left((d+3)\log(2)+d\log(x+1)\right)}\)
    is concave in \([0,+\infty]\), hence,
    by Jensen's inequality and the fact that
    \(K\sim\textnormal{Bin}(n,p_t)\), we have that
    \begin{equation*}
        \E \left[ \frac{K}{n} \sup_{D\in\mathcal{D}_t}
        \left|P^Y_K(D)-P_{\mathbb{A}}(D)\right| \right] \leq
        \frac{1}{n} \sqrt{2np_t{\left( (d+3)\log(2)+d\log(np_t+1) \right)}}.
    \end{equation*}
    The final expression now follows from \cite[Lemma
    4.2]{lhaut2022uniform}.
\end{proof}

    From the definition of the \(\coD{x}{P}\) and Lemma
    \ref{lem:inf_sup_bound} we have that,
    \begin{align}
        \sup_{x\in\rset^d\setminus\{0\}  } |\coD{x}{P} - \coD{x}{Q}| & \leq
        \sup_{x\in\rset^d\setminus\{0\}}\sup_{x\in H_{0,u},
        u\in\sphere}{\left| P\left({B}^c_{||x||}\cap H_{0,u}\right) -
        Q\left({B}^c_{||x||}\cap H_{0,u}\right) \right|}\nonumber
        \\
        & \leq \sup_{r\geq 0, u\in\sphere} \left|
        P\left({B}^c_{||r||}\cap H_{0,u}\right) -
        Q\left({B}^c_{||r||}\cap H_{0,u}\right) \right|\nonumber
        \\
        & \leq \sup_{D\in\mathcal{D}_0} \left| P(D)-Q(D)
        \right|\label{eq:uniform_bound}.
    \end{align}
    By Lemma \ref{lem:shattering_coef}, the shattering coefficients of
    the class \(\mathcal{D}_0\) satisfy \(s(\mathcal{A},n^2)\leq
    2\alpha n^{2d}\) for all \(n\geq (\alpha-1)^{-1/2}\).
    Hence,
    from \eqref{eq:uniform_bound} and Devroye's inequality
    \cite[Theorem 12.8]{devroye2013probabilistic} we obtain that
    for all \(\epsilon<1\) it holds that
    \begin{equation}\label{eq:devroye_ineq_cod}
        \forall n\geq 1/\sqrt{\alpha-1} \quad,\quad
        \P\left(\sup_{x\in\rset^d\setminus\{0\}  } |\coD{x}{P} -
        \widehat{\mathrm{pD}}(x)|>\epsilon\right)\leq 16e^8 k n^{2d} e^{-2n\epsilon^2},
    \end{equation}
    and equivalently, for all \(\delta\in(0,1)\), and \(n\geq (k-1)^{-1/2}\),
    \begin{equation}\label{eq:devroye_ineq_cod_prob}
        \P\left(\sup_{x\in\rset^d\setminus\{0\}  } |\coD{x}{P} -
            \widehat{\mathrm{pD}}(x)|\leq  \sqrt{\frac{1}{2n}\log{\left(\frac{16 e^8 k
        n^{2d}}{\delta}\right)}} \right)>1-\delta,
    \end{equation}
    Let \(c\geq 2d\), then, \(n^{c-2d}\geq 16 e^8 k \delta^{c-1}\)
    implies that \({\log({16 e^8 k n^{2d}/\delta})}\leq c\log(n/\delta)\),
    therefore, if \(n\geq (k-1)^{-1/2}\), \(c\geq 2d\) and
    \(n^{c-2d}\geq 16 e^8 k \delta^{c-1}\) it holds that
    \begin{equation}\label{eq:devroye_ineq_cod_prob_long}
        \P\left(\sup_{x\in\rset^d\setminus\{0\}  } |\coD{x}{P} -
        \widehat{\mathrm{pD}}(x)|\leq  \sqrt{\frac{c}{2n}\log (n/\delta)} \right)>1-\delta,
    \end{equation}

    Equation \eqref{eq:devroye_ineq_cod_prob_long} follows from taking
    \(k=17/16\) and \(c=2d\) in \eqref{eq:devroye_ineq_cod_prob_long}.

    We now proceed to the proof of \eqref{eq:normalizedDev_2}. Let \(t>0\) be
    fixed. By same argument used to obtain \eqref{eq:uniform_bound} we get that
    \begin{equation}\label{eq:uniform_bound_t}
        \sup_{x: r(x)\ge t   } |\widehat{\mathrm{pD}}(x) - \coD{x}{P}| \leq
        \sup_{D\in\mathcal{D}_t} \left|P(D)-P_n(D)\right|.
    \end{equation}

    From \cite[Proposition 4.1]{lhaut2022uniform} we have that
    \begin{equation*}
        \P\left( \sup_{A\in\mathcal{D}_t}\left| P_n(A) - P(A) \right|
            \leq \frac{2}{3n}\log(1/\delta) +
            2\sqrt{\frac{p_t}{n}\log(1/\delta)} + \E \left[
        \sup_{A\in\mathcal{D}_t} \left| P_n(A) - P(A) \right|\right] \right),
    \end{equation*}
    then, from \eqref{eq:uniform_bound_t} and Lemma
    \ref{lem:bound_expectation} we obtain that the
    following inequality holds with probability at least \(1-\delta\):
    \begin{equation*}
        p_t^{-1} \sup_{x: r(x)\ge t   } |\widehat{\mathrm{pD}}(x) - \coD{x}{P}| \leq
        \frac{2\log(1/\delta)}{3np_t} + \sqrt{\frac{2p_t}{n}}\left(
            \sqrt{2\log(1/\delta)} + \frac{\sqrt{2}}{2}+
        \sqrt{\log{\left(2^{d+3}(np_t+1)^d\right)}} \right).
    \end{equation*}
    Equation \eqref{eq:normalizedDev_2} now follows from the fact that
    if \((2\delta)^{d}\leq 1/8\), then
    \begin{equation*}
        \sqrt{d\log\left(\frac{np_t+1}{\delta}\right)} \geq
        \max\left(\sqrt{ 2\log(1/\delta)}\;,\; \frac{\sqrt{2}}{2}\;,\;
        \sqrt{\Bigl( (d+3)\log(2) + d\log(np_t+1) \Bigr)} \right).
    \end{equation*}

\subsection{Proofs of Results from Section~\ref{sec:Extremes}}

\subsubsection{Proof of Theorem \ref{theo:tailscoD}}

From \eqref{eq:rvWeak} we have
that \(p_t^{-1}= \PP[r(X)>t]^{-1}=L(t)t^{\alpha}\) where \(L\) is a slowly
varying function.
From Lemma \ref{lem:connect-aD-coD},  for all \(t>0\)
and \(x\in \R^d/\{0\}\)
\begin{equation} \label{eq:cod_factorization_t}
    p_t^{-1}\coD{tx}{P}= \frac{p_t^{-1}}{p^{-1}_{tr(x)}}
    \aD{\theta(x)}{P_{\theta|tr(x)}} =
    \frac{L(t)}{L(tr(x))}r(x)^{-\alpha}\aD{\theta(x)}{P_{\theta|tr(x)}}
\end{equation}
From~\eqref{eq:limit_form_angular} the sequence of angular distributions $P_{\theta|tr(x)}$ converges weakly to $\Phi$ as $t\to \infty$. In addition, under Assumption~\ref{as:smoothnessHalfSpaces},  the
angular halfspace depth is continuous with respect to its second (distributional) argument \citep[Theorem 3]{nagy2024theoretical}, so that 
and the slow variation of \(L\), imply that
\begin{equation}\label{eq:cod_pointwise_limit}
    \forall x\in\R^d/\{0\}\quad \lim_{t\to +\infty}{p_t^{-1}\coD{tx}{P}} =
    r(x)^{-\alpha}\aD{\theta(x)}{\Phi}=\coD{x}{\nu_1},
\end{equation}
where the last equality follows from
\eqref{eq:relationship_ad_limit_cod_exponent}.

We now prove \eqref{eq:uniform_limit_extremes}.
Let  \(K\) be a compact set not containing  the origin. Then, there
exist two positive
constants \(\lambda_0<\lambda_1\) such that $K\subseteq
\ball_{\lambda_0}^c\cap \ball_{\lambda_1}$.
By Theorem 1.2.1 in \cite{Bingham1987}, the convergence of
\(L(t)/L(ta)\) to  \(1\)
as \(t\) goes to infinity, is uniform on \([\lambda_0,
\lambda_1]\). Thus, by
Lemma \ref{lem:uniform_convergence_all} (with \(M=[\lambda_0,
    \lambda_1]\) and
\(\lambda=\lambda_0\)) we have that
\begin{equation*}
    \sup_{x\in K} \left| p_t^{-1}\coD{tx}{P} -
    \coD{x}{\nu_1} \right|\leq \sup_{x\in \ball_{\lambda_0}\cap
    \ball_{\lambda_1}} \left| p_t^{-1}\coD{tx}{P} -
    \coD{x}{\nu_1} \right|\xrightarrow[t\to\infty]{} 0.
\end{equation*}

We now prove uniform convergence on sets bounded away from the
origin, under the stronger condition \eqref{eq:rvStrong}. In this
setting $L(t)\to c=\nu(B^c)^{-1}$ as
\(t\to +\infty\). It is easy to see that this implies
$\sup_{a>\delta} | L(ta)/L(t) -1| \to 0$ as $t\to\infty$. 
Equation~\eqref{eq:uniform_limit_extremes_constant_l_with_b}
now follows from Lemma~\ref{lem:uniform_convergence_all}. Finally
\eqref{eq:uniform_limit_extremes_constant_l} is obtained by
multiplying each side by $c^{-1}$ and using the uniform convergence
of \(L(t)\) to \(c\) as \(t\) goes to \(+\infty\).

\subsubsection{Proofs of Proposition \ref{thm:radial-error-part}}
    This is a direct application of Lemma
    \ref{lem:relative_deviation_evt_estimator}, where we take \(\epsilon=1/2\)
for part ii).

\subsubsection{Proofs of Proposition
\ref{thm:control-angular-part}}

    Let
    \(\underline{u}_{n,k}=U(\frac{n}{k}\frac{1}{1+k^{-1/4}})\) and
    \(\overline{u}_{n,k}=U(\frac{n+1}{k}\frac{1}{1-k^{-1/4}})\).
    For all \(x\neq 0\), on the event
    \(\mathcal{G}_{n,k}=\{\underline{u}_{n,k}\leq
    R_k\leq \overline{u}_{n,k}\}\) we have that,
    \begin{align}
        \left|\aD{\theta(x)}{\Phi_k}-\aD{\theta(x)}{P_{\theta|r(x)}}\right|
        &\leq \sup_{\underline{u}_{n,k}\leq t \leq \overline{u}_{n,k}}
        \left|  \widehat{\mathrm{aD}}(\theta(x)) -
        \aD{\theta(x)}{P_{\theta|r(x)}} \right|,\nonumber\\
        &\leq V_\theta(x)+B_\theta(x),
        \label{eq:bis:angular_error_decomposition}
    \end{align}
    where
    \begin{align*}
        V_\theta(x) & = \sup_{\substack{\underline{u}_{n,k}\leq
        t\leq \overline{u}_{n,k}}} \left|
        \widehat{\mathrm{aD}}(\theta(x)) - \aD{\theta(x)}{P_{\theta|t}} \right|
        ,
        \\
        B_\theta(x) & = \sup_{\substack{\underline{u}_{n,k} \leq
        t\leq \overline{u}_{n,k}}}
        \left| \aD{\theta(x)}{P_{\theta|t}} -
        \aD{\theta(x)}{P_{\theta|r(x)}} \right|.
    \end{align*}

    Lemma~\ref{lem:deviation_engular_depth} shows that there is an event
    \(\mathcal{H}_{n,k}\)
    of probability at least \(1-\delta\) where
    \begin{equation*}
        \sup_{\substack{\underline{u}_{n,k}\leq
        t\leq \overline{u}_{n,k}}} \left|
        \widehat{\mathrm{aD}}(\theta(x)) -
        \aD{\theta(x)}{P_{\theta|t}} \right|\leq
        6C_{n,k}\sqrt{\frac{d\log{(2(k+k^{3/4}+1)/\delta)}}{k+k^{3/4}}}+\frac{4C_{n,k}}{3}\frac{\log(2/\delta)}{k+k^{3/4}},
    \end{equation*}
    hence, on \(\mathcal{H}_{n,k}\cap\mathcal{G}_{n,k}\) and for
    all \(x\neq 0\),
    \begin{equation}\label{eq:bis:bound_s_1_on_h}
        V_\theta(x)\leq
        6C_{n,k}\sqrt{\frac{d\log{(2(k+k^{3/4}+1)/\delta)}}{k+k^{3/4}}}+\frac{4C_{n,k}}{3}\frac{\log(2/\delta)}{k+k^{3/4}}.
    \end{equation}

    For \(B_\theta(x)\), notice that for \(k\geq 5\),
    \(U(n/(2k))<\underline{u}_{n,k}\), hence,
    for all \(x\) such that \(r(x)\geq U(n/(2k))\) it holds that
    \begin{align}
        B_\theta(x) & \leq \sup_{\substack{s,t\geq U(n/(2k))
        \\\omega\in\sphere} }{\left|\aD{\theta(\omega)}{P_{\theta|t}}
        - \aD{\theta(\omega)}{P_{\theta|s}}\right|}, \nonumber \\
        & \leq 2\sup_{\substack {t\geq U(n/(2k))   \\ \omega\in
        \sphere}}\left|\aD{\omega}{P_{\theta|t}} -
        \aD{\omega}{\Phi}\right|.
        \label{eq:bis:bound_s_2}
    \end{align}

    Combining \eqref{eq:bis:angular_error_decomposition},
    \eqref{eq:bis:bound_s_1_on_h} and \eqref{eq:bis:bound_s_2} we obtain
    that, on the event \(\mathcal{G}_{n,k}\cap\mathcal{H}_{n,k}\) and for all
    \(x\) such that \(r(x)\geq U(n/(2k))\)
    it holds that
    \(\left|\aD{\theta(x)}{\Phi_k}-\aD{\theta(x)}{P_{\theta|r(x)}}\right|\)
    is smaller than
    \begin{align*}
        6C_{n,k}\sqrt{\frac{d\log{(2(k+k^{3/4}+1)/\delta)}}{k+k^{3/4}}}+\frac{4C_{n,k}}{3}\frac{\log(2/\delta)}{k+k^{3/4}}+
        2\sup_{\substack {t\geq U(n/(2k)) \\ \omega\in
        \sphere}}\left|\aD{\omega}{P_{\theta|t}} -
        \aD{\omega}{\Phi}\right|.
    \end{align*}
    By Lemma \ref{lem:deviationEmpQuantiles},
    \(\P(\mathcal{E}_{n,k})\geq 1-\delta\), and
    \(\mathcal{E}_{n,k}\subseteq\mathcal{G}_{n,k}\), the result
    now follows from the union bound.

\subsection{Proofs of Results from Section~\ref{sec:standardize}}\label{sec:proofs_standardize}

\subsubsection{Proof of Proposition~\ref{prop:limit_coD_stand}}
  From the definitions,
  \begin{align*}
    r(y) \coD{y}{v_\# P} & = r(y) \PP[r(V) \ge  r(y )]~\aD{\theta(y)}{(v_\#P)_{\theta | r(y)}}. 
  \end{align*}
  On the one hand, Assumption~\ref{as:V_uniform_rv} yields
  \begin{equation}
    \label{eq:cv_unif_radial_V}
    \sup_{r(y)>t } \Big|  r(y) \PP[r(V) \ge  \|y \|]  - \Phitrad(\sphere) \Big|
    \xrightarrow[t\to\infty]{} 0. 
  \end{equation}
  On the other hand, from the
same argument as the one leading to~\eqref{eq:uniform_bound}, 
  $$
\sup_{r(y)>t}  |
\aD{\theta(y)}{v_\#P_{\theta|r(y)}} - \aD{\theta(y)}{\Phisprob} |
\le \sup_{u\in\sphere} | \PP[\theta(V)  \in H_{0,u}\cap \sphere~|~r(V)>t ] -
\Phisprob( H_{0,u}\cap \sphere) |.
$$
Assumption~\ref{as:V_uniform_rv} is precisely that the right-hand
side in the above display converges to $0$.
This, combined with (\ref{eq:cv_unif_radial_V}), yields
  $$
  \sup_{ y: \|y \| \ge t}   | r(y) \coD{y}{v_\# P} - \Phitrad(\sphere)
  \aD{\theta(y)}{\Phisprob}| \tto 0, 
  $$
  from which the result follows, using the fact that, by definition, we have 
$\Phitrad(\sphere)\aD{\point}{\Phisprob} = \aD{\point}{\Phitrad}. $

\subsubsection{Proof of Theorem \ref{thm:main_standardized}}

The argument for the proof is based on two intermediate propositions, proved in Subsections~\ref{subsubsec:A1} and~\ref{subsubsec:A2}.
\bigskip

\noindent {\bf Intermediate results.}
The following proposition leverages a  nonasymptotic control of
the deviations of the empirical angular measure established 
in~\cite{ClemenconJalalzaiSabourinSegers2023}. The statement also
encapsulates a new result: an explicit upper bound on the multiplicative constant
$c(d)$ in the cited reference. 
The upper bound given for the deviation term $D(k,\delta)$ is a simplification of \cite{ClemenconJalalzaiSabourinSegers2023}'s bound, given here for concreteness. 

\begin{proposition}[First term $\Delta_{s,1}(x)$ in~(\ref{eq:error_decompos_standardized})]\label{prop:bound_error_phitrad}
Under Assumption~\ref{as:nomass-near-boundaries}, for  $x \in\rset^d$
s.t. $\min_j\tilde v_j(x_j)>0$,
\begin{align*}
  |\hatcoDs{x} - \pD{\tilde v(x)}{\mus}|  
    & \le
    \|\tilde v(x)\|^{-1} \Delta_{\hatPhitrad},
\end{align*}
where
$$
\Delta_{\hatPhitrad}=
\sup_{A \in \mathcal{A}} |\hatPhitrad(A) - \Phitrad(A) |,
$$
and $\mathcal{A} = \{ H_{0,u}\cap \sphere^\tau, u\in\ball \}$.
If in addition, $\Phitrad$ has a bounded density $\phi_s$ w.r.t. the surface measure $\sigma_{d-1}$, then 
$\Delta_{\hatPhitrad}$ is less than the probability upper bound stated
in Theorem 3.1 from
\cite{ClemenconJalalzaiSabourinSegers2023}.  Namely, with probability $1-\delta$,
$$
\Delta_{\hatPhitrad} \le D(k,\delta) + B(k,n)
$$
where $B(k,n)$ and $D(k,\delta)$ are, respectively, a bias term and a
deviation term.  
The bias term is
$$
B(k,n) = \sup_{A\in\mathcal{A}, \sigma\in\pm} \left| \frac{n}{k} \PP[V \in (n/v) \Gamma_{A}^\sigma]  - \mus(\Gamma_A^\sigma)\right|,
$$
where the sets $\Gamma_A^\sigma$ are `framing sets' for the cone
generated by $A$, defined in
\cite{ClemenconJalalzaiSabourinSegers2023}, Section 3.2.
The deviation term $D(k,\delta)$ has a leading term as $k \to \infty$ (with everything else being fixed), which is  less than

\begin{align}
&C_1 k^{-1/2}  \Big(   (d+1)^{5/4} + d^{3/4} \log((d+1)/\delta) +
\tau^{-1/2} \sqrt{\log((d+1)/\delta)} (d + c(d) \log(d/c(d)) + \log(k))
  \Big)  \nonumber\\
& \le
C k^{-1/2}  \log((d+1)/\delta) \Big(   (d+1)^{5/4} +
\tau^{-1/2}  (d + c(d) \log(d/c(d)) + \log(k))
  \Big) \label{eq:bound_variance_term_angularmeasure}
  \end{align}
where $C_1,C$  are universal constants, and where 
$$
c(d) \le \|\phi_s\|_\infty\bigg(
 \frac{32\sqrt{2} d^{2}}{3\pi} + 
 \frac{d^{3/2}}{\sqrt{2\pi}}
      \max\Big(\frac{2 }{\sqrt{1 - 4\tau^2}}, 
       \pi 
       \Big)
\bigg)\sigma_{d-1}(\sphere_+).
$$

  \end{proposition}

We now turn to  the middle term $\Delta_{s,2}(x)$ in \eqref{eq:error_decompos_standardized}. We use intermediate geometric results from Appendix B in \cite{ClemenconJalalzaiSabourinSegers2023}. Formally, unlike the cited reference, we do not assume that the margins of \( X \) follow a unit Pareto distribution, as our analysis is not restricted to rank-based estimators. However the Pareto margins assumptions is made for notational convenience only in the cited reference, and all their  results  remain valid for arbitrary continuous marginal distributions, provided that \( x_j \)'s are replaced with \(v_j(x) =  p_j(x_j)^{-1} \) in the appropriate places, as evident from their proofs. We recall that \( p_j(x_j) = \PP[X_j > x_j] \) and \( \tilde{p}_j(x_j) \) is an estimator thereof. 

\begin{proposition}[Middle term $\Delta_{s,2}(x)$ in~(\ref{eq:error_decompos_standardized})]\label{prop:control_error_polardepths_from_tildev}
Assume that $\Phitrad$ is absolutely continuous w.r.t. the surface measure $\sigma_{d-1}$, with density $\phi_s$ bounded by $\|\phi_s\|_\infty<\infty$.    Let $x\in\rset^d$ such that 
$h := \max_j|\tilde p_j(x_j)/p_j(x_j) - 1 | <  1/8$. 
Then
\begin{align*}
\big| \pD{\tilde v(x)}{\mus} - \pD{ v(x)}{\mus} \big|
    &\le \frac{h/(1-h)}{\|\tilde v(x)\| \vee  \| v(x)\| }
    \Big( \Phitrad(\sphere)  + 
    \|\phi_s\|_\infty\frac{64\sqrt{2} d^{3/2}}{3\pi}\sigma_{d}(\sphere_+)
    \Big).
    \end{align*}
\end{proposition}

\bigskip 

\noindent {\bf Main argument.} We start with multiplying by $\|\tilde v(x)\|$ both sides of the error decomposition~(\ref{eq:error_decompos_standardized}), where the error terms $\Delta_{s,i}, i = 1,2,3$ are defined.
  \begin{itemize}
  \item  On an event $\mathcal{E}_2$ of probability at least $1-\delta_2$, the upper bounds stated in Proposition~\ref{prop:bound_error_phitrad} hold true with $\delta=\delta_2$, thus with the notations of the latter proposition, $\|\tilde v(x)\| \Delta_{s,1}(x) \le \Delta_{\hatPhitrad} \le D(k,\delta_2) + B(k,n)$ where $D(k,\delta)$ is less than~(\ref{eq:bound_variance_term_angularmeasure}), which is the stated upper bound.
    \item  For the second term, Proposition~\ref{prop:control_error_polardepths_from_tildev} yields immediately
  $$\|\tilde v(x)\|\Delta_{s,2}\le  \frac{h}{1-h} \Big(\Phitrad(\sphere)+ \|\phi_s\|_\infty\frac{64\sqrt{2} d^{3/2}}{3\pi}\sigma_{d}(\sphere_+) \Big) $$ on the event $\mathcal{E}_1$ defined in the statement of the theorem such that marginal estimation of the survival functions is (relatively) $h$-accurate.
  \item  Regarding the third term,  using Lemma~\ref{lem:control_err_v_from_err_pj}, we have on the event $\mathcal{E}_1$,  for $x\in\Omega$, 
  \begin{align*}
    \|v(x)\| \ge \|\tilde v(x) \| -  h \|\tilde v(x)\|, 
  \end{align*}
  so that $\|\tilde v(x)\|\ge t \Rightarrow \|v(x)\|\ge t(1-h)$. Moreover
  on $\mathcal{E}_1$, for $x\in \Omega$, from Lemma~\ref{lem:control_err_v_from_err_pj} again,  
  $$
  \frac{\|\tilde  v(x) \|}{\| v(x) \|}
  \le   \frac{\|\tilde  v(x) - v(x)\| +\|v(x) \|}{\| v(x) \|} \le
  1 + \frac{h}{1-h} = \frac{1}{1-h}
  $$
  Thus, on $\mathcal{E}_1$, 
  \begin{align*}
   \sup_{\| \tilde v(x)\| >t}  \|\tilde v(x)\| \Delta_{s,3}(x)
    &   =  \sup_{\| \tilde v(x)\| >t}  \frac{\|\tilde v(x)\|}{\|v(x)\|}
 \|v(x)\| \Big| \coD{v(x)}{\mus} - \coD{v(x)}{v_\#P}  \Big| 
    \\
    &  \le \frac{1}{1-h}  \sup_{y: r(y) \ge  t(1-h)} r(y)  \Big| \coD{y}{\mus} - \coD{y}{v_\#P}  \Big| \\
    & =
      \frac{1}{1-h}  \sup_{y: r(y) \ge  t(1-h)}  \Big|   \aD{\theta(y)}{\Phitrad} - r(y) \coD{y}{v_\#P}  \Big| \\
    & = \frac{1}{1-h} B_1(t(1-h)), 
  \end{align*}
  where the bias function $B_1$ is defined in Proposition~\ref{prop:limit_coD_stand}. 
  \end{itemize}
   The three upper bounds on $\|\tilde v(x)\| \Delta_{s,i}, i = 1,2,3$ derived above are  valid on $\mathcal{E}_1\cap\mathcal{E}_2$, which probability is at least $1-\delta_1 - \delta_2$. The result follows.

\subsubsection{Proof of Proposition~\ref{prop:control_relative_marginal_dev_x_small}}

 We use relative VC inequalities, as stated \emph{e.g.} in~\cite{boucheron2005theory}, Theorem 5.1. We consider here the class $\mathcal{F}_j = \{\un_{(x_j, \infty)}(\point),x_j \in\rset \}$. This is a VC class of dimension $2$. Using Sauer's Lemma, the logarithm of the shattering coefficient of this class for a sample   of size $2n$ thus satisfies
 $
\log \mathcal{S}_\mathcal{F}(2n) \le 2 \log(2en/2) = 2\log(en). 
$
The second  inequality in the cited reference implies that, on an event $\mathcal{E}_{1,j}$  with probability at least $1-\delta/(4d)$, for fixed $j$ 
  $$
 \sup_{x: x_j \le X_{ k/2,j}}
   \frac{\tilde p_j(x_j) -  p_j(x_j)}{\sqrt{\tilde p_j(x_j)}}  \le B(n), 
  $$
  with $$
  B(n) = 2 \sqrt{ \frac{2 \log(en) + \log(16d/\delta) }{n} }. 
      $$
      Using Lemma~\ref{lem:rel_dev_abc} with $A= \tilde p_j(x_j), B=B(n), C =  p_j(x_j)$ we obtain on the same event, for all $x$ such that $ x_j \le X_{ (k/2,j)}$, 
       $$
      \sup_{x: x_j \le X_{ (k/2,j)}}
      \frac{\tilde p_j(x_j)}{ p_j(x_j)} -1 \le \frac{B^2(n) }{p_j(x_j)} + \frac{B(n)}{\sqrt{p_j(x_j)}}.  
      $$
      Now $x_j\le X_{(k/2,j)}$ implies $p_j(x_j)\ge p_j(X_{(k/2,j)})$. We obtain on $\mathcal{E}_{1,j}$, 
      \begin{equation}
        \label{eq:bound_sup_tildep_j_ratio}
 \sup_{x: x_j \le X_{ k/2,j}}
   \frac{\tilde p_j(x_j)}{p_j(x_j)} -1 \le  \frac{B^2(n)}{p_j(X_{(k/2,j)})} + \frac{B(n)}{\sqrt{p_j(X_{(k/2,j)})}}. 
      \end{equation}
      For $k$ such that, with the notations of Lemma~\ref{lem:deviationEmpQuantiles} 
      $$
 \tilde h_1(\delta/(4d),k/2)  \le 1/2,  
      $$
      we have on another event $\mathcal{E}_{2,j}$ of probability at least $1- \delta/(4d)$, from the latter lemma,
      \begin{equation}
        \label{eq:lower_bound_p_j}
       p_j(X_{(k/2,j)})  \ge \frac{k }{4(n+1)}.  
      \end{equation}
       Thus on $\mathcal{E}_j = \mathcal{E}_{1,j}\cap \mathcal{E}_{2,j}$ (with probability at least $1 -  \delta/(2d)$), using the shorthand notation
      $
\ell(n,\delta) = 2 \sqrt{ 2 \log(en) + \log(16d/\delta) }, 
$
we obtain from~(\ref{eq:bound_sup_tildep_j_ratio}), 
      \begin{align}
        \sup_{x: x_j \le X_{ k/2,j}} \frac{\tilde p_j(x_j)}{p_j(x_j)} -1 \le 
        & B^2(n) \frac{4(n+1)}{k} + B(n) \sqrt{\frac{4(n+1)}{k} }\nonumber \\
        & =  \frac{\ell(n,\delta)^2}{n}  \frac{4(n+1)}{k} + \frac{\ell(n,\delta)}{\sqrt{n}} \sqrt{\frac{4(n+1)}{k} } \nonumber \\
        & = \frac{4 \ell(n,\delta)^2}{k} (1 +1/n) +
        \frac{2 \ell(n,\delta)}{\sqrt{k}} \sqrt{1 +1/n}
      \end{align}
      A union bound 
      shows that with probability at least $1- \delta/2$,
         \begin{equation}
          \label{eq:bound_sup_overj_tildepj_ratio_with_n_k}
          \max_{j\le d}\sup_{x: x_j \le X_{ k/2,j}} \frac{\tilde p_j(x_j)}{p_j(x_j)} -1 \le \frac{4 \ell(n,\delta)^2}{k} (1 +1/n) + 
          \frac{2 \ell(n,\delta)}{\sqrt{k}} \sqrt{1 +1/n}
\end{equation}
We now derive a uniform  upper bound for $1 - \frac{\tilde p_j(x_j)}{p_j(x_j)} $. From the first inequality in Theorem 5.1 from \cite{boucheron2005theory}, we have for fixed $j$, with probability $1 - \delta/(4d)$, 
\begin{align*}
  \sup_{x: x_j \le X_{ k/2,j}}
   \frac{p_j(x_j) - \tilde p_j(x_j)}{\sqrt{ p_j(x_j)}}  \le B(n). 
\end{align*}
Dividing both side by $\sqrt{p_j(x_j)}$ we get on this event for all
$x_j \le X_{ (k/2,j)}$,
$$
1 - \frac{\tilde p_j(x_j)}{p_j(x_j)} \le \frac{B(n)}{\sqrt{p_j(x_j)}}
  \le \frac{B(n)}{\sqrt{p_j(X_{k/2,j})}}. 
    $$
    Using again~(\ref{eq:lower_bound_p_j}) and a union bound we get that with probability at least $1 - \delta/2$, 
    \begin{equation}
      \label{eq:bound_inf_overj_tildepj_ratio_with_n_k}
\max_j \sup_{x: x_j \le X_{ k/2,j}} 1 - \frac{\tilde p_j(x_j)}{p_j(x_j)}
\le B(n) \sqrt{ \frac{4(n+1)}{k}} = 2 \frac{\ell(n,\delta)}{\sqrt{k}} \sqrt{1 + 1/n}. 
          \end{equation}
Combining~(\ref{eq:bound_sup_overj_tildepj_ratio_with_n_k}) and~(\ref{eq:bound_inf_overj_tildepj_ratio_with_n_k}), the result follows.

\subsubsection{Proof of Corollary~\ref{cor:total_error_standardized}}
  The first statement derives immediately from the combination of Proposition~\ref{prop:control_relative_marginal_dev_x_small} and Theorem~\ref{thm:main_standardized}.

  For the second statement, consider the event that
  for all $j\le d, X_{(k/2,j)} \ge  \overline{ u}_{n,k,j}$. On this event,
  we have
   $$\Omega_2(M,t) \subset \bigcap_{j\le d} \bigg(\{ x~:~x_j\le X_{(k/2,j)} 
  \text{ ~ or ~}
  x_j \in [\overline u_{n,k,j}, U_j(Mn/k)]\} \bigg). 
  $$
  Hence,  with $e_j(x_j) = \Big| \tilde p_j(x_j)/p_j(x_j) - 1\Big|$, on the latter event, 
  $$
  \max_j\sup_{x \in \Omega_2(t)} e_j(x_j)  
  \le \max\Big( \max_j\{e_j(x_j), x_j \le X_{(k/2,j)} \},
  \max_{j} \{e_j(x_j), x_j \in [\overline u_{n,k,j}, U_j(Mn/k)\Big). 
  $$
  The two terms in the maximum in the above display are controlled respectively by Proposition~\ref{prop:control_relative_marginal_dev_x_small}, and statement $(i)$ from Lemma~\ref{lem:relative_deviation_evt_estimator}  applied to the $j$ marginal distributions $F_j$. A union bound yields that with probability $1 - \delta - d\delta - d\PP[X_{(k/2,1)} \le  \overline{ u}_{n,k,1}]$, 
  \begin{equation}
      \label{eq:bound_err_on_omega2}
      \max_j \sup_{x \in \Omega_2(M,t)} e_j(x_j) \le h_2(\delta,M). 
  \end{equation}
We now derive an upper bound on $\PP[X_{(k/2,1)}\le  \overline{ u}_{n,k,1}]$, where 
$\overline{ u}_{n,k,1} = U_1((n+1)/k ~(1-k^{-1/4})^{-1} )$, with $U_1(y) = F_1^\leftarrow(1-1/y)$. Since we have assumed continuous margins for $X$, as in the proof of Lemma~\ref{lem:deviationEmpQuantiles}, denoting by $Z_{(i)}, i\le n$ the decreasingly ordered, order statistics of a uniform sample, and $Z_{(r)} =Z_{ ( \lfloor r\rfloor )}$ for non-integer $r$,  we may write 
\begin{align*}
    \PP[  X_{(k/2,1)}\le  \overline{ u}_{n,k,1} ]
    & = \PP\left( Z_{(k/2)} \le 1 - \frac{k}{n+1}(1 - k^{-1/4})\right) \\
    & \le \PP\left( Z_{(k/2)}  - 1 + \frac{\lfloor k/2\rfloor}{n+1}  \le 
    - \frac{k}{n+1}(1 - k^{-1/4}) + \frac{k/2+1}{n+1} \right) \\
    &\le \PP\left( Z_{(k/2) } - 1 + \frac{\lfloor k/2\rfloor}{n+1} \le 
     -\frac{k/2}{n+1}(1-H(k))
    \right), 
\end{align*}
with $H(k) = 2(k^{-1/4} + k^{-1})$. Now the same argument as the one leading to the exponential bound~\eqref{eq:deviationOrderUni-reiss-BothSides} in Lemma~\ref{lem:deviationEmpQuantiles} yields, for $t>0$, and $j\in \{1,n\}$,  
$$
\PP[Z_{(j)} - 1 + \frac{j}{n+1} \le  -t] \le \exp\left\{
- \frac{n t^2}{3\big( j/(n+1)) + t\big) } 
\right\}.
$$
The latter two displays yield, for $n\ge 12$,  
\begin{align*}
    \PP[  X_{(k/2,1)}\le  \overline{ u}_{n,k,1} ]
&\le \exp\left\{ -\frac{n ~\frac{k^2 (1-H(k))^2}{4(n+1)^2} }{
3\Big( \frac{k/2+1}{n+1} + \frac{k (1-H(k))}{2(n+1)}\Big)}\right\}\\
&=\exp\left\{ -   \frac{nk^2 (1-H(k))^2/4 }{
(n+1)3\Big( (k/2+1) + (k (1-H(k))/2\Big)}\right\}\\
&\le \exp\left\{ -   \frac{k}{13} ~ \frac{(1-H(k))^2 }{
 1 + 1/k - H(k)/2}\right\} \\
& = \exp\left\{ -   \frac{k}{13} ~ \frac{(1- 2/k^{1/4} - 2/k)^2 }{
 1 -1/k^{1/4} }\right\} \le  \exp\left\{ -   \frac{k}{13} ~ (1- 2/k^{1/4} - 2/k)^2 \right\} 
\end{align*}
  
  The second statement follows from \eqref{eq:bound_err_on_omega2} and the latter display,  by an application of Theorem~\ref{thm:main_standardized}.

  The proof of part~3 follows analogously to that of part~2, this time invoking statement~(ii) of Lemma~\ref{lem:relative_deviation_evt_estimator} with $\varepsilon=1/2$.

\subsubsection{Proof of Proposition~\ref{prop:bound_error_phitrad}}\label{subsubsec:A1}
For  $x \in\rset^d$ such that $\min_j\tilde v_j(x_j)>0$, we have
$\|\tilde v(x)\|>0$, thus
\begin{align*}
    & =  \Big| \frac{1}{\|\tilde v(x)\|}\aD{\theta(\tilde v(x))}{\hatPhitrad} -
                                            \frac{1}{\|\tilde v(x)\|}\aD{\theta(\tilde v(x))}{\Phitrad}\Big| \\
    & \le \|\tilde v(x)\|^{-1} \sup_{u\in \sphere } |\hatPhitrad(H_{0,u}\cap \sphere) - \Phitrad(H_{0,u}\cap\sphere)| \\
    & = \|\tilde v(x)\|^{-1} \sup_{H \in \mathcal{A}} |\hatPhitrad(A) - \Phitrad(A)\|, 
  \end{align*}
  where  $\mathcal{A} =  \{ H_{0,u}\cap \sphere^\tau, u\in\sphere \}$. Considering the intersection with $\sphere^\tau$ is legitimate under Assumption~\ref{as:nomass-near-boundaries}. 
  This is a sub-family of the class of sets considered in Example 3.1 from \cite{ClemenconJalalzaiSabourinSegers2023}. In the latter reference it is shown that this class satisfies their Assumption 3.2 and, with the additional hypothesis  that the angular measure has a bounded density,  also their Assumption~3.1. 
  Item (iii) from this Assumption involves, for each $A = H_{0,u}, u\in\sphere$, and $\varepsilon>0$, framing sets denoted by $
  A_+(\varepsilon), A_-(\varepsilon)
  $ such that $A_-(\varepsilon)\subset A \subset A_+(\varepsilon)$. It is also shown in \cite{ClemenconJalalzaiSabourinSegers2023} that for the considered class $\mathcal{A}$, one can take 
  $$
  A_-(u,\varepsilon) = \{x\in\sphere^{\tau+\varepsilon}: \langle x, u\rangle \ge \sqrt{d}\varepsilon \}~~;~~
  A_+(u,\varepsilon) = \{x\in\sphere^{\tau-\varepsilon}: \langle x, u\rangle \ge -\sqrt{d}\varepsilon \}. 
  $$
Then the constant $c = c(d)>0$ involved in Assumption 3.1, Item (iii) from~\cite{ClemenconJalalzaiSabourinSegers2023} can be chosen as 
$$
c(d) = \sup_{u\in\sphere, \varepsilon\in(0,1), } 
\varepsilon^{-1}\Phitrad(A_+(u,\varepsilon)\setminus A_-(u,\varepsilon)). 
$$
With the assumption that $\Phitrad$ has a bounded density, this simplifies as 
$$
c(d)  \le \|\phi_s\|_\infty\bigg(
\underbrace{\sup_{\varepsilon\in(0,1), u \in \sphere}
\varepsilon^{-1}\sigma_{d-1}(x\in\sphere_+: 
\big| \langle x, u \rangle \big| \in [0,\sqrt d \varepsilon] ) }_{c_1}
+ 
\underbrace{ \sup_{\varepsilon\in(0,1)}
\varepsilon^{-1}\sigma_{d-1}(\sphere^{\tau-\varepsilon}\setminus\sphere^{\tau+\varepsilon})}_{c_2}
\bigg)  
$$
  We use Lemma~\ref{lem:bound_surface_spherical_cap_first_orthant} with $\beta = \sqrt{d}\varepsilon$(\emph{resp.} Lemma~\ref{lem:surface_sphere_tau_plusminus_eps}) to derive an upper bound on $c_1$ (\emph{resp.} $c_ 2$). Namely we obtain 
  \begin{equation}
      \label{eq:bound_C1}
      \begin{aligned}
          c_1& \le 2 \sup_{\varepsilon\in(0,1), u \in \sphere}
\varepsilon^{-1}\sigma_{d-1}(x\in\sphere_+: 
 \langle x, u \rangle \in [0,\sqrt d \varepsilon] ) 
 \le 
 \frac{32\sqrt{2} d^{2}}{3\pi}\sigma_{d-1}(\sphere_+)~; \\
 c_2 & \le 
 \frac{d^{3/2}}{\sqrt{2\pi}}
      \max\Big(\frac{2 }{\sqrt{1 - 4\tau^2}}, 
       \pi 
       \Big)
      \sigma_{d-1}(\sphere_+).
      \end{aligned}
  \end{equation}

Adding the two upper bounds in the latter display yields the upper bound on $c(d)$ in the statement.

\subsubsection{Proof of Proposition~\ref{prop:control_error_polardepths_from_tildev}}\label{subsubsec:A2}
    Write
    \begin{align}
        &  \big| \pD{\tilde v(x)}{\mus} - \pD{ v(x)}{\mus} \big|
        = \Big| \frac{1}{\|\tilde v(x)\|} \aD{ \theta(\tilde  v)}{ \Phitrad} -
        \frac{1}{\| v(x) \| } \aD{\theta( v)}{ \Phitrad}\Big| \nonumber \\
        & \le \frac{ \big|\aD{\theta(\tilde v(x)) }{\Phitrad} -
        \aD{ \theta(v(x)) }{\Phitrad} \big| }{
        \|\tilde v(x) \|}   ~+~
        \aD{ \theta(v(x)) }{\Phitrad} \Big|  \frac{1}{\|\tilde v(x) \| } -
        \frac{1}{\| v(x)} \|\Big| \nonumber \\
        &\le \frac{ \big|\aD{\theta(\tilde v(x)) }{\Phitrad} -
        \aD{ \theta(v(x)) }{\Phitrad} \big| }{
        \|\tilde v(x) \|}   ~+~
        \Phitrad(\sphere) \Big|  \frac{1}{\|\tilde v(x)  \|} -
        \frac{1}{\| v(x) \| } \Big| \nonumber
    \end{align}
    Interchanging the role of $v$ and $\tilde v$ we obtain
    \begin{equation}
        \label{eq:bound_PDsinf}
        \begin{aligned}
            &  \big| \pD{\tilde v(x)}{\mus} - \pD{ v(x)}{\mus} \big|  \\
            & \le   \frac{ \big|\aD{\theta(\tilde v(x)) }{\Phitrad} -
            \aD{ \theta(v(x)) }{\Phitrad} \big| }{
            \|\tilde v(x) \| \vee  \| v(x) \| }  ~+~
            \Phitrad(\sphere) \Big|  \frac{1}{\|\tilde v(x)  \|} -
            \frac{1}{\| v(x) \| } \Big|
        \end{aligned}
    \end{equation}

    We now derive an upper bound for each term on the right-hand side
    of (\ref{eq:bound_PDsinf}). Regarding the second term, with $h =
    \max_j|\tilde p_j(x_j)/p_j(x_j) - 1|$, we have
    \begin{align}
        \Big|  \frac{1}{\|\tilde v(x)  \|} -
        \frac{1}{\| v(x) \| } \Big|
        & = \frac{\Big| \|\tilde v(x)  \| - \| v(x)  \|\Big|
        }{\|\tilde v(x)  \|~\| v(x)  \|} \nonumber \\
        & \le  \frac{ \|\tilde v(x) -  v(x)  \|  }{\|\tilde v(x)
        \|~\| v(x)  \|} \nonumber \\
        \text{(from Lemma~\ref{lem:control_err_v_from_err_pj})}
        & \le \frac{h}{1-h}\frac{ \|\tilde v(x)  \| \wedge \| v(x)  \| }{
        \|\tilde v(x)  \| ~\| v(x)  \|} \nonumber \\
        & = \frac{h}{1-h}\frac{1}{
        \|\tilde v(x)  \| \vee \| v(x)  \|}\label{eq:control_1_over_tildev}.
    \end{align}

    Turning to the first term on the right-hand side
    of~(\ref{eq:bound_PDsinf}), writing $\tilde v= \tilde  v(x)$ and
    $v=v(x)$ for convenience, and  assume without loss of generality
    that $\aD{\theta(\tilde  v) }{\Phitrad} \ge \aD{\theta( v)
    }{\Phitrad}$. Then
    \begin{align}
        \big|\aD{\theta(\tilde v) }{\Phitrad} - \aD{ \theta(v) }{\Phitrad} \big|
        & = \inf_{\tilde u\in\sphere:  \tilde v\in H_{0,\tilde u}}
        \Phi(H_{0,\tilde u}\cap\sphere)  -
        \inf_{ u\in\sphere:  v\in H_{0,u}} \Phi(H_{0, u}\cap\sphere).  \nonumber
    \end{align}
    Consider the angular error $h_\theta = \|\theta(\tilde v) -
    \theta(v)\|$. From Lemma~\ref{lem:control_err_v_from_err_pj}
    and~(\ref{eq:control_theta_from_vect}) we have
    $$
    h_\theta \le \frac{h}{1-h} < 
    1 / 4.
    $$
    For  $u\in\sphere$ such that $ v\in H_{0,u}$, we have
    $$
    \langle u, \theta(\tilde v) \rangle  = \langle u, \theta( v) \rangle +
    \langle u, \theta(\tilde v) - \theta( v) \rangle \ge - h_\theta.
    $$
    Let 
    $u' =  u'(u) =   \theta(u + h_\theta(\theta(\tilde v)-u))$. Then, letting
    $\lambda = \|u + h_\theta( \theta(\tilde v) -u) \|$, we have
    \begin{align*}
        \langle u', \theta(\tilde v) \rangle
        & = \lambda^{-1} \langle u + h_\theta( \theta(\tilde v)-u),
        \theta(\tilde v) \rangle \\
        & = \lambda^{-1}
        \Big((1-h_\theta) \langle u , \theta(\tilde v)\rangle  +
        h_\theta  \Big)\\
        &\ge  \lambda^{-1} \Big( -h_\theta(1-h_\theta) + h_\theta\Big) \\
        & \ge \lambda^{-1} h_\theta^2 \ge 0.
    \end{align*}
    As a consequence $\tilde v \in H_{0,u'}$. Thus we have the inclusion
    $$
    \{ u'(u): u\in\sphere,  v \in H_{0,u} \} \subset
    \{  \tilde u : \tilde u \in\sphere, \tilde v \in H_{0,\tilde u} \}.
    $$
    The infimum over the larger set being less than the infimum over
    the smaller subset, we obtain
    \begin{align}
        0\le \aD{\theta(\tilde v) }{\Phitrad} - \aD{ \theta(v) }{\Phitrad}
        & \le \inf_{u\in\sphere,  v \in H_{0,u} }
        \Phi(H_{0,u'(u)}\cap\sphere)  -
        \inf_{u\in\sphere,  v \in H_{0,u} } \Phi(H_{0,u}\cap\sphere),
        \nonumber \\
        & \le \sup_{ u\in\sphere,  v \in H_{0,u} }
        |\Phi(H_{0,u'(u)}\cap\sphere) -
        \Phi(H_{0,u}\cap\sphere)|. \label{eq:controlAD_from_PhiH_0_u}
    \end{align}
    Now from our definition of $u'(u)$
    and~(\ref{eq:control_theta_from_vect}), we have that for any $u\in\sphere$,
    \begin{align*}
        \|u'(u) - u\| = \| \theta(u  + h_\theta(\theta(\tilde v) -u)) - u\|
    &\le  2 \frac{\|h_\theta(\theta(\tilde v) -u)) \|}{
    1 \vee \| u + h_\theta(\theta(\tilde v) -u) \|} \\
    &\le 4 h_\theta.
\end{align*}
From the latter bound and~(\ref{eq:controlAD_from_PhiH_0_u}) we get
\begin{align}
    \Big|\aD{\theta(\tilde v) }{\Phitrad} - \aD{ \theta(v) }{\Phitrad}\Big|
    &\le \sup_{u \in\sphere, u' \in\sphere, \|u'-u\| \le 4 h_\theta}
    |\Phitrad(H_{0,u'}\cap\sphere) -
    \Phitrad(H_{0,u}\cap\sphere)| \nonumber \\
    & \le \sup_{u \in\sphere, u' \in\sphere, \|u'-u\| \le 4 h_\theta}
      \Phitrad\big(( H_{0,u'}\setminus H_{0,u} )\cap\sphere\big)   . \label{eq:controlAD_from_PhiH_0_u_3}
      \end{align}
Now for fixed $u,u'\in\sphere$, and for $x\in \sphere \cap( H_{0,u'}\setminus H_{0,u} )$ we have $\langle x,u'\rangle \ge 0$, and also 
\begin{align*}
\langle x,u'\rangle  &= \langle x,u\rangle  + \langle x,u'-u\rangle \\
& \le \|u'-u\|
\end{align*}
Recall the notation $M_{\beta,u} = \{x\in \sphere: ~ 0\le \langle x,u\rangle\le \beta \}$ from Lemma~\ref{lem:bound_surface_spherical_cap_first_orthant}. The above inequalities show that $(H_{0,u'}\setminus H_{0,u}) \subset M_{\|u'-u\|, u'}$, hence \eqref{eq:controlAD_from_PhiH_0_u_3} yields 
\begin{align}
    \Big|\aD{\theta(\tilde v) }{\Phitrad} - \aD{ \theta(v) }{\Phitrad}\Big|
    &\le \sup_{u' \in\sphere, }
    |\Phitrad(M_{ 4h_\theta,u'}\cap \sphere) \nonumber \\
    &\le \|\phi_s\|_\infty\frac{16\sqrt{2}(4h_\theta) d^{3/2}}{3\pi}\sigma_{d}(\sphere_+) \quad\textrm{(from Lemma~\ref{lem:bound_surface_spherical_cap_first_orthant})} \label{eq:bound_adDev_theta}
    \end{align}
From~(\ref{eq:bound_PDsinf}), \eqref{eq:control_1_over_tildev} and
\eqref{eq:bound_adDev_theta}, we obtain
\begin{align*}
    \big| \pD{\tilde v(x)}{\mus} - \pD{ v(x)}{\mus} \big|
    &\le \frac{h/(1-h)}{\|\tilde v(x)\| \vee  \| v(x)\| }
    \Big( \Phitrad(\sphere)  + 
    \|\phi_s\|_\infty\frac{64\sqrt{2} d^{3/2}}{3\pi}\sigma_{d}(\sphere_+)
    \Big),
\end{align*}

as claimed.

\section{Auxiliary Lemmas}\label{sec:appendix:auxiliary_lemmas}

\begin{lemma}\label{lem:uniform_convergence_sphere} Suppose
    \eqref{eq:rvWeak} and Assumption~\ref{as:smoothnessHalfSpaces}
    hold, then
    \begin{equation*}
        \sup_{x\in\sphere} \left| p_t^{-1}\coD{tx}{P}-\coD{x}{\nu_1}
        \right|\xrightarrow[t\to\infty]{} 0
    \end{equation*}
\end{lemma}
\begin{proof}[Proof of Lemma \ref{lem:uniform_convergence_sphere}]
    Let \(\{t_n\}_{n\in\N}\) be any sequence of positive numbers such
    that \(t_n\to +\infty\).
    From equation \eqref{eq:cod_factorization_t}, we have that, for
    all \(x\in\sphere\)
    \begin{equation*}
        p_{t_n}^{-1}\coD{t_nx}{P} = \aD{\theta(x)}{P_{\theta|t_n}},
    \end{equation*}
    From~(\ref{eq:limit_form_angular}),
    \(P_{\theta|t_n}\)
    converges weakly to $\Phi$ 
    hence, by Theorem 10 in \cite{nagy2024theoretical} we obtain that
    \begin{equation*}
        \sup_{x\in\sphere}|\aD{\theta(x)}{P_{\theta|t_n}} -
        \aD{\theta(x)}{\Phi}|\xrightarrow[n\to\infty]{} 0.
    \end{equation*}
    From~\eqref{eq:def_angular_measure}
    and Definition~\ref{def:coD}, we have that for any \(x\in\sphere\)
    \begin{equation*}
        \aD{x}{\Phi} = \inf_{u\in\sphere,x\in
        H_{0,u}}\Phi\left(H_{0,u}\right) =
        \inf_{u\in\sphere,x\in H_{0,u}}{\nu_1\left(B_1^c\cap
        H_{0,u}\right)}=\coD{x}{\nu_1},
    \end{equation*}
    which concludes the proof.
\end{proof}

\begin{lemma}\label{lem:uniform_convergence_all} Let \(p_t^{-1}=
    \PP[r(X)>t]^{-1}=L(t)t^{\alpha}\) where \(L\) is a slowly
    varying function and suppose
    \eqref{eq:rvWeak} and Assumption~\ref{as:smoothnessHalfSpaces}
    hold. Assume further
    that the convergence of \(L(t)/L(ta)\) to \(1\) as \(t\) goes to
    infinity is uniform on a
    subset \(M\subseteq\R\), that is
    \begin{equation}\label{eq:uniform_limit_l_n}
        \sup_{a\in M}
        \left|\frac{L(t)}{L(ta)}-1\right|\xrightarrow[t\to\infty]{} 0.
    \end{equation}
    Then, for all \(\lambda>0\), it holds that
    \begin{equation}
        \sup_{x:r(x)\geq \lambda, r(x)\in M} \left| p_t^{-1}\coD{tx}{P} -
        \coD{x}{\nu_1} \right|\xrightarrow[t\to\infty]{} 0.
    \end{equation}
   
\end{lemma}
\begin{proof}[Proof of Lemma \ref{lem:uniform_convergence_all}]
    Let \(B_\lambda=\{x\in\R^d: r(x)\in [\lambda,+\infty)\cap M\}\).
    By the triangle inequality, we have,
    \begin{equation*}
        \forall t>0\quad,\quad\sup_{x\in
        B_\lambda}{\left|p_t^{-1}\coD{tx}{P}-\coD{x}{\nu_1}\right|}\leq
        \sup_{x\in B_{\lambda}}{I_1(t,x)}+\sup_{x\in B_{\lambda}}{I_2(t,x)},
    \end{equation*}
    where
    \begin{align*}
        I_1(t,x) & = \left|
        p_t^{-1}\coD{tx}{P}\left(1-\frac{L(tr(x))}{L(t)}\right) \right|, \\
        I_2(t,x) &
        =\left|\frac{L(tr(x))}{L(t)}p_t^{-1}\coD{tx}{P}-\coD{x}{\nu_1}\right|.
    \end{align*}

    From equation \eqref{eq:cod_factorization_t}, we have that, for
    all \(x\in B_{\lambda}\) and \(t>0\)
    \begin{equation*}
        {I_1(t,x)}\leq r(x)^{-\alpha}
        \left|\frac{L(t)}{L(tr(x))}\right|
        \left|1-\frac{L(tr(x))}{L(t)}\right|\leq
        \lambda^{-\alpha}\left|\frac{L(t)}{L(tr(x))}-1\right|\leq
        \lambda^{-\alpha}\sup_{a\in K} \left|\frac{L(t)}{L(ta)}-1\right|,
    \end{equation*}
    and from \eqref{eq:uniform_limit_l_n} we get that \(\sup_{x\in
    B_{\lambda}}{I_1(t,x)}\to 0\) as \(t\to +\infty\).

    To control \(I_2\), notice that, for all \(x\in B_{\lambda}\) and
    \(t>0\), it holds that
    \begin{align}
        I_2(t,x) & =
        \left|\frac{L(tr(x))}{L(t)}p_t^{-1}\coD{tx}{P}-\coD{x}{\nu_1}\right|
        =
        \left|L\left(tr(x)\right)t^{\alpha}\coD{tx}{P}-\coD{x}{\nu_1}\right|,
        \nonumber
        \\
        & = \left| p_{tr(x)}^{-1}\coD{tr(x)}{P}r(x)^{-\alpha}  -
        r(x)^{-\alpha}\coD{\frac{x}{r(x)}}{\nu_1}\right|,
        \nonumber
        \\
        & = r(x)^{-\alpha}\left|
        p_{tr(x)}^{-1}\coD{tr(x)\frac{x}{r(x)}}{P}r(x)^{-\alpha}  -
        \coD{\frac{x}{r(x)}}{\nu_1} \right|,
        \nonumber    \\
        & \leq \lambda^{-\alpha}\left|
        p_{tr(x)}^{-1}\coD{tr(x)\frac{x}{r(x)}}{P}r(x)^{-\alpha}  -
        \coD{\frac{x}{r(x)}}{\nu_1} \right|.\label{eq:control_I_2}
    \end{align}

    Let \(\epsilon>0\) be fixed. From Lemma
    \ref{lem:uniform_convergence_sphere} we have that
    there exists \(t_{\epsilon}\) such that
    \begin{equation*}
        \forall t\geq t_\epsilon,\quad \sup_{w\in\sphere}\left|
        p_t^{-1}\coD{tw}{P} - \coD{w}{\nu_1} \right|\leq \epsilon
        \lambda^{\alpha}.
    \end{equation*}
    Let \(x\in B_{\lambda}\). Notice that if \(t\geq
    t_\epsilon/\lambda\), then \(tr(x)\geq t_\epsilon\),
    therefore,
    \begin{equation*}
        \forall t\geq t_\epsilon/\lambda, x\in B_{\lambda},\quad
        \left|p_{tr(x)}^{-1}\coD{tr(x)\frac{x}{r(x)}}{P} -
        \coD{\frac{x}{r(x)}}{\nu_1}\right| \leq \epsilon\lambda^{\alpha}.
    \end{equation*}
    Combining this with \eqref{eq:control_I_2} shows that
    \(\sup_{x\in B_{\lambda}}{I_2(t,x)}\)
    goes to \(0\) as \(t\) goes to \(+\infty\), which completes the proof.
\end{proof}

\begin{lemma}[Deviation of empirical
quantiles]\label{lem:deviationEmpQuantiles}
    Let $W$ be  continuous, real valued random variable with
    distribution function $F$. Let $F^\leftarrow$ denote the
    generalized (left-continuous) inverse of $F$ and
    \(U(t)=F^{\leftarrow}(1-t^{-1})\).
    Let $W_i, i \le n$ be an \iid~sample according to $F$, and let $W_{(1)}\ge
    \dots W_{(n)}$ denote the associated (decreasingly ordered)
    order statistics. Then with
    \[\tilde{h}_0(\delta,k):= \sqrt{\frac{3\log(2/\delta)}{k}} +
        \frac{3\log(2/\delta)}{k},\quad\textnormal{and}\quad
        \tilde{h}_1(\delta,k) = \sqrt{\frac{3\log(2/\delta)}{k}} +
    \frac{6\log(2/\delta)}{k},\]
    with probability at least $1-\delta$, it holds that
    \begin{equation*}
        U\left(
        \frac{n}{k\left(1+\tilde{h}_0(\delta,k)\right)} \right) \leq
        W_{(k)}\leq U\left(
        \frac{n+1}{k\left(1-\tilde{h}_1(\delta,k)\right)}\right).
    \end{equation*}
\end{lemma}
\begin{proof}[Proof of Lemma~\ref{lem:deviationEmpQuantiles}]
    Let $Z_i = F(W_i)$. By our continuity assumption, the $Z_i$'s
    form an independent sample of standard uniform random
    variables and \(Z_{(k)}=F(W_{(k)})\). By~\cite[Lemma
    3.1.1.]{reiss2012approximate}
    and using the fact that \(1-Z_{(k)}\eqd Z_{(n+1-k)}\)
    we have that, for \(k\leq n\) and \(t>0\)
    \begin{equation}\label{eq:deviationOrderUni-reiss-BothSides}
        \P\left( \left| Z_{(k)} -1 +\frac{k}{n+1} \right|\geq t
        \right)\leq 2\exp{\left(-\frac{nt^2}{3\left(\sigma^2+t\right)}\right)}
    \end{equation}
    with $\sigma^2 = (1-k/(n+1))({k}/(n+1))\leq k/n$. Let
    \(\alpha_n(\delta)=\frac{3\log{(2/\delta)}}{n}\) and
    \(t_n({\delta})=\alpha_n(\delta)+\sqrt{\sigma^2\alpha_n(\delta)}\).
    Inverting inequality~\eqref{eq:deviationOrderUni-reiss-BothSides}
    yields \(\P\left( \left| Z_{(k)} -1 +\frac{k}{n+1}
    \right|\leq t_n{(\delta)} \right)\geq 1-\delta\), hence,
    \begin{equation*}
        1-\frac{k}{n+1}-t_n{(\delta)}\leq Z_{(k)}\leq
        1-\frac{k}{n+1} + t_n(\delta)
    \end{equation*}
    with probability at least   \(1-\delta\). Using that
    \(Z_{(k)}=F(W_{(k)})\) and the fact that
    \(y\leq F(x) \leq z\) implies \( F^{\leftarrow}(y)\leq x \leq
    F^{\leftarrow}(z)\)
    for any \(x\in\R\) and \(y,z\in(0,1)\), the previous
    inequality yields that, with probability
    at least \(1-\delta\)
    \begin{equation*}
        F^{\leftarrow}\left(1-\frac{k}{n+1}-t_n{(\delta)}\right)\leq W_{(k)}\leq
        F^{\leftarrow}\left(1-\frac{k}{n+1} + t_n(\delta)\right).
    \end{equation*}
    Notice that,
    \begin{align*}
        1-\frac{k}{n+1}-t_n{(\delta)} & = 1-\frac{k}{n}\left(
            \frac{n}{n+1} + \frac{3\log(2/\delta)}{k} +
            \sqrt{\frac{3\log(2/\delta)}{k}}\sqrt{\frac{(n+1-k)n}{(n+1)^2}}
        \right),                                              \\
        & \geq
        1-\frac{k}{n}\left(1+\frac{3\log(2/\delta)}{k}+\sqrt{\frac{3\log(2/\delta)}{k}}\right),
    \end{align*}
    and similarly,
    \begin{align*}
        1-\frac{k}{n+1} + t_n(\delta) & = 1-\frac{k}{n+1}\left(
            1 -\frac{n+1}{n}\times\frac{3\log(2/\delta)}{k} -
            \sqrt{\frac{3\log(2/\delta)}{k}}\sqrt{\frac{n+1-k}{n}}
        \right),
        \\
        & \leq 1- \frac{k}{n+1}\left( 1-\frac{6 \log(2/\delta)}{k}
        - \sqrt{\frac{3\log(2/\delta)}{k}} \right).
    \end{align*}
\end{proof}

\begin{lemma}\label{lem:relative_deviation_evt_estimator} 
    Consider a heavy-tailed distribution \(F\) with tail index
    \(\alpha > 0\) and survival function \(\bar{F}(x) = x^{-\alpha}
    / L(x)\), where \(L\) is a slowly varying function. Let \(W_1,
    \ldots, W_n\) be an i.i.d.\ sample from \(F\), and denote by
    \(W_{(1)} \geq \cdots \geq W_{(n)}\) the order statistics in
    decreasing order. Suppose \(\widehat{\alpha}\) is an estimator
    of the tail index \(\alpha\) and let
    \(\Delta_{\alpha}=|\widehat{\alpha}-\alpha|\). For any \(k <
    n\) and \(x > 0\),
    define \(p_x = \bar{F}(x)\) and consider the estimator
    \(\widehat{p}_{x} = \frac{k}{n}
    \left(x/X_{(k)}\right)^{-\widehat \alpha}\).
    Assume \(k\geq (12\log(2/\delta))^2\), then

    \begin{enumerate}[label= (\roman*)]
        \item For any \(M\geq 2\) and \(\delta\in(0,1)\), the
            following inequalities hold
            with probability at least  \(1-\delta\):
            \begin{align}
                \sup_{x\geq \overline{u}_{n,k}}
                \frac{n}{k}\left|\widehat{p}_{x}-p_x\right| & \leq
                \frac{\Delta_{\alpha}}{e \min{(\alpha,
                \hat{\alpha})}} + \sqrt{\frac{12\log(2/\delta)}{k}} +
                \frac{12\log(2/\delta)}{k} + \frac{4}{n} +
                B_{RV},\label{eq:control_absolute_deviation}    \\
                \sup_{x\in \left[\overline{u}_{n,k},
                U(\frac{Mn}{k})\right]}
                \left|\frac{\widehat{p}_{x}}{p_x}-1\right|  & \leq M
                \left(\frac{\Delta_{\alpha}}{e \min{(\alpha,
                    \hat{\alpha})}} +
                    \sqrt{\frac{12\log(2/\delta)}{k}} +
                    \frac{12\log(2/\delta)}{k} + \frac{4}{n} +
                B_{RV}\right),\label{eq:control_relative_deviation_on_compacts}
            \end{align}

        \item Let \(\epsilon\in(0,1)\) and suppose further that \(L\)
            is bounded 
            and that \(n\)
            and \(k\) are such that
            \begin{equation}\label{eq:uniform_boundness_slowly_varying}
                \sup_{s\in[3/4,4]}{\left|\left(\frac{U(ns/k)}{U(n/k)}\right)^{\alpha}-s\right|}\leq
                \frac{1}{2}\quad\textnormal{and}\quad
                \frac{\overline u_{n,k}}{\underline u_{n,k}}  \le 2,
            \end{equation}
            and there exists an event \(\mathcal{D}_{n,k}\) of
            probability at least  \(1-\delta\)
            such that
            \begin{equation}\label{eq:control_absolute_deviation_alpha_estimator}
                \Delta_{\alpha}^{1-\epsilon}\leq \frac{4}{7}.
            \end{equation}

            Then, the following inequality holds with probability larger
            than \(1-2\delta\):
            \begin{equation}\label{eq:relative_deviation}
                \sup_{\overline{u}_{n,k}\leq x\leq
                \exp{\left(\Delta_\alpha^{-\epsilon}\right)}\underline{u}_{n,k}}
                \left|\frac{\widehat{p}_x}{p_x}-1\right|\leq
                \frac{7}{4} \Delta_{\alpha}^{1-\epsilon} +
                4\sqrt{\frac{3\log(2/\delta)}{k}} +
                \frac{24\log(2/\delta)}{k} + \frac{8}{n} + 4B_{SRV}.
            \end{equation}

        \item As \(n/k\to+\infty\), the bias term
            \(B_{RV}\) converges to \(0\). Moreover, if there is
            a non-zero constant \(c\) such that \(L(x)\to c\) as
            \(x\to +\infty\),
            then \(B_{SRV}\) also converges to \(0\) as \(n/k\to+\infty\).
    \end{enumerate}

\end{lemma}

\begin{proof}
    With
    \(\tilde{h}_0(\delta)\) and \(\tilde{h}_1(\delta)\)
    as in Lemma~\ref{lem:deviationEmpQuantiles}, define the
    following events:
    \begin{equation*}
        \mathcal{G}_{n,k} = \{\underline{u}_{n,k}\leq W_{(k)}\leq
        \overline{u}_{n,k}\} \quad,\quad
        \mathcal{E}_{n,k} = \left\{
            U\left(\frac{n}{k(1+\tilde{h}_0(\delta))}\right)\leq W_{(k)} \leq
        U\left(\frac{n+1}{k(1-\tilde{h}_1(\delta))}\right) \right\}.
    \end{equation*}
    Notice that
    \begin{equation}\label{eq:inclusion}
        k\geq (12\log(2/\delta))^2\Rightarrow
        \delta\geq
        2\exp{\left(-\sqrt{k}/12\right)}
        \Rightarrow
        \mathcal{E}_{n,k}\subseteq\mathcal{G}_{n,k}\quad\textnormal{and}\quad\P\left(\mathcal{E}_{n,k}\right)\geq
        1-\delta,
    \end{equation}
    where the second statement follows from
    Lemma~\ref{lem:deviationEmpQuantiles}.

    \textbf{Proof of (i):} We have that \(\frac{n}{k}|\widehat p_{x}
        - p_{x}|\leq
    I_1+I_2\) where
    \begin{equation*}
        I_1 = \bigg|  \Big(\frac{x}{W_{(k)}} \Big)^{-\widehat \alpha}  -
        \Big(\frac{x}{W_{(k)}}  \Big)^{- \alpha} \bigg|\;\;
        \textnormal{and}\;\;I_2  = \bigg|  \Big(\frac{x}{W_{(k)}}
        \Big)^{-\alpha} - \frac{n}{k}\bar{F}\left(x\right)  \bigg|.
    \end{equation*}

    In order to uniformly bound \(I_1\), consider the function
    \(f(a)=(x/W_{(k)})^{-a}\). On \(\mathcal{G}_{n,k}\) it holds that
    \(x\geq W_{(k)}\) for all \(x\geq \overline{u}_{n,k}\) then,
    for each such \(x\) there exists \(c_x\in [\min(\alpha,
            \hat{\alpha}), \max(\alpha,
    \hat{\alpha})]\) such that, on \(\mathcal{G}_{n,k}\),
    \begin{align}
        \left|  \left(\frac{x}{W_{(k)}}\right)^{-\widehat
        \alpha}-\left(\frac{x}{W_{(k)}}\right)^{- \alpha} \right| & =
        \Delta_{\alpha}\left|\log{\frac{x}{W_{(k)}}}\right|\left(\frac{x}{W_{(k)}}\right)^{-c_x}
        \leq
        \Delta_{\alpha}\left|\log{\frac{x}{W_{(k)}}}\right|\left(\frac{x}{W_{(k)}}\right)^{-\min{(\alpha,\hat{\alpha})}}\nonumber
        \\
        & \leq \Delta_{\alpha} \sup_{s\geq 1} \log(s)s^{-\min{(\alpha,
        \hat{\alpha})}} = \frac{\Delta_{\alpha}}{e \min{(\alpha,
        \hat{\alpha})}}.\label{eq:bound_I_1}
    \end{align}

    For \(I_2\), notice that it can be further expanded as
    \begin{align*}
        I_2 & \le \underbrace{\bigg|  \Big(\frac{x}{W_{(k)}} \Big)^{-\alpha} -
            \Big(\frac{x}{U(n/k)
        } \Big)^{-\alpha}  \bigg| }_{V_{1}(x)}+
        \underbrace{\bigg|  \Big(\frac{x}{U(n/k)} \Big)^{-\alpha} -
        \frac{n}{k}\bar{F}\left(x\right)\bigg|}_{\leq B_F}.
    \end{align*}

    Using that \(x\geq \overline{u}_{n,k} >
    U(n/k)\), we obtain
    \begin{equation*}
        V_{1}(x) = x^{-\alpha}\Bigl| W_{(k)}^\alpha -
        U(n/k)^{\alpha} \Bigr|\leq
        \left|W_{(k)}^{\alpha}U(n/k)^{-\alpha}-1\right|.
    \end{equation*}

    On \(\mathcal{E}_{n,k}\), we have that
    \begin{equation}\label{eq:r_k_ratio_control}
        U\left(\frac{n}{k(1+\tilde{h}_0(\delta))}\right)^{\alpha}U(n/k)^{-\alpha}\leq
        W_{(k)}^\alpha U(n/k)^{-\alpha}\leq
        U\left(\frac{n+1}{k(1-\tilde{h}_1(\delta))}\right)^{\alpha}U(n/k)^{-\alpha}.
    \end{equation}
    Under the hypothesis \(k\geq (12\log(2/\delta))^2\),
    we have that
    \(\tilde{h}_1(\delta)\leq 1/2\),
    combining this with the right-hand side
    of~\eqref{eq:r_k_ratio_control} we obtain that
    \begin{align*}
        W_{(k)}^\alpha U(n/k)^{-\alpha} -1 & \leq
        \frac{1}{1-\tilde{h}_1(\delta)} - 1 +
        U\left(\frac{n+1}{k(1-\tilde{h}_1(\delta))}\right)^{\alpha}U(n/k)^{-\alpha}
        - \frac{1}{1-\tilde{h}_1(\delta)}
        \\
        & \leq 2\tilde{h}_1(\delta) + \sup_{s\in[1,2]}{\left|
            \left(\frac{U\left((n+1)s/k\right)}{U(n/k)}\right)^{\alpha}
        -s\right|},
        \\
        & \leq 2\tilde{h}_1(\delta) +
        \sup_{s\in[1,4]}{\left|\left(\frac{U(ns/k)}{U(n/k)}\right)^{\alpha}-s\right|}+4/n.
    \end{align*}

    A similar argument shows that
    \begin{align*}
        W_{(k)}^\alpha U(n/k)^{-\alpha} -1 & \geq
        \frac{1}{1+\tilde{h}_0(\delta)} - 1 +
        U\left(\frac{n}{k(1+\tilde{h}_0(\delta))}\right)^{\alpha}U(n/k)^{-\alpha}
        - \frac{1}{1+\tilde{h}_0(\delta)},
        \\
        & \geq -\tilde{h}_0(\delta) -
        \sup_{s\in[3/4,1]}{\left|\left(\frac{U(ns/k)}{U(n/k)}\right)^{-\alpha}-s\right|}.
    \end{align*}

    Therefore, on \(\mathcal{E}_{n,k}\) and for all \(x\geq
    \overline{u}_{n,k}\), it holds that
    \begin{equation}\label{eq:variance_quantile}
        V_{1}(x) \leq \sqrt{\frac{12\log(2/\delta)}{k}} +
        \frac{12\log(2/\delta)}{k} + \frac{4}{n} +
        \sup_{s\in[3/4,4]}{\left|\left(\frac{U(ns/k)}{U(n/k)}\right)^{\alpha}-s\right|}.
    \end{equation}

    By~\eqref{eq:inclusion}, we have
    \(\mathcal{E}_{n,k}\subseteq\mathcal{G}_{n,k}\).
    Therefore, on the event \(\mathcal{E}_{n,k}\) (which has
    probability at least \(1-\delta\)) the
    inequalities~\eqref{eq:bound_I_1}
    and~\eqref{eq:variance_quantile} both hold, hence
    for all \(x \geq \overline{u}_{n,k}\), the quantity
    \(\frac{n}{k}|\widehat p_{x} - p_{x}|\)
    is bounded above by
    \begin{equation*}
        \frac{\Delta_{\alpha}}{e \min{(\alpha,
        \hat{\alpha})}} + \sqrt{\frac{12\log(2/\delta)}{k}} +
        \frac{12\log(2/\delta)}{k} + \frac{4}{n} +
        B_\textnormal{RV}(n,k),
    \end{equation*}
    which completes the proof of \eqref{eq:control_absolute_deviation} .

    Notice that,
    \begin{align*}
        \sup_{x\in [\overline{u}_{n,k}, U(\frac{Mn}{k})]}
        \left|\frac{\widehat{p}_{x}}{p_x}-1\right| =
        p_x^{-1}\sup_{x\in [\overline{u}_{n,k},
        U(\frac{Mn}{k})]}\left|\widehat{p}_{x}-p_x\right|\leq
        M\sup_{x\geq
        \overline{u}_{n,k}}\frac{n}{k}\left|\widehat{p}_{x}-p_x\right|,
    \end{align*}
    inequality \eqref{eq:control_relative_deviation_on_compacts}
    now follow from \eqref{eq:control_absolute_deviation}.

    \textbf{Proof of (ii):} For any \(x>0\), it holds that
    \begin{align*}
        \widehat{p}_x & = \frac{k}{n}
        \left(\frac{x}{W_{(k)}}\right)^{-\widehat \alpha} =
        \bar{F}\left(U(n/k)\right)\left(\frac{x}{W_{(k)}}\right)^{-\widehat
        \alpha} =
        \frac{U(n/k)^{-\alpha}}{L(U(n/k))}\left(\frac{x}{W_{(k)}}\right)^{-\widehat
        \alpha+\alpha}\left(\frac{x}{W_{(k)}}\right)^{-\alpha}
        \\
        & = \left(\frac{x}{W_{(k)}}\right)^{-\widehat
        \alpha+\alpha}\left(\frac{U(n/k)}{W_{(k)}}\right)^{-\alpha}\frac{x^{-\alpha}}{L(U(n/k))}
        = \left(\frac{x}{W_{(k)}}\right)^{-\widehat
        \alpha+\alpha}\left(\frac{U(n/k)}{W_{(k)}}\right)^{-\alpha}
        \frac{L(x)}{L(U(n/k))} p_{x},
    \end{align*}
    hence,
    \begin{equation}\label{eq:ratio_product_decomposition}
        \frac{\widehat{p}_x}{p_x} =
        \left(\frac{x}{W_{(k)}}\right)^{-\widehat
        \alpha+\alpha}\left(\frac{U(n/k)}{W_{(k)}}\right)^{-\alpha}
        \frac{L(x)}{L(U(n/k))} =
        \left(E_1(x)+1\right)\left(E_2(x)+1\right)\left(E_3(x)+1\right),
    \end{equation}
    where
    \begin{equation*}
        E_1(x) = \left(\frac{x}{W_{(k)}}\right)^{-\widehat
        \alpha+\alpha} -1, \quad
        E_2(x) = \left(\frac{U(n/k)}{W_{(k)}}\right)^{-\alpha} -1, \quad
        E_3(x) = \frac{L(x)}{L(U(n/k))} -1.
    \end{equation*}

    The hypotheses \(\overline u_{n,k} /\underline u_{n,k} \le 2<e\)
    and
    \eqref{eq:control_absolute_deviation_alpha_estimator} imply
    that \(\overline{u}_{n,k} <
    \exp{\left(\Delta_\alpha^{-\epsilon}\right)}\underline{u}_{n,k}\),
    therefore on \(\mathcal{G}_{n,k}\cap\mathcal{D}_{n,k}\) and for
    all \(x\) such that \(\overline{u}_{n,k}\leq x\leq
    \exp{\left(\Delta_\alpha^{-\epsilon}\right)}\underline{u}_{n,k}\),

    it holds that
    \begin{equation*}
        \Delta_{\alpha}\log\left(\frac{x}{W_{(k)}}\right)\leq
        \Delta_{\alpha}\log\left(\frac{\exp{\left(\Delta_\alpha^{-\epsilon}\right)}\underline{u}_{n,k}}{\underline{u}_{n,k}}\right)
        = \Delta_{\alpha}^{1-\epsilon}<1       
    \end{equation*}
    then, we can apply the inequality \(|e^z-1|<\frac{7}{4}|z|\)
    for \(|z|<1\) \cite[pp.82]{bullen1998dictionary} to \(E_1(x)\),
    obtaining that on the event
    \(\mathcal{G}_{n,k}\cap\mathcal{D}_{n,k}\) and for all \(x\)
    such that \(\overline{u}_{n,k}\leq x\leq
    \exp{\left(\Delta_\alpha^{-1/2}\right)}\underline{u}_{n,k}\),
    it holds that
    \begin{equation*}
        |E_1(x)| = \left|\left(\frac{x}{W_{(k)}}\right)^{-\widehat
        \alpha+\alpha}
        -1\right|=\left|\exp{\left(\left(\alpha-\widehat{\alpha}\right)\log\left(\frac{x}{W_{(k)}}\right)\right)}-1\right|\leq
        \frac{7}{4}\Delta_{\alpha}\log\left(\frac{x}{W_{(k)}}\right)\leq
        \frac{7}{4} \Delta_{\alpha}^{1-\epsilon},
    \end{equation*}
    therefore, on the event
    \(\mathcal{G}_{n,k}\cap\mathcal{D}_{n,k}\), it holds that
    \begin{equation}\label{eq:control_E_1}
        \sup_{\overline{u}_{n,k}\leq x\leq
        \exp{\left(\Delta_\alpha^{-\epsilon}\right)}\underline{u}_{n,k}}{\left|
        E_1(x) \right| } \leq \frac{7}{4}
        \Delta_{\alpha}^{1-\epsilon}\leq 1.
    \end{equation}
    Equations \eqref{eq:r_k_ratio_control} and
    \eqref{eq:variance_quantile} alongside the hypotheses
    that
    \(\sup_{s\in[3/4,4]}{\left|\left(\frac{U(ns/k)}{U(n/k)}\right)^{\alpha}-s\right|}\leq
    \frac{1}{2}\)
    and \(\delta\geq
    2\exp{\left(-\sqrt{k}/12\right)}\) show
    that on \(\mathcal{E}_{n,}\),
    \begin{equation}\label{eq:control_E_2}
        \sup_{x\geq \overline{u}_{n,k}} |E_2(x)|\leq
        \sqrt{\frac{12\log(2/\delta)}{k}} +
        \frac{12\log(2/\delta)}{k} + \frac{4}{n} +
        \sup_{s\in[3/4,4]}{\left|\left(\frac{U(ns/k)}{U(n/k)}\right)^{\alpha}-s\right|}
        < 1.
    \end{equation}

    For any \(s\geq 1\), it holds that \(\frac{L(U(n/k)s)}{L(U(n/k))}
        = s^{\alpha}\frac{\bar{F} \left(
        U(n/k)s \right) }{\bar{F} \left(
    U(n/k)\right)} = s^{\alpha}\frac{n}{k} p_{U(n/k)s}\), therefore
    \begin{equation}\label{eq:B_L_alternative_form}
        s^{\alpha}\left|
        \frac{n}{k}p_{U(n/k)s} - s^{-\alpha} \right| =
        \left|\frac{L(U(n/k)s)}{L(U(n/k))}-1\right|,
    \end{equation}
    hence, the boundedness of \(L\) guarantees that \(B_L\) is
    well-defined. Moreover,
    \begin{equation}\label{eq:control_E_3}
        \sup_{x\geq \overline{u}_{n,k}} |E_3(x)|\leq \sup_{x\geq  U(n/k)}
        \left|\frac{L(x)}{L(U(n/k))} -1\right| = \sup_{s\geq 1}
        \left|\frac{L(U(n/k)s)}{L(U(n/k))}-1\right| = B_{L}.
    \end{equation}

    Then,
    from \eqref{eq:ratio_product_decomposition}, \eqref{eq:control_E_1},
    \eqref{eq:control_E_2} and \eqref{eq:control_E_3} the following
    inequality holds
    on
    \(\mathcal{E}_{n,k}\cap\mathcal{G}_{n,k}\cap\mathcal{D}_{n,k}=\mathcal{E}_{n,k}\cap\mathcal{D}_{n,k}\)
    for all \(x\) such that \(\overline{u}_{n,k}\leq x\leq
    \exp{\left(\Delta_\alpha^{-\epsilon}\right)}\underline{u}_{n,k}\)
    \begin{align*}
        \left|\frac{\widehat{p}_x}{p_x}-1\right| & \leq
        |E_1(x)|+|E_2(x)|+|E_3(x)| + \sum_{1\leq i<j\leq
        3}|E_i(x)E_j(x)| + |E_1(x)E_2(x)E_3(x)|, \nonumber
        \\
        & \leq |E_1(x)| + 2|E_2(x)|+ 4|E_3(x)|,\nonumber
        \\
        & \leq \frac{7}{4} \Delta_{\alpha}^{1-\epsilon} +
        4\sqrt{\frac{3\log(2/\delta)}{k}} +
        \frac{24\log(2/\delta)}{k} + \frac{8}{n} +
        4B_{SRV}(n,k).
    \end{align*}
    Equation \eqref{eq:relative_deviation} now follows from
    the union bound.

    \textbf{Bias terms:} The function \(U\) is regularly varying with
    index \(1/\alpha\), hence \(U^{\alpha}\)
    is regularly varying with index \(1\) \cite[Proposition
    1.5.7]{Bingham1987}. Therefore,
    \(U^{\alpha}(ts)/U^{\alpha}(s)\) converges uniformly to \(s\) as
    \(t\to +\infty\)
    on closed intervals \cite[Theorem 1.5.2]{Bingham1987}, which shows
    that
    \(B_U\)
    goes to \(0\) as \(n/k\to +\infty\). The same argument, but using that
    \(\bar{F}\) is regularly with index \(-\alpha\), shows that
    \(B_F\) goes to \(0\)
    as \(n/k\to+\infty\).

    For \(B_L\), let \(\epsilon<2\) be fixed. Because \(L(t)\to
    c>0\), we can find
    \(T_\epsilon\) such that for all \(t\geq T_{\epsilon}\) it holds that
    \(\left|L(t)-c\right|<\epsilon c/4\). Then, for all \(t\geq T_{\epsilon}\)
    and \(s\geq 1\) it holds that
    \begin{equation*}
        \left|\frac{L(ts)}{L(t)}-1\right|=\frac{\left|L(ts)-L(t)\right|}{L(t)}\leq
        \frac{2}{c}\left|L(ts)-L(t)\right|\leq
        \frac{2}{c}\left(\left|L(ts)-c\right|+\left|L(t)-c\right|\right)\leq
        \epsilon,
    \end{equation*}
    which shows that \(\sup_{s\geq 1}\left|L(ts)/L(t)-1\right|\to 0\)
    as \(t\to+\infty\) and concludes the proof.
\end{proof}

\begin{lemma}\label{lem:deviation_engular_depth} Let \(P\) be a
    probability measure on \(\R^d\), \(X\) a random variable with
    distribution \(P\) and \(X_1,\ldots,X_d\) an i.i.d. sample of
    \(X\). For any \(t>0\), define \(p_t=\P(r(X)\geq
    t)\) and take \(\hataD{}{}\) and \(\aD{}{}\) as in Algorithm
    \ref{alg:CO-depthEstimation}.
    Then, for any \(0<a<b\) and \(\delta\in(0,1)\) such that
    \((2\delta)^{d}\leq 1/8\), the following
    inequality holds with
    probability bigger than \(1-\delta\)
    \begin{equation}
        \sup_{\omega\in\sphere, s\in[a,b]} \left|
        \hataD{\omega}{P_{\theta|s}} - \aD{\omega}{P_{\theta|s}}
        \right|
        \le \frac{p_a}{p_b} \left(6\sqrt{\frac{d\log(2(np_a+1)/\delta)}{np_a}} +
        \frac{4}{3}\frac{\log(2/\delta)}{np_a}\right).
    \end{equation}
\end{lemma}
\begin{proof}
    For any \(\omega\in\sphere\) and \(s\in[a,b]\), it holds that
    \begin{align*}
        \left|
        \hataD{\omega}{P_{\theta|s}} - \aD{\omega}{P_{\theta|s}}
        \right| &= p_s^{-1}\left| p_s \hataD{\omega}{P_{\theta|s}} -
        \pD{s\omega}{P}\right| \\
        &\leq p_s^{-1}\left( \left| p_s\hataD{\omega}{P_{\theta|s}} -
            \hatpD{s\omega}{P} \right| + \left|
        \hatpD{s\omega}{P}-\pD{s\omega}{P} \right| \right)\\
        &\leq p_b^{-1}\left( \left| p_s\hataD{\omega}{P_{\theta|s}} -
            \hatpD{s\omega}{P} \right| + \left|
        \hatpD{s\omega}{P}-\pD{s\omega}{P} \right| \right)\\
        &\leq \frac{p_a}{p_b}\left( T_0(s,\omega)+T_1(s,\omega)\right)
    \end{align*}
    where \(T_0\) and \(T_1\) are defined as follows
    \begin{equation*}
        T_0(s,\omega) = p_a^{-1}\left|
        p_s\hataD{\omega}{P_{\theta|s}} - \hatpD{s\omega}{P}
        \right|\;,\quad T_1(s,\omega)=p_a^{-1}\left|
        \hatpD{s\omega}{P}-\pD{s\omega}{P} \right|.
    \end{equation*}
    Denote by \(\hat{p}_s\) the empirical
    estimator of \(p_s\) based on \(X_1,\ldots,X_n\), then
    \begin{align}
        T_0(s,\omega)&=p_a^{-1}\hataD{\omega}{P_{\omega|s}} \left| p_s-\hat{p}_s
        \right|\leq p_a^{-1} \left| p_s-\hat{p}_s \right|\leq
        p_a^{-1}\sup_{t\geq a}\left| p_s-\hat{p}_s \right|,\\
        T_1(s,\omega) &\leq p_a^{-1}\sup_{x:r(x)\geq a} \left| \hatpD{x,P} -
        \pD{x,P} \right|.
    \end{align}
    The result now follows from Theorem~\ref{thm:stat-coDepth},
    Corollary 4.4 in \cite{lhaut2022uniform}
    and the fact that the VC-dimension of the class of half-lines is 1.
\end{proof}

The following technical lemma allows rewriting relative VC inequalities in a way suitable for our purposes, namely for the proof of Proposition~\ref{prop:control_relative_marginal_dev_x_small}.  
\begin{lemma}[Technical inversion tool for relative deviations, see \cite{lugosi2002pattern} p. 338]\label{lem:rel_dev_abc}
  Let $A,B,C\ge 0$ with $B>0$ such that
  $$
\frac{A - C}{ \sqrt{A}} \le B. 
$$
Then
\begin{equation*}
\frac{A}{C} -1 \le B^2/C + B/\sqrt{C}. 
\end{equation*}
\end{lemma}
\begin{proof} 

  Let $x= \sqrt A$. The hypothesis is that $x^2 - C \le B x$, or 
  $$
p(x) := x^2 - Bx - C \le 0.  
$$
The determinant of the quadratic form is $\Delta = B^2 +4 C > 0$. Then $p(x)\le 0$ implies
$$
x \le x_+ := (B + \sqrt{\Delta})/2 \le B+ \sqrt{C}. 
$$

Injecting  this condition in the hypothesis we obtain
$$
A -C \le B^2 + B\sqrt{C}.
$$

\end{proof}

The next two geometric technical lemmas  (Lemmas~\ref{lem:bound_surface_spherical_cap_first_orthant}, \ref{lem:surface_sphere_tau_plusminus_eps}) control the surface area of slices and unions of slices of the first orthant of the unit sphere.  They are later used in Proposition~\ref{prop:bound_error_phitrad} to obtain multiplicative constants depending on $d$ for the deviations of the empirical angular measure and  in particular to obtain an upper bound on the quantity  $c(d)$ that appears in the main result of \cite{ClemenconJalalzaiSabourinSegers2023}. This upper bound   is not given in the cited reference.

    \begin{lemma}[Surface of spherical slices in the first orthant]\label{lem:bound_surface_spherical_cap_first_orthant} 
    For \(\beta\in(0,1)\) and
    \(u\in\sphere\), define the slice \(M_{\beta,u}=\{x\in\sphere:0\leq \langle u,x \rangle \leq \beta\}\). 
    Then,
    \begin{equation}\label{eq:bound_surface_spherical_cap_first_orthant}
        \sup_{u\in\sphere}{\sigma_{d-1}(M_{\beta,u}\cap \sphere_{+})} \leq \frac{16\sqrt{2}\beta d^{3/2}}{3\pi}\sigma_{d-1}(\sphere_+) = \frac{\sqrt{2}\beta d^{3/2}}{3\pi}\frac{ \sigma_{d-1}(\sphere)}{2^{d-4}}.
    \end{equation}
    \end{lemma}
\begin{proof}[Proof of Lemma
    \ref{lem:bound_surface_spherical_cap_first_orthant}]
    Let \(u\in\sphere\) be fixed and \(Y=(Y_1,\ldots,Y_d)\) be a
    random variable that follows a uniform
    distribution on \(\sphere\), then, there exists i.i.d. standard normal
    random variables \(V_1,\ldots,V_d\) such that
    \begin{equation}\label{eq:joint_distribution_uniform_sphere}
        \left(Y_1,\ldots,Y_d \right) =
        \left(\frac{V_1}{\sqrt{V_1^2+\dots V_d^2}},\ldots,
        \frac{V_d}{\sqrt{V_1^2+\dots V_d^2}} \right).
    \end{equation}

    Define \(W\) as the restriction of
    \(Y\) to the unit orthant \(\sphere_+\). Then
    \begin{equation}\label{eq:spherical_cap_w}
        \P\left( W\in M_{\beta,u} \right) =  \P\left( Y\in M_{\beta,u} | Y\in
        \sphere_+ \right) = \frac{\P\left(
        Y\in M_{\beta,u}\cap \sphere_+ \right) }{\P \left( Y\in\sphere_+
        \right)} = \frac{
            \sigma_{d-1} \left( M_{\beta,u}\cap\sphere_+ \right)
        }{\sigma_{d-1}(\sphere_+)}.
    \end{equation}

    Denote by \(f\) the density function of the standard normal
    distribution and by \(G:\R^d\to\sphere\) the transformation defined by
    \(G(v_1,\ldots,v_d)=\left(
    \frac{v_1}{||v||},\ldots,\frac{v_d}{||v||} \right)\). For any
    Borelian set \(A\)
    in \(\sphere_+\), we have that
    \begin{align*}
        \P\left( W\in A \right) &= \P\left( Y\in A|Y\in\sphere_+
        \right) = 2^{d}\P(Y\in A) =
        2^d\int_{\R^d}\I\{ G(v)\in A \}\prod_{i=1}^{d}f(v_i)dv_i,\\
        &=\int_{\R^d}\I\{ G(v)\in A \}\prod_{i=1}^{d}2f(v_i)dv_i =
        \int_{(0,+\infty)^d}\I\{ G(v)\in A \}\prod_{i=1}^{d}2f(v_i)dv_i,
    \end{align*}
    where for the last expression we have used that \(I\{G(v)\in A\}=0\)
    for all \(v\not\in (0,+\infty)^d\). Using that the density function
    of the Half-Normal distribution
    equals \(2f\I_{[0,+\infty)}\), we have that
    \begin{equation}\label{eq:joint_distribution_uniform_orthant}
        \left(W_1,\ldots,W_d \right) \overset{d}{=}
        \left(\frac{X_1}{\sqrt{X_1^2+\dots X_d^2}},\ldots,
        \frac{X_d}{\sqrt{X_1^2+\dots X_d^2}} \right),
    \end{equation}
    where the \(X_i\) are i.i.d. standard Half-Normal random
    variables, therefore,
    \begin{equation}\label{eq:prob_bound_interval}
        \P\left( W\in M_{\beta,u} \right) = \P\left(0\leq
            \frac{\innerprod{u}{X}}{||X||_2}
        \leq \beta \right)\leq \P\left( \left|
        \frac{\innerprod{u}{X}}{||X||_2} \right|\leq \beta \right)
    \end{equation}
    The distribution of \(W\) is invariant by permutations of its
    components, hence,
    the distribution of \(\innerprod{u}{W}\) is invariant by permutations of the
    components of \(u\),
    so, without loss of generality, we can assume that \(u\) is such
    that \(|u_1|\geq|u_2|\geq\dots\geq|u_n|\). Moreover, we  Let
    \(s=\mathrm{sign}{(u_1)}\) and define the
    random variables \(Z = \sum_{i=2}^{d}u_iX_i\) and \(Q=\sum_{i=2}^{d}X_i^2\),
    then
    \begin{equation}
        \innerprod{u}{X} =
        s|u_1|X_1+Z\quad\mathnormal{and}\quad||X||_2=\sqrt{X_1^2+Q},
    \end{equation}
    where \(X_1\) is independent of \(Z\) and \(Q\).

    Suppose
    \(\beta<|u_1|/2\) (for \(\beta\geq |u_1|/2\) the inequality
        \eqref{eq:bound_surface_spherical_cap_first_orthant} is
    trivial), then we have that:
    \begin{align}
        \left\{\left|
        \frac{\innerprod{u}{X}}{||X||_2} \right|\leq \beta\right\} &=
        \left\{ \Bigl|s|u_1|X_1+Z\Bigr|\leq \beta\sqrt{X_1^2+Q}
        \right\},\nonumber \\
        &\subseteq \left\{ X_1^2\left(|u_1|^2-\beta^2\right) +
        2s|u_1|ZX_1+Z^2-\beta^2Q\leq 0 \right\},\nonumber \\
        &\subseteq \left\{ X_1\in
        \left[\min(A_{-},A_{+}),\max(A_{-},A_{+})\right]\right\}\label{eq:events_inclussion},
    \end{align}
    where
    \begin{equation}\label{eq:interval_limits}
        A_{\pm}=\frac{-s|u_1|Z\pm\sqrt{|u_1|^2Z^2-\left(Z^2-\beta^2Q\right)
        \left(|u_1|^2-\beta^2\right) }}{|u_1|^2-\beta^2} =
        \frac{-s|u_1|Z\pm\beta\sqrt{ Z^2+Q(|u_1|^2-\beta^2) }}{|u_1|^2-\beta^2}
    \end{equation}
    From \eqref{eq:events_inclussion}, and using that \(X_1\) is
    independent of \(Z\) and \(Q\) we obtain that
    \begin{align}
        \P\left( \left|
            \frac{\innerprod{u}{W}}{||W||_2} \right|\leq \beta
        \Bigl.\Bigr| Z,Q \right)
        &\leq \P\Bigl( X_1\in
            \left[\min(A_{-},A_{+}),\max(A_{-},A_{+})\right] \Bigl.\Bigr|
        Z,Q \Bigr),\nonumber\\
        &\leq \P\Bigl( X_1\in
        \left[\min(A_{-},A_{+}),\max(A_{-},A_{+})\right] \Bigr)\leq
        2\sup_{t\geq 0}f(t)\left| A_+-A_{-} \right|,\nonumber\\
        &\leq\frac{2}{\pi}\left| A_+-A_{-}
        \right|.\label{eq:conditional_prob_bound}
    \end{align}
    From \eqref{eq:interval_limits} we have that \(\left| A_+-A_{-}\right| =
    2\beta\sqrt{\frac{Z^2+Q(|u_1|^2-\beta^2)}{(|u_1|^2-\beta^2)^2}}\). By
    Cauchy-Schwartz inequality, \(|Z|\leq\sqrt{Q}\); combining this
    with \(\beta<|u_1|/2\) and \(|u_1|\geq 1/\sqrt{d}\), we obtain that
    \begin{equation}\label{eq:interval_size_bound}
        \left| A_+-A_{-}\right|\leq
        \frac{2\beta\sqrt{2Q}}{|u_1|^2-\beta^2}\leq
        \frac{8\beta\sqrt{2Q}}{3|u_1|^2}\leq
        \frac{8\beta d\sqrt{2Q}}{3}.
    \end{equation}
  
    Taking expectations in \eqref{eq:conditional_prob_bound} and using
    \eqref{eq:interval_size_bound} yields
    \begin{equation*}
        \P\left( \left|\frac{\innerprod{u}{X}}{||X||_2} \right|\leq
        \beta \right) \leq
        \frac{16\sqrt{2}\beta d}{3\pi} \E[\sqrt{Q}]\leq
        \frac{16\sqrt{2}\beta d^{3/2}}{3\pi},
    \end{equation*}
    where the last inequality follows from Jensen's inequality and the that fact
    that \(Q\) follows a Chi-squared distribution with \(d-1\) degrees
    of freedom. The result now follows from
    \eqref{eq:spherical_cap_w} and \eqref{eq:prob_bound_interval}.
\end{proof}

    \begin{lemma}[Surface area of $\sphere^{\tau-\varepsilon}\setminus\sphere^{\tau+\varepsilon }$]\label{lem:surface_sphere_tau_plusminus_eps}
      Let $\tau<1/2$ and $\varepsilon \le 1$. Then
$$
      \sigma_{d-1}(\sphere^{\tau-\varepsilon }\setminus\sphere^{\tau+\varepsilon }) \le
      \frac{d^{3/2}}{\sqrt{2\pi}}
      \max\Big(\frac{2 }{\sqrt{1 - 4\tau^2}}, 
       \pi
       \Big)
      \sigma_{d-1}(\sphere_+) \varepsilon.
      $$
    
    \end{lemma}
    \begin{proof}
      Let $\tau<1/2,\epsilon\in(0,1), j\in\{1,\ldots, d\}$ be fixed, and let 
      $M_j = \{x \in \sphere_+: x_j\in
      [\tau-\varepsilon ,\tau+\varepsilon ]\}$. Then
      $$
      M_j = \{x \in \sphere_+: x_j\in
      [\max(\tau-\varepsilon,0) ,\min(\tau+\varepsilon,1) ]\}, 
      $$
      and 
      $\sphere^{\tau-\varepsilon }\setminus\sphere^{\tau+\varepsilon }\subset
      \cup_{j\le d} M_j$,  so that
      \begin{equation}
          \label{eq:bound_union_slices1}
          \sigma_{d-1}(\sphere^{\tau-\varepsilon }\setminus\sphere^{\tau+\varepsilon })
      \le d\sigma_{d-1}(M_j).
      \end{equation} 
      Now for fixed $j$, 
      let $\Delta_j$ denote the subset of $\sphere$ comprised between the parallel circles
      $\mathcal{C}_{\max(\tau-\varepsilon,0) }$ and $\mathcal{C}_{\min(\tau+\varepsilon, 1) }$, where $\mathcal{C}_b = \{x\in\sphere: x_j = b\}$. Then $M_j = \Delta_j\cap\sphere_+$, and by symmetry with respect to the $d-1$ coordinates different from $j$, we have 
      \begin{equation}
          \label{eq:bound_union_slices_2}
          \sigma_{d-1}(M_j) = 2^{-d+1}\sigma_{d-1}(\Delta_j).
      \end{equation} 
      We are left with computing the surface measure $\sigma_{d-1}(\Delta_j)$. 
      For $0<b<1$, the circle $\mathcal{C}_b$ is isometric with the sphere in the ambient space $\rset^{d-1}$ with radius $r_b = \sqrt{1-b^2}\le 1$. 
Recall from~\eqref{eq:surface_measure_sphere} that       that the surface of the $d-2$-dimensional unit sphere is 
       $s_{d-2} = \frac{2 \pi^{(d-1)/2}}{\Gamma((d-1)/2)},  $
so that 
      \begin{align*}
        \frac{s_{d-2}}{s_{d-1}}  = \frac{\pi^{-1/2} \Gamma(d/2)}{\Gamma((d-1)/2)}
        \le \sqrt{\frac{d}{2\pi}}, 
      \end{align*}
      where we have used Gautschi's inequality to bound from above the ratio of Gamma's. 
      Whence, for $0\le b\le 1$, we have
$$
\sigma_{d-2}(\mathcal{C}_b) = r_b^{d-2} s_{d-2} = 
\le \sqrt{\frac{d}{2\pi}} \sigma_{d-1}(\sphere).
$$
     
      Finally, for $\varepsilon<\tau <1/2$, the geodesic distance between the two circles $\mathcal{C}_{\tau- \varepsilon }$ and $\mathcal{C}_{\tau+\varepsilon }$ is
      $$
      \rho(\mathcal{C}_{\tau-\varepsilon }, \mathcal{C}_{\tau+\varepsilon }) =
      \arcsin(\tau+\varepsilon ) - \arcsin(\tau-\varepsilon ) \le
      \frac{2 \varepsilon }{\sqrt{1 - (\tau+\varepsilon )^2}} \le 
      \frac{2 \varepsilon }{\sqrt{1 - 4\tau^2}}.
      $$
      Also for $\varepsilon\ge \tau$, 
$$
      \rho(\mathcal{C}_{0}, \mathcal{C}_{\min(\tau+\varepsilon,1) }) \le 
      \rho(\mathcal{C}_{0}, \mathcal{C}_{\min(2\varepsilon,1) }) \le 
      \arcsin(\min(2\varepsilon,1)) \le \pi\varepsilon,
      $$
      where the last inequality is obtained by studying the variations of the functions $x\mapsto \arcsin(\min(2x,1))/x$ on $(0,1)$. 
      Combining the latter  three displays, we obtain 
      \begin{equation}\label{eq:bound_union_slice_3}
      \sigma_{d-1}(\Delta_j) \le \sup_{0\le b\le 1}\sigma_{d-1}(\mathcal{C}_b)\rho(\mathcal{C}_{\max(\tau-\varepsilon,0) }, \mathcal{C}_{\min(\tau+\varepsilon,1) }) 
      \le \sqrt{\frac{d}{2\pi}} \sigma_{d-1}(\sphere)
      \max\Big(\frac{2 }{\sqrt{1 - 4\tau^2}}, 
       {\pi}\Big)\varepsilon.
      \end{equation}
      Finally,  combining Equations~\eqref{eq:bound_union_slices1}, \eqref{eq:bound_union_slices_2}, \eqref{eq:bound_union_slice_3} and using $\sigma(\sphere_+) = 2^{-d}\sigma_{d-1}(\sphere)$, we obtain
      \begin{align*}
      \sigma_{d-1}(\sphere^{\tau-\varepsilon}\setminus\sphere^{\tau+\varepsilon}) 
      & \le 
      d \sigma_{d-1}(M_j) \le d \frac{\sigma_{d-1}(\Delta_j)}{2^{d-1}}\\
      &\le \frac{d^{3/2}}{\sqrt{2\pi}}
      \max\Big(\frac{2 }{\sqrt{1 - 4\tau^2}}, 
       \pi 
       \Big)
      \frac{\sigma_{d-1}(\sphere)}{2^{d-1}} \varepsilon\\
      &=\frac{d^{3/2}}{\sqrt{2\pi}}
      \max\Big(\frac{2 }{\sqrt{1 - 4\tau^2}}, 
       \pi
       \Big)
      \sigma_{d-1}(\sphere_+) \varepsilon.\\
      \end{align*}

    \end{proof}

The following result  (and its proof) is a mere variation of Lemma B.1 in \cite{ClemenconJalalzaiSabourinSegers2023}.
 \begin{lemma}[Controlling the error of $\tilde v$ from that of the $\tilde{ p}_j$'s]\label{lem:control_err_v_from_err_pj}
  Let  $x\in\rset^d$ be such that with  
   $$
h = \max_j |\tilde{ p}_j(x_j)/ p_j(x_j) - 1| , 
$$
we have $|h|<1$. Then
$$
\| \tilde v(x) - v(x) \|\le \min\Big(~ \frac{h}{1-h} \| v(x) \|~,~ h\| \tilde v(x)\|~\Big) \le \frac{h}{1-h} \min\big( \| v(x) \|, \| \tilde v(x) \|\big) . 
$$
  \end{lemma}
  \begin{proof}[Proof of Lemma~\ref{lem:control_err_v_from_err_pj}]
    For $j\le d$ we have by hypothesis,
    \begin{equation}\label{eq:bound_ratio_vj}
        v_j(x_j)/\tilde{v}_j(x_j)  \in[1-h, 1+h],
    \end{equation}
    so $
    |{v}_j(x_j) - \tilde{v}_j(x_j) | \le h \tilde{v}_j(x_j)
    $, or equivalently
    $$
    |{v}_j(x_j) - \tilde{v}_j(x_j) |^2 \le h^2  \tilde{v}_j(x_j)^2.
    $$
    Summing over $j\in\{1,\ldots, d\}$ and taking the squared root yields
    $\| \tilde v(x) - v(x) \|\le   h \| \tilde v(x) \| $.

    Now  by inversion of (\ref{eq:bound_ratio_vj}), also
    $ \tilde{v}_j(x_j)/{v}_j(x_j)  \in[(1+h)^{-1}, (1-h)^{-1}]$, thus
    $$
    |\tilde{v}_j(x_j)/{v}_j(x_j) -1 | \le h/(1-h).
    $$ Repeating the previous argument with $\tilde v $ and $v$
    interchanged and $h$ replaced with $h/(1-h)$ yields 
    $\| \tilde v(x) - v(x) \|\le  \| v(x) \| h/(1-h)$.
    The result follows.
\end{proof}

We also quote Lemma B.3 from
\cite{ClemenconJalalzaiSabourinSegers2023}, stating that, for
$v_1,v_2\in\rset^d\setminus\{0\}$,
\begin{equation}
    \label{eq:control_theta_from_vect}
    \|\theta(v_1) - \theta(v_2) \|\le \frac{ \|v_1-v_2\|}{\|v_1\|\vee \|v_2\|}.
\end{equation}

\end{document}